\documentclass[11pt,twoside]{amsart}
\usepackage{amsmath,amssymb,amsthm}
\catcode`\@=11 \@addtoreset{equation}{section}

\catcode`\@=12

\newtheorem{Theorem}{Theorem}[section]
\newtheorem{Proposition}{Proposition}[section]
\newtheorem{Lemma}{Lemma}[section]
\newtheorem{Corollary}{Corollary}[section]
\newtheorem{Definition}{Definition}[section]
\newtheorem{Remark}{Remark}[section]

\newcommand{\bTheorem}[1]{
\begin{Theorem} \label{T#1} }
\newcommand{\eT}{\end{Theorem}}

\newcommand{\bProposition}[1]{
\begin{Proposition} \label{P#1}}
\newcommand{\eP}{\end{Proposition}}

\newcommand{\bLemma}[1]{
\begin{Lemma} \label{L#1} }
\newcommand{\eL}{\end{Lemma}}

\newcommand{\bCorollary}[1]{
\begin{Corollary} \label{C#1} }
\newcommand{\eC}{\end{Corollary}}

\newcommand{\bFormula}[1]{
\begin{equation} \label{#1}}
\newcommand{\eF}{\end{equation}}

\newcommand{\bDefinition}[1]{
 \begin{Definition} \label{D#1}}
\newcommand{\eD}{\end{Definition}}

\newcommand{\bProof}{{\bf Proof: }}

\def\cA{{\mathcal A}}

\def\cC{{\mathcal C}}
\def\cF{{\mathcal F}}
\def\cH{{\mathcal H}}
\def\cI{{\mathcal I}}
\def\cJ{{\mathcal J}}
\def\cL{{\mathcal L}}
\def\cM{{\mathcal M}}
\def\cO{{\mathcal O}}
\def\cP{{\mathcal P}}
\def\cQ{{\mathcal Q}}
\def\cS{{\mathcal S}}
\def\cU{{\mathcal U}}
\def\cX{{\mathcal X}}
\def\cY{{\mathcal Y}}
\def\eps{\varepsilon}

\let\wt=\widetilde
\let\wh=\widehat

\def\fb{\mathfrak{b}}
\def\fd{\mathfrak{d}}
\def\fj{\mathfrak{j}}

\newcommand{\Int}{\displaystyle \int}
\newcommand{\Frac}{\displaystyle \frac}

\renewcommand{\Re}{{\rm Re}\,}

\def\d{\partial}
\def\ddj{\dot \Delta_j}
\def\ddk{\dot \Delta_k}

\newcommand{\du}{\delta\!\vu} 

\newcommand{\db}{\delta\! b}
\renewcommand{\dj}{\delta\! j}

\newcommand{\vr}{\varrho}

\newcommand{\vu}{\vc{u}}
\newcommand{\vc}[1]{{\vec #1}}
\newcommand{\Div}{{\rm div}_x}
\renewcommand{\div}{{\rm div}\,}
\newcommand{\Grad}{\nabla_x}

\newcommand{\vv}{\vc{v}}

\newcommand{\ep}{\varepsilon}

\newcommand{\R}{\mathbb R}

\newcommand{\Z}{\mathbb Z}
\newcommand{\N}{\mathbb N}

\newcommand{\T}{\mathbb T}
\date\today
\makeindex
\begin{document}

%%%%%%%%%%%%%%%%%%%%%%%%%%%%%%%%

%%%%%%%%%%%%%%%%%%%%%%%%%%%%%%%%

\title{Diffusive limits  for a barotropic model of radiative flow}
\author{Rapha\"el Danchin$^*$, Bernard Ducomet$^{**}$}
\maketitle
\centerline{$^*$Universit\'{e} Paris-Est,  LAMA (UMR 8050), UPEMLV, UPEC, CNRS, Institut Universitaire de France,}
\centerline{61 avenue du G\'en\'eral de Gaulle, 94010 Cr\'eteil Cedex 10}
\medskip
\centerline{$^{**}$CEA, DAM, DIF, D\'epartement de Physique Th\'eorique et Appliqu\'ee}
\centerline{F-91297 Arpajon, France}

\begin{abstract}
 Here we aim at justifying rigorously  different types of physically relevant  diffusive limits  for radiative flows.
For simplicity, we consider the barotropic situation,   and adopt  the so-called $P1$-approximation of the radiative
transfer equation. In the critical functional framework, we establish the existence of 
global-in-time strong solutions corresponding to small enough data, and exhibit  uniform estimates 
with respect to the coefficients of the system. Combining with standard compactness arguments, this enables us to
 justify rigorously the convergence of the solutions to the expected  limit systems. 

Our results hold true 
in the whole space $\R^n$ as well as in a periodic box $\T^n$ with $n\geq2.$ 
\end{abstract}

%\maketitle

\section{Introduction}
\label{i}

We consider the barotropic version of a model of radiation hydrodynamics. 
Our main goal is to provide the rigorous justification of asymptotics 
that have been investigated formally and numerically by  Lowrie, Morel and Hittinger \cite{LMH}, 
and mathematically by the second author and \v S. Ne\v casov\'a in \cite{DN3,DN3b,DN4} in the finite energy weak solutions framework.

The fluid is described by standard classical fluid mechanics for the mass density $\vr$ and the velocity
field $\vu$ as functions of the time $t\in \R_+$ and of the (Eulerian) spatial
coordinate $x\in\Omega$   where $\Omega$  is either the whole space  $\R^n$ 
or  some periodic box $\T^n$  with $n\geq2.$  

 Radiation acts through  some  \emph{radiative momentum source} $\vc S_F$ which is given by 
$$\vec S_F= \frac{1}{c}\ \int_0^{\infty}\int_{{\cS}^{n-1}}\vc \omega S\ d\vc \omega\, d\nu,$$
where $c$ is the light speed. 
\medbreak
The radiative source $S=S(t,x,\vc\omega,\nu)$ depends on the direction vector $\vc\omega\in
\cS^{n-1}$ (where $\cS^{n-1}$ denotes the unit sphere of $\R^n$), and on the frequency 
$\nu\geq0$ of the photons, and   is given by 
$$
S=\sigma_a\bigl(B(\nu,\vr)-{\mathcal I}\bigr)+\sigma_s\bigl(\tilde {\mathcal I}-{\mathcal I}\bigr)\quad\hbox{where}\quad
\tilde {\mathcal I}:=\frac{1}{|{\cS}^{n-1}|}\int_{{\cS}^{n-1}}{\mathcal I}\ d\vc \omega.$$
The radiative intensity $\cI$ obeys the transfer equation
\bFormula{i4}
\frac{1}{c}\ \partial_t {\mathcal I} + \vc \omega \cdot \Grad {\mathcal I} = S\ \ \ \ \ \ \mbox{in} \ (0,T) \times \Omega\times {\cS}^{n-1} \times (0, \infty).\eF
In the present paper, as in \cite{DD,DD3}, we make the following simplifying assumptions
\begin{enumerate}
\item Isotropy~: the transport coefficients  $\sigma_a$ and $\sigma_s$ 
are independent of  $\vc \omega$;
 \item `Gray' hypothesis~: $\sigma_a$ and  $\sigma_s$ are independent of $\nu$;
 \item  `$P1$ hypothesis'~:  the averaged radiative intensity   $I:=\int_0^{\infty}{\mathcal I}\ d\nu$
 is given by the ansatz
 \bFormula{Idec}
I=I_0+\vc \omega\cdot \vec I_1,
\eF
where  $I_0$ and $\vec I_1$  are independent of  $\vec \omega$ and $\nu.$ 
\end{enumerate}
\smallbreak
Plugging   \eqref{Idec} in   \eqref{i4}, and computing the $0$th and $1$st order momentum 
with respect to    $\vec \omega,$ we find out the following evolution equations for
$I_0$ and $I_1$ (keeping the same notation $B$ for the distribution function averaged
in $\nu$)
\bFormula{J4bis}
\frac{1}{c}\ \partial_t I_0 + \frac{1}{n}\ \Div \vec I_1 = \sigma_a(\vr)(B(\vr)-I_0),
\eF
\bFormula{J4ter}
\frac{1}{c}\ \partial_t \vec I_1 +  \Grad I_0 = (\sigma_a(\vr)+\sigma_s(\vr))\vc I_1.
\eF
Besides, the radiative force is now given by
\begin{equation}\label{eq:SF}
\vc S_F=\biggl(\frac{\sigma_a(\vr)+\sigma_s(\vr)}n\biggr)\vc I_1.
\end{equation}

In order to identify the most relevant asymptotic regimes, we rewrite
the equations in dimensionless form. 
To this end, introduce some 
reference hydrodynamical quantities (length, time, velocity, density, pressure): 
$\bar L, \ \bar T, \ \bar U, \bar\vr, \ \bar p,$
  and  reference    radiative quantities (radiative intensity,  absorption and scattering coefficients and equilibrium function): $\bar I, \ \bar\sigma_{a}, \ \bar\sigma_{s}$ and $\bar B.$
 \medbreak 
Let $Sr:={\bar L}/{\bar T\bar U},$ $Ma:={\bar U}/{\sqrt{\bar\vr\bar p}}$ and $Re:={\bar U\bar\vr \bar L}/{\bar\mu}$
 be the Strouhal, Mach and  Reynolds numbers corresponding to hydrodynamics. Let also define 
 ${\mathcal C}:={c}/{\bar U},
\ {\mathcal L}:=\bar L\bar\sigma_{a},
\ {\mathcal L}_s:={\bar\sigma_{s}}/{\bar\sigma_{a}},$
 various dimensionless numbers corresponding to radiation.
 In all that follows, we assume  our flow to be  strongly under-relativistic so that \emph{$\cC$ is large.}

Choosing $\bar B=\bar I,$  we discover  that the evolution of the dimensionless
unknowns (still denoted in the same way) is  governed by the following 
system of equations
$$\left\{
\begin{array}{l}
\!\! Sr\, \partial_t \vr + \div (\vr \vu) = 0,\\[1.5ex]
 \!\! Sr\,\partial_t (\vr \vu) + \div (\vr \vu \otimes \vu) 
+ \frac{1}{Ma^2} \nabla p-  \frac{1}{Re}\!\left(\div(\mu\nabla\vu\!+\!{}^t\nabla\vu)\!+\!\nabla(\lambda\div \vu)\right)
=\cL\bigl(\frac{\sigma_a\!+\!{\mathcal L}_s\sigma_s}{n}\bigr)\vc I_1,\\[1.5ex]
\!\!\frac{Sr}{{\cC}}\, \partial_t I_0 + \frac{1}{n}\ \mbox{div} \vec I_1
={\mathcal L} \sigma_a\left(B- I_0\right),\\[1.5ex]
\!\!\frac{Sr}{{\cC}}\, \partial_t \vec I_1 +  \nabla I_0
=-{\mathcal L}\left( \sigma_a+{\mathcal L}_s \sigma_s\right)\vec I_1,
\end{array}\right.
$$
where $\vr=\vr(t,x)\in\R_+$ and $\vu=\vu(t,x)$ stand for the density and pressure, respectively, 
$p=P(\rho)$ is the pressure, $\lambda=\lambda(\rho)$ and $\mu=\mu(\rho)$ are 
the viscosity coefficients. The given functions $P,$ $\lambda$ and $\mu$ are supposed sufficiently
smooth, and we make the following strict ellipticity assumption
$$
\nu:=\lambda+2\mu>0\quad\hbox{and}\quad \mu>0.
$$
In our recent work \cite{DD3}, we gave a  mathematical justification 
of the low Mach number asymptotics. In the present paper, we investigate another type of physically relevant asymptotic
regimes, which are of \emph{diffusive} type. They  correspond to the case where $\cC$ is large
 and  all the other dimensionless numbers, but $\cL$ and $\cL_s,$ are of order $1.$
To make it more concrete, take
$$Ma= Sr=Re=1,\quad {\mathcal C}=\varepsilon^{-1},\quad
\bar\vr=P'(\bar\vr)=B'(\bar\vr)=\sigma_a(\bar\vr)=\sigma_s(\bar\vr)=1,
$$
 where $\varepsilon$ is a small positive number, bound to tend to $0.$ 
 \medbreak
 Because we shall  focus on small perturbations of
 the reference density $\bar\vr=1,$ 
 it is convenient to introduce the new unknown  $b:=B(\vr)-B(1).$
 In this context,  all the functions of $\vr$ may be written in terms of $b.$
 Setting $j_0:=I_0-B(1)$ and $\vc j_1:=\vc I_1,$ and using exponents to emphasize the dependency with respect to $\ep,$ we eventually 
 get  the following  system
\begin{equation}\label{eq:NSrad}\left\{
\begin{array}{l}
\partial_t b^\ep + \vu^\ep \cdot\nabla b^\ep +(1+k_1(b^\ep))\div \vu^\ep=0,\\[1.5ex]
 \partial_t  \vu^\ep + \vu^\ep\cdot\nabla\vu^\ep- (1+k_2(b^\ep)){\mathcal A}\vu^\ep
 +(1+k_3(b^\ep))\nabla b^\ep = \frac{{\mathcal L}(1+\cL_s)}n
 (1+k_4(b^\ep))\vc j_1^\ep,\\[1.5ex]
 \ep\partial_t j_0^\ep + \frac{1}{n}\: \div\vc j_1^\ep = {\mathcal L}(b^\ep-j_0^\ep),\\[1.5ex]
\ep\partial_t \vc j_1^\ep+  \nabla j_0^\ep = -{\mathcal L}(1+\cL_s)\vc j_1^\ep,\\[1.5ex]
 \end{array}\right.
\end{equation}  with  ${\mathcal A}:= \mu\Delta+(\lambda+\mu)\nabla\div$
  and where  $k_1,$ $k_2,$ $k_3,$ $k_4$  are smooth functions vanishing at~$0.$

 %%%%%%%%%%%%%%%%%%%%%%%%%%%%%%%%%%%%%%%%%%
 
 \section{Formal asymptotics}\label{s:formal}

Let us first  present  some formal computations so as to  exhibit
 the limit equations we can  get from \eqref{eq:NSrad}
in different types of diffusive asymptotic regimes.  We restrict to the case where 
   the following necessary and sufficient linear stability condition (derived in \cite{DD}) is fulfilled
\begin{equation}\label{eq:stabcond}
n\nu\cL>\eps\biggl(\frac{2+\cL_s}{1+\cL_s}\biggr)\cdotp
\end{equation}
Note that \eqref{eq:stabcond}  implies that $\liminf \cL\eps^{-1}>0$ for $\eps$ going to $0.$ 
\medbreak
In all that follows,  it is assumed that 
$(b^\eps,\vu^\eps,j_0^\eps,\vc j_1^\eps)$  converges to $(b,\vu,j_0,\vc j_1)$ in some suitable space with enough regularity to pass to the limit in the nonlinear terms.

\subsection{Case $\cL\approx\ep$ and $\cL_s\to+\infty$}
%For simplicity, we take $\cL=\frac{\kappa\ep}{n\nu}$ with large enough $\kappa$ (in accordance with \eqref{eq:stabcond}). 
%Then \eqref{eq:NSrad} rewrites 
%$$\left\{\begin{array}{l}
%\partial_t b^\ep + \vu^\ep \cdot\nabla b^\ep +(1+k_1(b^\ep))\div \vu^\ep=0,\\[1.5ex]
 %\partial_t  \vu^\ep + \vu^\ep\cdot\nabla\vu^\ep- (1+k_2(b^\ep)){\mathcal A}\vu^\ep
 %+(1+k_3(b^\ep))\nabla b^\ep = \frac{\kappa\ep}{n^2\nu}(1+\cL_s)(1+k_4(b^\ep))\vc j_1^\ep,\\[1.5ex]
 %\partial_t j_0^\ep + \frac{1}{\ep n}\: \div\vc j_1^\ep +\frac{\kappa}{n\nu}(j_0^\ep-b^\ep)=0,\\[1.5ex]
%\partial_t \vc j_1^\ep+  \frac1\ep\nabla j_0^\ep +\frac{\kappa}{n\nu}(1+\cL_s)\vc j_1^\ep=\vc 0,\\[1.5ex]
 %\end{array}\right.$$
Denoting by $\cP$ the $L^2$ orthogonal projector on divergence free vector fields, we get 
\begin{equation}\label{eq:Pj1}
\cP\vc j_1^\eps(t)=e^{-\frac\cL\ep(1+\cL_s)t}\,\cP\vc j_1^\eps(0).
\end{equation}
Hence $\cP\vc j_1^\eps$ tends to $\vc 0$ for $\ep\to0.$

\subsubsection{Subcase $\cL^2\cL_s\to0$}

Setting $\cQ:={\rm Id}-\cP,$ we see that the equation for $j_0^\eps$ entails that $\cQ\vc j_1^\eps=\cO(\eps).$
Next, the equation for $\cQ\vc j_1^\eps$ implies that $\nabla j_0^\eps$ goes to 
$\vc 0,$ too, because $\eps^2\cL_s\to0$.      Assuming that $j_0$ decays to $0$ at infinity, this yields  $j_0=0.$
\smallbreak
\medbreak
{}From the equation for $\vc j_1^\eps,$ we also get
\begin{equation}\label{eq:rad}
-\cL(1+\cL_s)\vc j_1^\eps=\nabla j_0^\eps +\cO(\eps).
\end{equation}
Hence $\ep(1+\cL_s)\vc j_1^\ep$ goes to $\vc 0$ and  $(b,\vc u)$ thus satisfies the barotropic Navier-Stokes equations.
In other words, the radiative effect becomes negligible
in the asymptotic $\cL\approx\eps$ and $\eps^2\cL_s\to0$ with $\cL_s\to+\infty$

\subsubsection{Subcase $\lim_{\ep\to0}\cL^2\cL_s\in(0,+\infty)$}
This is the so-called \emph{nonequilibrium diffusion regime}. 
The analysis of the previous paragraph shows that $\vc j_1^\eps=\cO(\eps)$ (hence  $\vc j_1=\vec 0$)
and that \eqref{eq:rad} holds true. 
The new fact is that the equation for $j_0^\eps$ combined with \eqref{eq:rad}  implies that 
\begin{equation}\label{eq:radb}
\d_t j_0^\eps+\frac\cL\ep\bigl(j_0^\eps-b^\eps\bigr)
-\frac1n\frac\cL\ep\frac1{\cL^2\cL_s}\Delta j_0^\eps=\cO(\eps).
\end{equation}
 Now, if we assume  that
$$
\frac\cL\ep\to\frac\kappa{n\nu}\quad\hbox{and}\quad
\cL^2\cL_s\to\frac m{\nu^2},
$$
for some $m\in(0,+\infty)$ and $\kappa>1$ (see \eqref{eq:stabcond}),  then
  $(b,\vc u)$  satisfies the following compressible Navier-Stokes equations coupled with a parabolic equation
\begin{equation}\label{eq:noneq1}\left\{
\begin{array}{l}
 \partial_t b+ \vu\cdot\nabla b+ (1+k_1(b))\div \vu=0,\\[1ex]
\partial_t \vu + \vu\cdot\nabla\vu -(1+k_2(b))\cA\vu + (1+k_3(b))\nabla b
    +\frac1n(1+k_4(b))\nabla j_0=\vc 0,\\[1ex]
\partial_t j_0+ \frac\kappa{n\nu}\bigl(j_0-b-\frac{\nu^2}{nm}\Delta j_0\bigr)=0.
\end{array}\right.
\end{equation}

\subsubsection{Subcase $\cL^2\cL_s\to+\infty$}

We still have $\vc j_1^\eps=\cO(\eps)$, \eqref{eq:rad} and thus \eqref{eq:radb} holds true.
Now, as $\cL^2\cL_s\to+\infty$ and $\cL\approx\ep,$ the r.h.s. of \eqref{eq:radb} tends to $0.$
Therefore, if we assume as before that   $\cL/\ep\to\kappa/(n\nu)$ then 
we find out that $(b,\vu, j_0)$ satisfies the following \emph{degenerate nonequilibrium diffusion system}
\begin{equation}\label{eq:noneq2}\left\{
\begin{array}{l}
 \partial_t b+ \vu\cdot\nabla b+ (1+k_1(b))\div \vu=0,\\[1ex]
\partial_t \vu + \vu\cdot\nabla\vu -(1+k_2(b))\cA\vu + (1+k_3(b))\nabla b
    +\frac1n(1+k_4(b))\nabla j_0=\vc 0,\\[1ex]
\partial_t j_0+\frac\kappa{n\nu}( j_0-b)=0.
\end{array}\right.
\end{equation}

%%%%%%%%%%%%%%%%%%%%%%%
 
 \subsection{Case $\ep\ll\cL\ll1$}

Recall that we have  \eqref{eq:rad} while the equation for $j_0^\ep$  implies that
\begin{equation}\label{eq:radter}
\div\vc j_1^\ep=n\cL(b^\ep-j_0^\ep)+\cO(\ep).
\end{equation}
Hence $\cQ\vc j_1=0$ (as $\cL\to0$),  and 
\begin{equation}\label{eq:radquater}
\Delta j_0^\ep+n\cL^2(1+\cL_s)(b^\ep-j_0^\ep)=\cO(\ep)+\cO(\ep\cL(1+\cL_s)).
\end{equation}

\subsubsection*{Subcase $\cL^{2}\cL_s\to0$}

Then \eqref{eq:radquater} implies that $\Delta j_0=0$ and thus 
$j_0\equiv0$ (if one assumes that $j_0\to0$ at $\infty$). 
Consequently, \eqref{eq:rad} implies that the radiative force in the velocity 
equation tends to $0$ when $\ep$ goes to $0.$
Therefore $(b,\vc u)$ just satisfies the classical compressible Navier-Stokes equation.

\subsubsection*{Subcase $\nu^2\cL^2\cL_s\to m\in(0,+\infty)$}

We  have $\vc j_1=\vc 0,$ and Relations \eqref{eq:rad}, \eqref{eq:radquater} imply 
that $(b,\vc u,j_0)$ fulfills the following Navier-Stokes-Poisson system
\begin{equation}\label{eq:eq1}\left\{
\begin{array}{l}
 \partial_t b+ \vu\cdot\nabla b+ (1+k_1(b))\div \vu=0,\\[1ex]
\partial_t \vu + \vu\cdot\nabla\vu -(1+k_2(b))\cA\vu + (1+ k_3(b))\nabla b
    +\frac1n(1+k_4(b))\nabla j_0=\vec 0,\\[1ex]
-\nu^2\Delta j_0+mn(j_0-b)=0.
\end{array}\right.
\end{equation}

\subsubsection*{Subcase $\cL^{2}\cL_s\to+\infty$}

Then \eqref{eq:radquater} implies that $j_0=b.$ Combining with \eqref{eq:rad}, we thus find out that  $(b,\vc u)$ fulfills
the following compressible Navier-Stokes equation with \emph{modified pressure law}
\begin{equation}\label{eq:eq2}\left\{
\begin{array}{l}
 \partial_t b+ \vu\cdot\nabla b+ (1+k_1(b))\div \vu=0,\\[1ex]
\partial_t \vu + \vu\cdot\nabla\vu -(1+k_2(b))\cA\vu + \bigl(1+\frac 1n +k_3(b)+\frac1n{k_4(b)}\bigr)\nabla b=\vec 0.\end{array}    \right.
\end{equation}

 %%%%%%%%%%%%%%%%%%%%%%%
 
 \subsection{Case $\nu\cL\to \ell\in(0,+\infty)$}

\subsubsection*{Subcase $\nu^2\cL^{2}\cL_s\to m\in[0,+\infty)$}
 
 Passing to the limit in \eqref{eq:radquater} gives
 \begin{equation}\label{eq:eq1bis}
 -\nu^2\Delta j_0+n(\ell^2+m)(j_0-b)=0.
 \end{equation}
 So we get  System \eqref{eq:eq1} for $(b,j_0,\vu)$ with the last equation replaced by 
\eqref{eq:eq1bis}.
 
  \subsubsection*{Subcase $\cL_s\to+\infty$}
 
 Exactly as in the case $\cL\to0,$ we get $j_0=b,$ $\vc j_1=\vc 0,$ and 
 $(b,\vu)$ satisfies \eqref{eq:eq2}.

%%%%%%%%%%%%%%%%%%%%%%%%%%%%%%%%%%%%%%%

\subsection{Case $\cL\to+\infty$}

Relation \eqref{eq:rad} implies that $\vc j_1=0,$ and thus, according to 
\eqref{eq:radter}, we have  $j_0=b.$ 
Therefore \eqref{eq:rad} implies that
$$
\cL(1+\cL_s)\vc j_1^\ep\to \nabla b,
$$
and  $(b,\vu)$ thus satisfies \eqref{eq:eq2}. 

 \bigbreak
 
 To make a long story short, the above formal computations pointed out
  five types of  asymptotic regimes. They are governed by
\begin{enumerate}
\item The ordinary compressible Navier-Stokes equations with null radiation
(if $\cL\to0$ and $\cL^2\cL_s\to0$);
\item The compressible Navier-Stokes equation 
with an extra pressure term see \eqref{eq:eq2} (equilibrium diffusion regime 
corresponding to $\ep\ll\cL$ and $\cL^2\cL_s\to+\infty,$ or $\cL\to+\infty$);
\item The Navier-Stokes-Poisson equations  \eqref{eq:eq1} (or \eqref{eq:eq1bis})
(case $\ep\ll\cL\lesssim 1$ and $\nu^2\cL^2\cL_s\to m\in(0,+\infty)$);
\item The compressible Navier-Stokes equations coupled with a parabolic equation \eqref{eq:noneq1}
(nonequilibrium diffusion regime $\cL\approx\ep$ and $\nu^2\cL_s\cL^2\to m\in(0,+\infty)$);
\item  The compressible Navier-Stokes equations coupled with a damped equation \eqref{eq:noneq2}
(degenerate  nonequilibrium diffusion regime $\cL\approx\ep$ and $\cL_s\cL^2\to+\infty$).
\end{enumerate}

 \medbreak
 
 The rest of the paper is devoted to justifying rigorously the last four 
  asymptotics globally in time in the framework of small solutions with critical regularity.
In the next section, we introduce a few notations that will be needed to define our functional framework,
and give an overview of the strategy. 
Section \ref{s:linear} is devoted to a fine analysis of the linearized equations \eqref{eq:NSrad} about $(0,\vc 0,0,\vc 0),$
which turns out to be essentially the key to proving global results and justifying the diffusive asymptotics 
we have in mind. 
The next three sections are devoted to the rigorous justification of 
the nonequilibrium diffusion regime $\cL\approx\ep$ and $\cL_s\cL^2\gtrsim1,$
the equilibrium diffusion regime $\cL\to+\infty$ and of the Poisson type diffusion regime 
 ($\ep\ll\cL\lesssim 1$ and $\nu^2\cL^2\cL_s\to m\in(0,+\infty)$). 
 In all of those sections, we establish a global-in-time existence result for the expected   limit system, and for  \eqref{eq:NSrad} supplemented
 with uniform estimates (for coefficients $\cL$ and $\cL_s$  satisfying the assumptions of the studied regime), 
and eventually show the  convergence  of the solutions of \eqref{eq:NSrad} to those
 of the expected limit system.  Some estimates, of independent interest, for the solutions to a class of linear ODEs corresponding to 
 the linearized equations of \eqref{eq:NSrad} in the Fourier space are postponed in the appendix.

%%%%%%%%%%%%%%%%%%%%%%%%%%%%%%%%%%%%%%%%%%%%%%%%%%%

\section{Functional framework and overview of the method}\label{s:results}

The functional framework 
we shall work in is modeled on the linearized equations corresponding to \eqref{eq:NSrad}, 
and is thus the same as in our first paper \cite{DD} devoted to the global well-posedness issue
in critical regularity spaces for small perturbations of a stable constant state. 
The key to proving asymptotic results however, is to prescribe norms  depending on the parameters $\ep,$ $\cL$ 
and $\cL_s,$  so as  to get optimal uniform estimates, enabling our
justifying rigorously the different diffusive  asymptotics exhibited above. 
\smallbreak 
Let us  first very briefly recall the definition of homogeneous Besov spaces $\dot B^s_{2,1}$ (the reader is referred to \cite{BCD}, Chap. 2 for more details). 
For simplicity, we focus on the $\R^n$ case (adapting the
construction to the torus  being quite straightforward).
Fix some smooth radial bump function $\chi:\R^n\to[0,1]$  with $\chi\equiv1$ on $B(0,1/2)$ and $\chi\equiv0$ outside $B(0,1),$
nonincreasing with respect to the radial variable. Let $\varphi(\xi):=\chi(\xi/2)-\chi(\xi).$ 
The elementary spectral cut-off operator entering in the Littlewood-Paley decomposition  is defined by 
$$
\ddj u:=\varphi(2^{-j}D)u=\cF^{-1}(\varphi(2^{-j}D)\cF u),\qquad j\in\Z
$$
where we denote by $\cF$ the standard Fourier transform in $\R^n.$
\medbreak
For any $s\in\R,$ the \emph{homogeneous Besov space} $\dot B^s_{2,1}$ is the set of tempered distributions $u$ so that
$$
\|u\|_{\dot B^s_{2,1}}:=\sum_{j\in\Z} 2^{js}\|\ddj u \|_{L^2}<\infty,
$$
and \begin{equation}\label{eq:LFconv} \lim_{\lambda\to+\infty}\chi(\lambda D) u=0\quad\hbox{in}\quad L^\infty.\end{equation}
\medbreak
As pointed out in \cite{DD}, scaling considerations that neglect  low order terms of  System \eqref{eq:NSrad} suggest that critical regularity 
is $\dot B^{\frac n2-1}_{2,1}$ for $\vu_0,$ $j_{0,0}$ and $\vc j_{1,0},$ and $\dot B^{\frac n2}_{2,1}$ for $b_0.$ 
However,  to handle  lower order terms, one has to make additional assumptions
for the  low frequencies. To this end, 
it is convenient to introduce the following notation (where $\eta$ stands for  a positive  parameter)
$$
\|u\|_{\dot B^s_{2,1}}^{\ell,\eta}:=\sum_{2^k\leq2\eta} 2^{ks}\|\ddk u\|_{L^2}\quad\hbox{and}\quad
\|u\|_{\dot B^s_{2,1}}^{h,\eta}:=\sum_{2^k\geq\eta/2} 2^{ks}\|\ddk u\|_{L^2},
$$
and also
$$
u^{\ell,\eta}:=\sum_{2^k\leq\eta}\ddk u\quad\hbox{and}\quad
u^{h,\eta}:=\sum_{2^k>\eta}\ddk u.
 $$
Note that  $\|u^{\ell,\eta}\|_{\dot B^s_{2,1}}\leq C\|u\|_{\dot B^s_{2,1}}^{\ell,\eta}$
 and  $\|u^{h,\eta}\|_{\dot B^s_{2,1}}\leq C\|u\|_{\dot B^s_{2,1}}^{h,\eta}.$
Because the Littlewood-Paley decomposition is not quite orthogonal, it is important 
 to allow for a small overlap in the above definition of norms.
 
In some places, we will have to specify also the behavior for the middle frequencies, 
 by considering  for given $0<\eta<\eta',$ 
 $$
\|u\|_{\dot B^s_{2,1}}^{m,\eta,\eta'}:=\sum_{\eta\leq 2^k\leq\eta'} 2^{ks}\|\ddk u\|_{L^2}.
$$
Broadly speaking, our strategy to justify the different types of diffusive limits is as  follows:
\begin{itemize}
\item Step 1: We prove `uniform estimates' for the global solutions to \eqref{eq:NSrad}, 
uniform meaning that we want a bound independent of $\ep,$ but the 
norm itself may depend `in a nice way' of the parameters $\ep,$ $\cL$ and $\cL_s.$
\item Step 2:  We show that the limit system is globally well-posed in the small data case. 
\item Step 3: We take advantage of estimates of Step 1 to exhibit  weak compactness properties. 
Combining with the uniqueness result of Step 2, this allows to conclude to  the convergence  of 
the whole family of solutions of \eqref{eq:NSrad} to those of the limit system. 
\end{itemize}
The most technical part   is step 1, as it  requires a fine analysis of the
linearized equations of \eqref{eq:NSrad} about $0$ that keeps track of the coefficients $\cL,$ $\cL_s$ and $\ep.$ 
Schematically, in the Fourier space, one has to resort to different types of estimates for  low, medium and high frequencies. 
The low frequency analysis is carried out by considering  approximate eigenmodes  of the system, that
are constructed by a perturbative method from the (explicit) eigenmodes corresponding to null frequency. 
A part of the difficulty is that the `fluid modes' are of parabolic type, hence the corresponding eigenvalues tend quadratically to $0$
when the frequency size tends to $0$  while the radiative modes are expected to be exponentially damped. 
The high frequency analysis is inspired by the corresponding one for the barotropic Navier-Stokes equations, 
after noticing that coupling between radiative and fluid unknowns occurs only through $0$ order terms, and thus 
tend to be negligible for very high frequencies. 
Last but not least, medium frequency regime has to be looked at with the greatest care, as
the low and high frequency regimes need not overlap.  There is
no general strategy for handling them, apart from guessing  approximate eigenmodes of the system.

%%%%%%%%%%%%%%%%%%%%%%%%%%%%%%%%%%%%%%%%%%%%%%%%%%%

\section{Uniform estimates for the linearized equations}\label{s:linear}

In order to reduce  the study to the case where the
total viscosity  $\nu:=\lambda+2\mu$  is $1,$ and  to get
a symmetric first order system for the radiative unknowns, let  us set
\begin{equation}\label{eq:nu}
\bigl(b,\vu,j_0,\vec{ j_1}\bigr)(t,x):=
(b^\ep,\vu^\ep,\sqrt n\,j_0^\ep,\vc j^\ep_1)(\nu t,\nu x).
\end{equation}
Then  $(b^\ep,\vu^\ep,j_0^\ep,\vc j_1^\ep)$ satisfies 
\eqref{eq:NSrad}
if and only if $(b,\vu,j_0,\vc j_1)$ satisfies 
\begin{equation}\label{eq:NSradbis}\left\{
\begin{array}{l}
\partial_t b + \vu \cdot\nabla b +(1+k_1(b))\div \vu=0,\\[1.5ex]
 \partial_t  \vu + \vu\cdot\nabla\vu- (1+k_2(b))\wt{\mathcal A}\vu
 +(1+k_3(b))\nabla b = \frac{\wt\cL\cM}n
 (1+k_4(b))\vc j_1,\\[1.5ex]
 \ep\partial_t j_0 + \frac{1}{\sqrt n}\: \div\vc j_1= \wt{\mathcal L}(b-\sqrt n\, j_0),\\[1.5ex]
\ep\partial_t \vc j_1+ \frac1{\sqrt n}{\nabla j_0} = -\wt\cL\cM\vc j_1,\\[1.5ex]
 \end{array}\right.
\end{equation} 
with
\begin{equation}\label{eq:notcoeff}
{\cM}:=1+\cL_s, \quad\wt\cL:=\nu\cL\ \hbox{ and }\  \wt\cA:=\nu^{-1}\cA.
\end{equation}
 The  corresponding linearized system  reads
\begin{equation}\label{eq:diff}
\left\{\begin{array}{l}
 \partial_t b+ \div \vu=f,\\[1ex]
  \partial_t \vu -\wt\cA\vu +  \nabla b    -\frac{\wt\cL\cM}n\vc j_1=\vec g,\\[1ex]
\ep \partial_t j_0+  \frac1{\sqrt n}\,\div\vec j_1
+\wt\cL (j_0-\sqrt n\,b)=0,\\[1ex]
 \ep\partial_t \vec j_1 +  \frac1{\sqrt n}\,\nabla j_0+\wt\cL\cM\vec j_1=\vec 0.
 \end{array}\right.
\end{equation}
On one hand, the coupling between the incompressible part of $\vu$ and $\vc j_1$
 that is $\cP\vu$ and $\cP\vc j_1$ where $\cP$ stands for the projector
 on divergence-free vector-fields is obvious as
\begin{equation}\label{eq:Pu}
\d_t\cP\vu-\frac\mu\nu\Delta\cP\vu=\frac{\wt\cL\cM}n\cP\vc j_1,
\end{equation} and 
$$
\cP\vc j_1(t)=e^{-\frac{\wt\cL{\cM} t}\ep}\cP\vc j_1(0),
$$
hence in any functional space $X$ we have
\begin{equation}\label{eq:Pj1b}
\wt\cL\cM\|\cP\vc j_1\|_{L^1(X)}\leq \ep\|\cP\vc j_1(0)\|_X.
\end{equation}
On the other hand, the coupling between $b,$  $d:=\Lambda^{-1}\div\vu,$ $j_0$ 
and $j_1:=\Lambda^{-1}\div\vc j_1$  (where $\Lambda^s:=(-\Delta)^{s/2}$) is quite complicated: in Fourier variables, we have
\begin{equation}\label{eq:diff0}
\frac d{dt}\left(\begin{array}{c}\wh b\\\wh d\\\wh j_0\\\wh j_1\end{array}\right)
+\left(\begin{array}{cccc}0&\rho&0&0\\
-\rho&\rho^2&0&-\frac{\wt\cL\cM}n\\
-\frac{\sqrt n\wt\cL}\ep&0&\frac{\wt\cL}\ep&\frac\rho{\ep\sqrt n}\\
0&0&-\frac{\rho}{\ep\sqrt n}&\frac{\wt\cL\cM}\ep\end{array}\right)
\left(\begin{array}{c}\wh b\\\wh d\\\wh j_0\\\wh j_1\end{array}\right)
=\left(\begin{array}{c}0\\0\\0\\0\end{array}\right)\cdotp
\end{equation}
The analysis that has been performed in \cite{DD}   pointed out the following
necessary and sufficient stability condition
\begin{equation}\label{eq:stabdiffu}
\wt\cL>\frac\ep n(1+\cM^{-1}).
\end{equation}
So we shall make this assumption in all that follows. Of course one also has to keep in mind that $\cM>1,$  a consequence of $\cM:=1+\cL_s.$
For notational simplicity, we shall simply denote $\wt\cL$ by $\cL$ in the following computations.

\subsection{Estimates for small $\rho$}

In order to prove estimates in the case $0\leq\rho\leq C_1$ (with $C_1\geq\sqrt{1+n^{-1}}$), we shall use that  \eqref{eq:diff0} enters in the class of ODEs 
 that has been considered in the Appendix. Indeed, it corresponds to \eqref{eq:class}  with 
\begin{equation}\label{eq:coeffdiff}
 \varsigma=\frac{\wt\cL\cM}n,\quad\eta=\frac{\sqrt n\wt\cL}\ep,\quad\beta=\frac{\wt\cL}\ep,\quad
\alpha=\frac1{\ep\sqrt n},\quad\gamma=\frac{\wt\cL\cM}\ep\cdotp
\end{equation}

\subsubsection{The case $\cL\gtrsim1$ and $\cL\ep\cM\gtrsim1$}

We shall follow the first approach proposed in  Appendix \ref{s:A}. It  corresponds  to the following 
matrices $A_0,$ $A_1,$ $A_2$ and $B_1$
$$
A_0=\left(\begin{array}{cccc} 0&0&0&0\\0&0&0&0\\0&0&\frac\cL\ep&0\\0&0&0&\frac{\cL\cM}\ep\end{array}\right),
\qquad
A_1=\left(\begin{array}{cccc} 0&1&0&0\\-1-\frac1n&0&0&0\\0&0&0&\frac{1+\ep^2}{\sqrt n\,\ep}\\0&0&-\frac1{\sqrt n\,\ep}&0\end{array}\right),
$$
$$
B_1=-\left(\begin{array}{cccc} 0&0&0&\frac\ep n\\0&0&\frac1{n^{3/2}}&0\\0&{\sqrt n}&0&0\\\frac1\ep&0&0&0\end{array}\right)
\quad\hbox{and}\quad
A_2=\left(\begin{array}{cccc} 0&0&0&0\\0&1&0&-\frac\ep n\\0&0&0&0\\0&0&0&0\end{array}\right)\cdotp
$$
Therefore  we  set 
\begin{equation}\label{eq:Pdiffu}
P:=\left(\begin{array}{cccc}0&0&0&\frac{\ep^2}{n\cL\cM}\\
0&0&\frac\ep{n^{3/2}\cL}&0\\
0&-\frac{\ep\sqrt n}\cL&0&0\\
-\frac1{\cL\cM}&0&0&0\end{array}\right),
\end{equation}
which corresponds to the change of unknowns
\begin{equation}\label{eq:chgdiffu}
\left(\begin{array}{c}\wh\fb\\ \wh\fd\\ \wh\fj_0\\ \wh\fj_1\end{array}\right)
:=\left(\begin{array}{cccc}1&0&0&\frac{\ep^2}{n\cL\cM}\rho\\
-\frac\ep{n\cL}\rho&1&\frac\ep{n^{3/2}\cL}\rho&\frac\ep n\\
-\sqrt n&-\frac{\sqrt n\,\ep}{\cL}\rho&1&-\frac{\ep^2}{\sqrt n\,\cL}\rho\\
-\frac1{\cL\cM}\rho&0&0&1\end{array}\right)\left(\begin{array}{c}\wh b\\ \wh d\\ \wh j_0\\ \wh j_1\end{array}\right)\cdotp
\end{equation}
According  to \eqref{eq:smallrho1},  working with $(\wh a,\wh d,\wh j_0,\wh j_1)$ or $(\wh\fb,\wh\fd,\wh\fj_0,\wh\fj_1)$ is equivalent whenever
\begin{equation}\label{eq:smallrhodiffu1}
\rho\lesssim\cL\min(\ep^{-1},\cM).
\end{equation}
Let us first compute the matrices  $PB_1,$ $[P,A_1]$ and $A_3$ appearing in \eqref{eq:V}
$$
PB_1=\frac\ep{n\cL}\left(\begin{array}{cccc}-\cM^{-1}&0&0&0\\
0&-1&0&0\\0&0&1&0\\
0&0&0&\cM^{-1}\end{array}\right),
$$
$$[P,A_1]=\frac1\cL\left(\begin{array}{cccc}0&0&-\frac{\ep(1+\cM^{-1})}{n^{3/2}}&0\\
0&0&0&\frac{1+\ep^2(1+\cM^{-1}(n+1))}{n^2}\\
\frac{1+\ep^2(1+\cM(1+n))}{\ep\sqrt n\,\cM}&0&0&0\\
0&-1-\cM^{-1}&0&0\end{array}\right),
$$
$$
A_3=\frac1n\left(\begin{array}{cccc}0&\frac{\ep^2}{\cL^2}&0&\frac{\ep^3}{n\cL^2\cM^2}\\
\frac\ep{\cL\cM}-\frac{(1+n^{-1})\ep^2}{\cL^2\cM^2}&0&\frac\ep{\cL\sqrt n}-\frac{\ep^2}{\cL^2n^{3/2}}&0\\
0&0&0&\frac{\ep(1+\ep^2)}{\sqrt n(\cL\cM)^2}\\
0&0&-\frac\ep{\cL^2\sqrt n}&0\end{array}\right)\cdotp
$$
Because  $\cL\gtrsim1,$ we thus have $|A_3|\lesssim \frac\ep{\cL}\cdotp$
Hence, up to a $\cO(\ep\rho^3/\cL)$ term, the equations for $(\wh\fb,\wh\fd)$ read
$$\displaylines{
\frac d{dt}\left(\begin{array}{c}\wh\fb\\\wh\fd\end{array}\right)
+\rho\left(\begin{array}{cc}0&1\\-1-n^{-1}&0\end{array}\right)\left(\begin{array}{c}\wh\fb\\\wh\fd\end{array}\right)
+\rho^2\left(\begin{array}{cc}-\frac\ep{n\cL\cM}&0\\0&1-\frac\ep{n\cL}\end{array}\right)\left(\begin{array}{c}\wh\fb\\\wh\fd\end{array}\right)
\hfill\cr\hfill=\rho^2\left(\begin{array}{cc}\frac{\ep(1+\cM^{-1})}{n^{3/2}\cL}&0\\0&\frac\ep n-\frac{1+\ep^2(1+\cM^{-1}(n+1))}{n^2\cL}\end{array}\right)
\left(\begin{array}{c}\wh\fj_0\\\wh\fj_1\end{array}\right)\cdotp}
$$
In order to estimate $(\wh\fb,\wh\fd),$ we just follow the method of  Appendix \ref{s:B}, which requires 
Condition \eqref{eq:stabdiffu} and
\begin{equation}\label{eq:rholf}
\rho\leq\frac{\sqrt{1+n^{-1}}}{1-\frac{(1-\cM^{-1})\ep}{n\cL}}\cdotp
\end{equation}
Keeping \eqref{eq:smallrhodiffu1} in mind and noticing that
$$\tilde\nu=1-\frac\ep{n\cL}(1+\cM^{-1}),$$ 
is of order $1$ for small $\ep,$ we thus conclude that if 
\begin{equation}\label{eq:diffuassumption1}
\rho\leq\sqrt{1+n^{-1}}\quad\hbox{and}\quad
\rho\lesssim\cL\min(\ep^{-1},\cM),
\end{equation}
then 
\begin{multline}\label{eq:diffu1}
|(\wh\fb,\wh\fd)(t)|+\rho^2\int_0^t|(\wh\fb,\wh\fd)|\,d\tau
\lesssim|(\wh\fb,\wh\fd)(0)|\\+\rho^2\int_0^t\biggl(\frac\ep\cL|\wh\fj_0|+\Bigl(\ep+\frac1\cL\Bigr)|\wh\fj_1|\biggr)d\tau
+\frac{\ep\rho^3}{\cL}\int_0^t|(\wh\fb,\wh\fd,\wh\fj_0,\wh\fj_1)|\,d\tau.
\end{multline}
Next, we see that the equations for $(\wh\fj_0,\wh\fj_1)$ read (omitting the  $\cO(\ep\rho^3/\cL)$ term)
\begin{multline}\label{eq:j0j1}
\frac d{dt}\left(\begin{array}{c}\wh\fj_0\\\wh\fj_1\end{array}\right)
+\left(\begin{array}{cc}\frac\cL\ep+\frac{\ep\rho^2}{n\cL}&0\\0&\frac{\cL\cM}\ep+\frac{\ep\rho^2}{n\cL\cM}\end{array}\right)\left(\begin{array}{c}\wh\fj_0\\\wh\fj_1\end{array}\right)+\frac\rho{\sqrt n\,\ep}\left(\begin{array}{cc}0&1+\ep^2\\-1&0\end{array}\right)\left(\begin{array}{c}\wh\fj_0\\\wh\fj_1\end{array}\right)
\\=\rho^2\left(\begin{array}{cc}-\frac{1+\ep^2(1+\cM(1+n))}{\ep\sqrt n\,\cL\cM}&0\\0&\frac{1+\cM^{-1}}\cL\end{array}\right)\left(\begin{array}{c}\wh\fb\\\wh\fd\end{array}\right)\cdotp\end{multline}
Therefore, computing
\begin{equation}\label{eq:diffu0}
\frac d{dt}\biggl(|\wh\fj_0|^2+(1+\ep^2)|\wh\fj_1|^2\biggr),
\end{equation}
so as to eliminate the term in $\rho,$ we end up with 
\begin{multline}\label{eq:diffu2}
|(\wh\fj_0,\wh\fj_1)(t)|+\frac\cL\ep\int_0^t|\wh\fj_0,\wh\fj_1|\,d\tau\lesssim|(\wh\fj_0,\wh\fj_1)(0)|
\\+\rho^2\int_0^t\biggl(\Bigl(\frac1{\ep\cL\cM}+\frac\ep\cL\Bigr)|\wh\fb|+\frac1\cL|\wh\fd|\biggr)d\tau
+\frac{\ep\rho^3}{\cL}\int_0^t|(\wh\fb,\wh\fd,\wh\fj_0,\wh\fj_1)|\,d\tau.
\end{multline}
Now, adding up \eqref{eq:diffu1} and \eqref{eq:diffu2}, we easily conclude that
if $\ep$ is small enough
and \begin{equation}\label{eq:diffuassumption2}
\cL\min(1,\ep\cM)\gtrsim 1,
\end{equation}
 then we have
\begin{equation}\label{eq:diffulf1}
|(\wh\fb,\wh\fd,\wh\fj_0,\wh\fj_1)(t)|+\rho^2\int_0^t|(\wh\fb,\wh\fd)|\,d\tau
+\frac\cL\ep\int_0^t|(\wh\fj_0,\wh\fj_1)|\,d\tau\lesssim |(\wh\fb,\wh\fd,\wh\fj_0,\wh\fj_1)(0)|,
\end{equation}
whenever $0\leq\rho\leq\sqrt{1+n^{-1}}.$
\medbreak
Now, resuming to the $\wh\fj_1$ equation in \eqref{eq:j0j1}, and evaluating the first
order term according to \eqref{eq:diffulf1}, we deduce that, in addition 
\begin{equation}\label{eq:diffulf1b}
\frac{\cL\cM}\ep\int_0^t|\wh\fj_1|\,d\tau\lesssim |(\wh\fb,\wh\fd,\wh\fj_0,\wh\fj_1)(0)|.
\end{equation}

\subsubsection{The case $\ep\ll\cL\lesssim1$ with $\ep\cL^2\cL_s^2\gg1$ and $\cL^2\cL_s\gg1$}

If  $\cL\ll1$ then  plugging \eqref{eq:diffu1} in \eqref{eq:diffu2} does not
allow to get \eqref{eq:diffulf1} any longer. 
In order to overcome this, we shall  follow  the \emph{second} approach 
proposed in Appendix  \ref{s:A}  with coefficients defined as in \eqref{eq:coeffdiff}: 
we  set
$$
P=\left(\begin{array}{cccc}0&0&0&\frac{\ep^2}{n\cL\cM}\\
0&0&\frac\ep{n^{3/2}\cL}&0\\
0&-\frac{\ep\sqrt n}{\cL}&0&\frac{1+\ep^2}{\sqrt n(1-\cM)\cL}\\
-\frac1{\cL\cM}&0&\frac1{\sqrt n\,(1-\cM)\cL}&0\end{array}\right),
$$
and we thus have, remembering that $\cM-1=\cL_s$ 
\begin{equation}\label{eq:chgdiffu2}
V=\left(\begin{array}{c}\wh\fb\\ \wh\fd\\ \wh\fj_0\\ \wh\fj_1\end{array}\right)
:=\left(\begin{array}{cccc}1&0&0&\frac{\ep^2}{n\cL\cM}\rho\\
-\frac\ep{n\cL}\rho&1&\frac\ep{n^{3/2}\cL}\rho&\frac\ep n\\
-\sqrt n&-\frac{\sqrt n\,\ep}{\cL}\rho&1&-\frac{1+\ep^2\cM}{\sqrt n\,\cL\cL_s}\rho\\
\frac\rho{\cL\cM\cL_s}&0&-\frac{\rho}{\sqrt n\,\cL_s\cL}&1\end{array}\right)\left(\begin{array}{c}\wh b\\ \wh d\\ \wh j_0\\ \wh j_1\end{array}\right)\cdotp
\end{equation}
The determinant of the above matrix is 
$$
\biggl(1+\frac{\ep^2}{n\cL^2}\rho^2\biggr)\biggl(1+\frac{\ep^2}{n\cL^2\cM^2}\rho^2\biggr)-
\frac{1+\ep^2}{n\cL^2\cL_s^2}\rho^2,
$$
Hence working with $(\wh b,\wh d,\wh j_0,\wh j_1)$ or $(\wh\fb,\wh\fd,\wh\fj_0,\wh\fj_1)$ is equivalent whenever
\begin{equation}\label{eq:smallrhodiffu2}
\rho\lesssim\frac\cL\ep\quad\hbox{and}\quad\rho\leq\sqrt n\,\cL\cL_s.
\end{equation}
Then following the computations of Appendix \ref{s:A}, second approach, and 
setting  $A_3:=(PA_0-A_1)P^2+ A_2 P$  leads to
\begin{multline}\label{eq:Vdiffu}
\frac d{dt}V+\rho\left(\begin{array}{cccc}0&1&0&0\\-1-\frac1n&0&0&0\\0&0&0&0\\0&0&0&0\end{array}\right)\!V
\\+\left(\begin{array}{cccc}-\frac{\ep}{n\cL\cM}\rho^2&0&0&0\\
0&\bigl(1\!-\!\frac{\ep}{n\cL}\bigr)\rho^2&0&0\\
0&0&\frac\cL\ep+\bigl(\ep+\frac{1+\ep^2}{\ep\cL_s}\bigr)\frac{\rho^2}{n\cL}&0\\
0&0&0&\frac{\cL\cM}\ep+\bigl(\frac{\ep}{n\cL\cM}\!-\!\frac{1+\ep^2}{n\ep\cL\cL_s}\bigr)\rho^2\end{array}\right)\!V
\\=\rho^2\left(\begin{array}{cccc}0&0&\frac{(1+\cM^{-1})\ep}{n^{3/2}\cL}&0\\
0&0&0&\frac\ep n-\frac{1+\ep^2}{n^2\cL}-\frac{\ep^2(1+n)}{n^2\cL\cM}\\
-\frac{1+\ep^2}{\ep\sqrt n\,\cL_s\cL}-\frac{(n+1)\ep}{\sqrt n\,\cL}&0&0&0\\
0&-\frac{1}{\cL\cL_s\cM}&0&0
\end{array}\right)V\hfill\cr\hfill
+\rho^3(I+\rho P)A_3(I+\rho P)^{-1}V.\end{multline}
Let us  bound $A_3$ in order to determinate for which values of $\rho$ the last term in \eqref{eq:Vdiffu} is indeed negligible. Just writing that $|A_3|\leq |P|^2(|P|\,|A_0|+|A_1|)+|A_2|\,|P|$ using the explicit values of $A_0,$ $A_1$ and $A_2$ and   
\begin{equation}\label{eq:diffuboundP}
|P|\lesssim\frac1\cL\max\biggl(\ep,\frac1{\cL_s}\biggr),
\end{equation}
does not provide an accurate enough bound  for $A_3.$ Hence one has to go to further computations. Now, we get 
$$
PA_0P^2=\left(\!\!\begin{array}{cccc}
0&\frac{\ep^2}{n\cL^2\cL_s}&0&\frac{\ep}{n^2\cL^2}\bigl(\frac{1+\ep^2}{\cL_s^2}-\frac{\ep^2}{\cM^2}\bigr)\\ 
\frac{1+\ep^2}{n^2\cL^2\cM\cL_s}&0&\frac1{n^{5/2}\cL^2}\bigl(\frac{1+\ep^2}{\cL_s^2}-\ep^2\bigr)&0\\ 
0&-\frac{(1+\ep^2)\cM}{\sqrt n\,\cL^2\cL_s^2}&0&\frac{(1+\ep^2)\cM}{n^{3/2}\ep\cL_s\cL^2}\bigl(\frac{\ep^2}{\cM^2}-\frac{1+\ep^2}{\cL_s^2}\bigr)\\  
-\frac{\ep(1+\ep^2)}{n\cL^2\cM\cL_s^2}&0&\frac{\ep}{n^{3/2}\cL_s\cL^2}\bigl(\ep^2-\frac{1+\ep^2}{\cL_s^2}\bigr)&0 
\end{array}\!\!\right),
$$
$$
A_2P=\left(\!\!\begin{array}{cccc}0&0&0&0\\
\frac\ep{n\cL\cM}&0&\frac\ep{n^{3/2}\cL}&0\\
0&0&0&0\\0&0&0&0\end{array}\!\!\right)
\!\quad\!\hbox{and}\!\!\quad
A_1P^2=\left(\!\!\begin{array}{cccc}0&-\frac{\ep^2}{n\cL^2}&0&-\frac{\ep(1+\ep^2)}{n\cL^2\cL_s}\\
\frac{(1+n^{-1})\ep^2}{n\cL^2\cM^2}&0&\frac{(1+n^{-1})\ep^2}{n^{3/2}\cL^2\cL_s\cM}\\
0&0&0&0\\0&0&0&0\end{array}\!\!\right)\cdotp
$$
Hence, given that $\ep\ll\cL\lesssim1$ and that $\cL_s\approx\cM$ in 
the regime that we are considering, one may conclude that
\begin{equation}\label{eq:diffuA3}
|A_3|\lesssim\max\biggl(\frac\ep\cL,\frac1{\cL^2\cL_s},\frac1{\ep\cL^2\cL_s^2}\biggr)\cdotp
\end{equation}
Note that we still have $\tilde\nu\to1$ for $\ep\to0.$ Hence
 applying the method of the appendix to handle $(\wh\fb,\wh\fd),$ we find out that 
if \eqref{eq:stabdiffu}, \eqref{eq:rholf} and  \eqref{eq:smallrhodiffu2} are fulfilled then
$$\displaylines{
|(\wh\fb,\wh\fd)(t)|+\rho^2\int_0^t|(\wh\fb,\wh\fd)|\,d\tau\lesssim|(\wh\fb,\wh\fd)(0)|
+\rho^2\int_0^t\biggl(\frac{\ep}{\cL}|\wh\fj_0|+\frac1\cL|\wh\fj_1|\biggr)d\tau
\hfill\cr\hfill+\rho^3\max\biggl(\frac\ep\cL,\frac1{\cL^2\cL_s},\frac1{\ep\cL^2\cL_s^2}\biggr)\int_0^t|V|\,d\tau.}
$$
As regards the radiative modes, we have
$$\displaylines{
|\wh\fj_0(t)|+\biggl(\frac\cL\ep+\Bigl(\ep+\frac{1+\ep^2}{\ep\cL_s}\Bigr)\frac{\rho^2}{n\cL}\biggr)\int_0^t|\wh\fj_0|\,d\tau
\leq|\wh\fj_0(0)|+C\rho^2\int_0^t\biggl(\frac1{\ep\cL\cL_s}+\frac\ep\cL\biggr)|\wh\fb|\,d\tau
\hfill\cr\hfill+C\rho^3\max\biggl(\frac\ep\cL,\frac1{\cL^2\cL_s},\frac1{\ep\cL^2\cL_s^2}\biggr)\int_0^t|V|\,d\tau,}$$
$$\displaylines{
|\wh\fj_1(t)|+\biggl(\frac{\cL\cM}\ep+\Bigl(\frac{\ep}{n\cL\cM}-\frac{1+\ep^2}{n\ep\cL\cL_s}\Bigr)\rho^2\biggr)\int_0^t|\wh\fj_1|\,d\tau
\leq|\wh\fj_1(0)|+C\frac{\rho^2}{\cL\cL_s\cM}\int_0^t|\wh\fd|\,d\tau
\hfill\cr\hfill+C\rho^3\max\biggl(\frac\ep\cL,\frac1{\cL^2\cL_s},\frac1{\ep\cL^2\cL_s^2}\biggr)\int_0^t|V|\,d\tau.}
$$
From the above three inequalities, we get for any $A\in(0,1]$
$$\displaylines{
|(\wh\fb,\wh\fd)(t)|+A|\wh\fj_0(t)|+|\wh\fj_1(t)|+\rho^2\int_0^t|(\wh\fb,\wh\fd)|\,d\tau
+A\frac\cL\ep\int_0^t|\wh\fj_0|\,d\tau+\frac{\cL\cL_s}\ep\int_0^t|\wh\fj_1|\,d\tau
\hfill\cr\hfill
\lesssim|(\wh\fb,\wh\fd)(0)|+A|\wh\fj_0(0)|+|\wh\fj_1(0)|
+\rho^2\int_0^t\biggl(\frac{\ep}{\cL}|\wh\fj_0|+\frac1\cL|\wh\fj_1|\biggr)d\tau
+A\rho^2\int_0^t\biggl(\frac1{\ep\cL\cL_s}+\frac\ep\cL\biggr)|\wh\fb|\,d\tau\hfill\cr\hfill
+\frac{\rho^2}{\cL\cL_s^2}\int_0^t|\wh\fd|\,d\tau
+\rho^3\max\biggl(\frac\ep\cL,\frac1{\cL^2\cL_s},\frac1{\ep\cL^2\cL_s^2}\biggr)
\int_0^t|V|\,d\tau.}
$$
Now, we notice that taking $A=c_0\min(1,\ep\cL\cL_s)$ for a sufficiently small constant $c_0$ 
allows to absorb all the terms of the r.h.s. (but the data) by the l.h.s. 
provided we have $0\leq\rho\leq\sqrt{1+n^{-1}}$ 
\begin{equation}
\ep\ll\cL\lesssim1,\quad \ep\cL^2\cL_s^2\gg1\ \hbox{ and }\ \cL^2\cL_s\gg1.
\end{equation}
We thus conclude that for all $0\leq\rho\leq\sqrt{1+n^{-1}},$ we have 
\begin{multline}\label{eq:diffulf2}
|(\wh\fb,\wh\fd,\min(1,\ep\cL\cL_s)\wh\fj_0,\wh\fj_1)(t)|+\rho^2\int_0^t|(\wh\fb,\wh\fd)|\,d\tau
+\frac\cL\ep\min(1,\ep\cL\cL_s)\int_0^t|\wh\fj_0|\,d\tau\\+\frac{\cL\cM}\ep\int_0^t|\wh\fj_1|\,d\tau\leq C |(\wh\fb,\wh\fd,\min(1,\ep\cL\cL_s)\wh\fj_0,\wh\fj_1)(0)|.
\end{multline}
Let us point out that  in the case where $\cL^2\cL_s\approx1$  (even if $\cL\approx\ep$ in fact) 
then the same computation will lead to \eqref{eq:diffulf2}, but only for $0\leq\rho\leq c,$
with $c$ a small enough constant.

%%%%%%%%%%%%%%%%%%%%%%%%%%%%%%%%%%%%%%%%%

\subsubsection{The case $\ep\ll\cL\lesssim\ep^{1/2}$ and $\cL^2\cL_s\approx1$}

As the value $\sqrt{1+n^{-1}}$ will not play any particular role, we fix some $C_1>c$ in the following computations. We want  to  get \eqref{eq:diffulf2} for  $\rho\in[c,C_1].$ To this end, we introduce 
$\zeta_0$ and $\zeta_1$ such that
$$
\wh\zeta_0:= \wh j_0-\frac{\sqrt n}{1+\frac{\rho^2}{n\cL^2\cM}}\wh b
\quad\hbox{and}\quad
\wh\zeta_1:=\wh j_1-\frac\rho{\sqrt n \,\cL\cM}\wh j_0.
$$
Then we discover that $(\wh b,\wh d, \wh\zeta_0,\wh\zeta_1)$ fulfills
$$
\left\{\begin{array}{l}
\d_t\wh b+\rho\wh d=0,\\[1ex]
\d_t\wh d+\rho^2\wh d-\rho\Biggl(1+\Frac{1}{n\bigl(1+\frac{\rho^2}{n\cL^2\cM}\bigr)}\Biggr)\wh b
=\frac\rho{n^{3/2}}\wh\zeta_0+\frac{\cL\cM}n\wh\zeta_1,\\[1ex]
\d_t\wh\zeta_0+\Frac\cL\ep\wh\zeta_0=-\frac{\rho}{\ep\sqrt n}\wh\zeta_1-\sqrt n\frac\rho{1+\frac{\rho^2}{n\cL^2\cM}}\wh d,\\[1ex]
\d_t\wh\zeta_1+\Frac\cL\ep\biggl(\cM-\Frac{\rho^2}{n\cL^2\cM}\biggr)\wh\zeta_1
=\frac\rho{\sqrt n\,\ep\cM}\biggl(1+\frac{\rho^2}{n\cL^2\cM}\biggr)\wh\zeta_0.
\end{array}\right.
$$
Let $\rho$ be in  $[c,C_1].$
For the first two equations,  performing  the standard barotropic estimates (which rely on the use of $\cU_\rho$ defined in \eqref{def:U})
leads to 
\begin{equation}\label{eq:diffulf5}
|(\wh b,\wh d)(t)|+\int_0^t|(\wh b,\wh d)|\,d\tau
\lesssim |(\wh b,\wh d)(0)|+\int_0^t|\wh\zeta_0|\,d\tau+\cL\cL_s\int_0^t|\wh\zeta_1|\,d\tau.
\end{equation}
For $\wh\zeta_0,$ it is obvious that 
\begin{equation}\label{eq:diffulf6}
|\wh\zeta_0(t)|+\frac\cL\ep\int_0^t|\wh\zeta_0|\,d\tau\leq|\wh\zeta_0(0)|+\frac{\rho}{\ep\sqrt n}
\int_0^t|\wh\zeta_1|\,d\tau+\rho\sqrt n\int_0^t|\wh d|\,d\tau,
\end{equation}
and, as  our conditions on $\cL$ and $\cL_s$ guarantee that $\rho\leq\sqrt{n/2}\,\cL\cM$ for small enough $\ep,$ we also have
\begin{equation}\label{eq:diffulf7}
|\wh\zeta_1(t)|+\frac{\cL\cM}{2\ep}\int_0^t|\wh\zeta_1|\,d\tau
\leq|\wh\zeta_1(0)|+\frac\rho{\sqrt n\,\ep\cM}\biggl(1+\frac{\rho^2}{n\cL^2\cM}\biggr)\int_0^t|\wh d|\,d\tau.
\end{equation}
Putting together those three inequalities, we readily get for all $A,B>0,$ observing that  $\rho^2\approx\cL^2\cM\approx1$
$$\displaylines{
|(\wh b,\wh d)(t)|+A\frac{\ep}\cL|\wh\zeta_0(t)|
+B\ep|\wh\zeta_1(t)|+\int_0^t|(\wh b,\wh d)|\,d\tau
+A\int_0^t|\wh\zeta_0|\,d\tau+B\cL\cM\int_0^t|\wh\zeta_1|\,d\tau\hfill\cr\hfill
\lesssim |(\wh b,\wh d)(0)|+A\frac{\ep}\cL|\wh\zeta_0(0)|+B\ep|\wh\zeta_1(0)|
+\cL\cM\int_0^t|\wh\zeta_1|\,d\tau\hfill\cr\hfill+\int_0^t|\wh\zeta_0|\,d\tau
+\frac{A}\cL\int_0^t|\wh\zeta_1|\,d\tau+A\frac{\ep}\cL\int_0^t|\wh d|\,d\tau
+\frac B\cM\int_0^t|\wh\zeta_0|\,d\tau.}
$$
It is now clear that if one takes first $A$ large enough (independently of $\ep$) and  $B$ much larger, 
all the integrals of the r.h.s. may be absorbed 
by the left-hand side, as  $\cM^{-1}\ll\rho,$ 
and we thus get for all $\rho\in[c,C_1]$
\begin{multline}\label{eq:diffulf3a}
|(\wh b,\wh d)(t)|+\frac{\ep}\cL|\wh\zeta_0(t)|
+\ep|\wh\zeta_1(t)|+\int_0^t|(\wh b,\wh d)|\,d\tau\\
+\int_0^t|\wh\zeta_0|\,d\tau+\cL\cM\int_0^t|\wh\zeta_1|\,d\tau
\lesssim |(\wh b,\wh d)(0)|+\frac{\ep}\cL|\wh\zeta_0(0)|+\ep|\wh\zeta_1(0)|.
\end{multline}
%That inequality may be improved from the equations satisfied by $\wh\zeta_0$ and $\wh\zeta_1$ by computing 
%$$\frac12\frac d{dt}\biggl(|\wh\zeta_0|^2+\frac\cM{1+\frac{\rho^2}{n\cL^2\cM}}|\wh\zeta_1|^2\biggr).$$
%We easily get $$|(\wh\zeta_0,\sqrt\cM\wh\zeta_1)(t)|+\frac\cL\ep\int_0^t|(\wh\zeta_0,\sqrt\cM\wh\zeta_1)|\,d\tau
%\lesssim |(\wh\zeta_0,\sqrt\cM\wh\zeta_1)(0)|+\int_0^t|\wh d|\,d\tau,$$
%which, combined with \eqref{eq:diffulf5} and observing that $\ep^2\cM\ll1,$ implies that 
%\begin{equation}\label{eq:diffulf8}|(\wh b,\wh d,\wh\zeta_0,\sqrt\cM\wh\zeta_1)(t)|+\int_0^t|(\wh b,\wh d)|\,d\tau
%+\frac\cL\ep\int_0^t|(\wh\zeta_0,\sqrt\cM\wh\zeta_1)|\,d\tau\lesssim |(\wh b,\wh d,\wh\zeta_0,\sqrt\cM\wh\zeta_1(0)|.\end{equation}
Plugging this new inequality in \eqref{eq:diffulf7}, we easily deduce that
$$
|\wh\zeta_1(t)|+\frac{\cL\cM}\ep\int_0^t|\wh\zeta_1|\,d\tau\lesssim   |\wh\zeta_1(0)|+\cL |\wh\zeta_0(0)| +\frac1{\ep\cM}|(\wh b,\wh d)(0)|,
$$
then inserting this information and \eqref{eq:diffulf3a}  in \eqref{eq:diffulf6}, we discover that
$$
|\wh\zeta_0(0)|+\frac{\cL}\ep\int_0^t|\wh\zeta_0|\,d\tau\lesssim      |\wh\zeta_0(0)| +\cL   |\wh\zeta_1(0)|+
\biggl(1+\frac1{\ep\cL\cM^2}\biggr) |(\wh b,\wh d)(0)|.
$$
Therefore, putting \eqref{eq:diffulf3a} and the above two inequalities together, using that 
 $|(\wh b,\wh d,\wh\zeta_0,\wh\zeta_1)|\approx|(\wh b,\wh d,\wh j_0,\wh j_1)|$
 \emph{and assuming in addition that $\cL\lesssim\ep^{1/2},$}
  we conclude that
\begin{multline}\label{eq:diffulf9}
|(\wh b,\wh d,\frac\ep\cL\wh j_0,\wh j_1)(t)|+\int_0^t|(\wh b,\wh d)|\,d\tau+
\frac\cL\ep\int_0^t|\wh\zeta_0|\,d\tau+\frac{\cL\cL_s}\ep \int_0^t|\wh\zeta_1|\,d\tau\\
\lesssim |(\wh b,\wh d,\frac\ep\cL\wh j_0,\wh j_1)(0)|\quad\hbox{for all}\ c\leq\rho\leq C_1.
\end{multline}
Note that due to the expression of $\wh\zeta_1,$ one may replace $\wh\zeta_1$ with $\wh j_1$ of $\wh\fj_1$ in the integral. 
\medbreak
Still in the case $\cL\lesssim\ep^{1/2},$ we claim that we have the following inequality 
\begin{equation}\label{eq:diffulf10}
|\rho \wh j_0(t)|+\frac\cL\ep\int_0^t|\rho\wh \zeta_0|\,d\tau\lesssim
|(\wh b,\wh d,\rho \wh j_0,\frac\ep\cL \wh j_0,\wh j_1)(0)|\quad\hbox{for all}\quad 0\leq\rho\leq C_1,
\end{equation}
which turns out to be crucial in the justification of the asymptotics
toward \eqref{eq:eq1}.
\medbreak
Indeed, inequality \eqref{eq:diffulf6} does not require any assumption on $\rho$ and thus
implies that
$$
|\rho\wh\zeta_0(t)|+\frac\cL\ep\int_0^t|\rho\wh\zeta_0|\,d\tau\leq|\rho\wh\zeta_0(0)|+\frac{\rho^2}{\ep\sqrt n}
\int_0^t|\wh\zeta_1|\,d\tau+\rho^2\sqrt n\int_0^t|\wh d|\,d\tau.$$
For $0\leq\rho \leq c$ (resp. $c\leq\rho\leq C_1$), the last term may be bounded according to \eqref{eq:diffulf2}
(resp. \eqref{eq:diffulf9}).  Regarding the term with $\wh\zeta_1,$ we notice that
$$
\wh\zeta_1=\wh\fj_1+\frac{\rho}{\cL\cL_s\cM}\biggl(\frac{\wh j_0}{\sqrt n}+\wh b\biggr),
$$
hence \eqref{eq:diffulf2} and the fact that $\cL^2\cL_s\approx1$ guarantee that for $0\leq\rho\leq c$
$$\begin{array}{lll}
\Frac{\rho^2}{\ep\sqrt n}
\Int_0^t|\wh\zeta_1|\,d\tau&\lesssim& \cL\biggl(\Frac{\cL\cM}\ep\Int_0^t|\wh\fj_1|\,d\tau
+\frac{\cL^2}\ep\int_0^t\rho^2|\wh j_0+\sqrt n\wh b|\,d\tau\biggr)\\[2ex]
&\lesssim& \cL(1+\frac{\cL^2}\ep)|(\wh b,\wh d,\ep\cL\cL_s\wh j_0,\wh j_1)(0)|.
\end{array}
$$
In the case $\cL\lesssim\ep^{1/2},$ it is obvious that Inequality \eqref{eq:diffulf9} implies
that the above inequality is also true in the range $c\leq\rho\leq C_1,$ which completes the 
proof of \eqref{eq:diffulf10}.

%%%%%%%%%%%%%%%%%%%%%%%%%%%%%%%%%%%%%%%%%%

\subsubsection{The case $\cL\approx\ep$ and $\cL_s\ep^2\gtrsim1$}

We saw that if \eqref{eq:stabdiffu} is fulfilled then  \eqref{eq:diffulf2} holds
true  on some small interval $[0,c].$  
So we have to fill in the gap between $c$ and $\sqrt{1+n^{-1}}.$
As the value $\sqrt{1+n^{-1}}$ does not play any particular role, we fix some $C_1>c$ and look for
estimates  if  $\rho\in[c,C_1].$ 
For simplicity, take  $\cL=\kappa\ep/n$ (with $\kappa>1$ owing to \eqref{eq:stabdiffu}).
\smallbreak
Setting  $\wh\zeta_1:=\wh j_1-\frac{1}{\sqrt n\,\cL\cM}\rho\wh j_0=\wh j_1-\frac{\sqrt n}{\kappa\ep\cM}\rho\wh j_0$ as before, we observe that
\begin{equation}\label{eq:diff-lf}\left\{
\begin{array}{l}
\d_t\wh b+\rho\wh d=0,\\[1ex]
\d_t\wh d+\rho^2\wh d-\rho\wh b=\frac\rho{n^{3/2}}\wh j_0+\frac{\kappa\ep\cM}{n^2}\wh\zeta_1,\\[1ex]
\d_t\wh j_0+\bigl(\frac\kappa n+\frac{\rho^2}{\kappa\ep^2\cM}\bigr)\wh j_0+\frac\rho{\sqrt n\ep}\wh\zeta_1=\frac\kappa{\sqrt n}\wh b,\\[1ex]
\d_t\wh\zeta_1+\bigl(\frac{\kappa\cM}n-\frac{\rho^2}{\kappa\ep^2\cM}\bigr)\wh\zeta_1=
\frac\rho{\sqrt n\,\ep\cM}
\bigl(1+\frac{\rho^2n}{\kappa^2\ep^2\cM}\bigr)\wh j_0-\frac\rho{\ep\cM}\wh b.
\end{array}\right.
\end{equation}
Let us focus on the subsystem corresponding to  the first three equations, namely
\begin{equation}\label{eq:sub}
\left\{\begin{array}{l}
\d_t\wh b+\rho\wh d=0,\\[1ex]
\d_t\wh d+\rho^2\wh d-\rho\wh b-\frac\rho{n^{3/2}}\wh j_0=\wh f,\\[1ex]
\d_t\wh j_0+\bigl(\frac\kappa n+\frac{\rho^2}{\kappa\ep^2\cM}\bigr)\wh j_0-\frac\kappa{\sqrt n}\wh b=\wh g.
\end{array}\right.
\end{equation}
If we have the stronger condition $\ep^2\cM\to\infty$ then we rewrite System \eqref{eq:sub} as follows
\begin{equation}\label{eq:diff-mf2} 
\frac d{dt}\begin{pmatrix}\wh b\\\wh d\\\wh j_0\end{pmatrix}
+\begin{pmatrix} 0&\rho&0\\
-\rho&\rho^2&-\frac\rho{n^{3/2}}\\
-\frac\kappa{\sqrt{n}}&0&\frac{\kappa}{n}
\end{pmatrix}\begin{pmatrix}\wh b\\\wh d\\\wh j_0
\end{pmatrix}=\begin{pmatrix}0\\\wh f\\-\frac{\rho^2}{\kappa\ep^2\cM}\wh j_0+\wh g
\end{pmatrix}\cdotp
\end{equation}
The eigenvalues of the matrix $M_\rho$ in the left-hand side are the roots of 
the polynomial $-X^3+a_1(\rho)X^2-a_2(\rho)X+a_3(\rho)$ with 
$$
a_1(\rho)=\frac\kappa{n}+\rho^2,\quad
a_2(\rho)=\Bigl(1+\frac\kappa n\Bigr)\rho^2,\quad
a_3(\rho)=\Bigl(1+\frac1n\Bigr)\frac{\rho^2\kappa}{n}\cdotp
$$
According to Routh-Hurwitz criterion, those roots have \emph{positive} real part if and only if 
$$
a_1(\rho)>0,\qquad\left|\begin{array}{cc} a_1(\rho)&1\\a_3(\rho)&a_2(\rho)\end{array}\right|>0
\quad\hbox{and}\quad
\left|\begin{array}{ccc} a_1(\rho)&1&0\\a_3(\rho)&a_2(\rho)&a_1(\rho)\\
0&0&a_3(\rho) \end{array}\right|>0.
$$
As $a_1(\rho)$ and $a_3(\rho)$ are positive, it suffices to check  the second  condition, that is
$$a_1(\rho)a_2(\rho)-a_3(\rho)=\Bigl(1+\frac\kappa n\Bigr)\rho^4+\frac{\rho^2\kappa}{n^2}(\kappa-1)>0,
$$
and this is indeed the case for all $\rho>0,$ as $\kappa>1.$ 
\medbreak
In particular,   all the eigenvalues of the matrix $M_\rho$ have positive real part
 if we assume $\rho$ to belong to the compact set $[c,C_1].$
Therefore (see \cite{DD}) there exist   two  positive constants $c_2$ and $C_2$ 
depending only on $c$ and $C_1,$  so that the matrix $M_\rho$ satisfies
$$
\bigl|e^{-tM_\rho}\bigr|\leq C_2 e^{-c_2t}\quad\hbox{for all }\ t\geq0\ \hbox{ and }\ \rho\in[c,C_1].
$$
By taking advantage of Duhamel's formula, we thus deduce that
$$
|(\wh b,\wh d,\wh j_0)(t)|+\int_0^t|(\wh a,\wh d,\wh j_0)|\,d\tau
\lesssim |(\wh b,\wh d,\wh j_0)(0)|+\int_0^t|(\wh f,\wh g)|\,d\tau
+\frac1{\ep^2\cM}\int_0^t|\wh j_0|\,d\tau.
$$
Of course, owing to the assumption $\ep^2\cM\to\infty,$ the last term of the r.h.s. may be 
absorbed by the l.h.s., for $\ep$ going to $0.$
So we get
\begin{equation}\label{eq:diffulf3}
|(\wh b,\wh d,\wh j_0)(t)|+\int_0^t|(\wh b,\wh d,\wh j_0)|\,d\tau\lesssim |(\wh b,\wh d,\wh j_0)(0)|+\int_0^t|(\wh f,\wh g)|\,d\tau.
\end{equation}
In the case where $\ep^2\cM$ does not go to $\infty$ then we have to proceed slightly differently.
If  we assume (for simplicity) that 
 $\ep^2\cM$ tends to some $m>0,$
then the matrix $M_\rho$ in \eqref{eq:diff-mf2} has to be changed in
$$
N_\rho=\begin{pmatrix} 0&\rho&0\\
-\rho&\rho^2&-\frac\rho{n^{3/2}}\\
-\frac\kappa{\sqrt n}&0&\frac{\kappa}{n}+\frac{\rho^2}{\kappa m}
\end{pmatrix}\cdotp
 $$
The above analysis based on Routh-Hurwitz theorem still holds 
 as  the additional term has `the good sign', and one may conclude, as before,
 that \eqref{eq:diffulf3} is satisfied for all $\rho\in[c,C_1].$
 \medbreak
 In every case $\ep^2\cM\gtrsim1,$ resuming to \eqref{eq:diff-lf}, 
 Inequality \eqref{eq:diffulf3} allows us to get for all $\rho\in[c,C_1]$
\begin{equation}\label{eq:diffulf4}
|(\wh b,\wh d,\wh j_0)(t)|+\int_0^t|(\wh b,\wh d,\wh j_0)|\,d\tau\lesssim |(\wh b,\wh d,\wh j_0)(0)|+\ep\cM\int_0^t|\wh\zeta_1|\,d\tau.
\end{equation}  
 Next, from the equation of $\wh\zeta_1,$ we readily get
 $$
 |\wh\zeta_1(t)|+\cM\int_0^t|\wh\zeta_1|\,d\tau\lesssim|\wh\zeta_1(0)|+(\ep\cM)^{-1}\int_0^t|(\wh b,\wh j_0)|\,d\tau.
 $$
 Hence, adding up to Inequality \eqref{eq:diffulf4}, we conclude that  \eqref{eq:diffulf2} is also true   for all $\rho\in[c,C_1].$ 
 This completes the proof of estimates in the low frequency regime $\rho\in[0,C_1].$

%%%%%%%%%%%%%%%%%%%%%%%%%%%%%%%%%%%%%

\subsection{Estimates for middle frequencies}
\subsubsection*{The case $\liminf \cM=+\infty$}
As in the previous paragraph, introduce   $\wh\zeta_1:=\wh j_1-\frac\rho{\sqrt n\,\cL\cM}\,\wh j_0.$
The system fulfilled by $(\wh b,\wh d,\wh j_0,\wh\zeta_1)$ reads
\begin{equation}\label{eq:diff-mf}\left\{
\begin{array}{l}
\d_t\wh b+\rho\wh d=0,\\[1ex]
\d_t\wh d+\rho^2\wh d-\rho\wh b=\frac\rho{n^{3/2}}\wh j_0+\frac{\cL\cM}n\wh\zeta_1,\\[1ex]
\d_t\wh j_0+\frac1\ep\bigl(\cL+\frac{\rho^2}{n\cL\cM}\bigr)\wh j_0+\frac\rho{\sqrt n\,\ep}\wh\zeta_1=\sqrt n\frac\cL\ep\wh b,\\[1ex]
\d_t\wh\zeta_1+\bigl(\frac{\cL\cM}\ep-\frac{\rho^2}{n\ep\cL\cM}\bigr)\wh\zeta_1=\frac{\rho}{\sqrt n\cL\cM\ep}
\bigl(\cL+\frac{\rho^2}{n\cL\cM}\bigr)\wh j_0-\frac\rho{\ep\cM}\wh b.
\end{array}\right.
\end{equation}
The subsystem corresponding to  the first three equations is
\begin{equation}\label{eq:sub1}
\left\{\begin{array}{l}
\d_t\wh b+\rho\wh d=0,\\[1ex]
\d_t\wh d+\rho^2\wh d-\rho\wh b-\frac\rho{n^{3/2}}\wh j_0=\wh f,\\[1ex]
\d_t\wh j_0+\frac\cL\ep\bigl(1+\frac{\rho^2}{n\cL^2\cM}\bigr)\wh j_0-\sqrt n\frac\cL\ep\wh b=\wh g,
\end{array}\right.
\end{equation}
with \begin{equation}\label{eq:diffumf00}\wh f=\frac{\cL\cM}n\,\wh\zeta_1\quad\hbox{and}\quad
\wh g=-\frac\rho{\sqrt n\,\ep}\,\wh\zeta_1.\end{equation}
Assume that  $(\wh f,\wh g)\equiv(0,0)$ for a while and set
 \begin{equation}\label{def:U}
 \cU_\rho^2:=2|(\wh b,\wh d)|^2-2\rho\Re(\wh b\,\overline{\wh d})+|\rho\wh b|^2.
 \end{equation}
On one hand, we have
\begin{equation}\label{eq:diff-mf0}
\frac12\frac d{dt}\cU_\rho^2+\rho^2|\wh b|^2+\rho^2|\wh d|^2=\frac\rho{n^{3/2}}\Re\bigl((2\wh d-\rho\wh b)\,\overline{\wh j_0}\bigr),
\end{equation}
and on the other hand, 
$$
\frac12\frac d{dt}|\wh j_0|^2+\frac\cL\ep\biggl(1+\frac{\rho^2}{n\cL^2\cM}\biggr)|\wh j_0|^2=\sqrt n\frac\cL\ep\Re(\wh b\,\overline{\wh j_0}).
$$
Therefore
$$
\frac12\frac d{dt}\biggl(\cU_\rho^2+\frac\ep{n^2\cL}|\rho\wh j_0|^2\biggr)
+\rho^2|(\wh b,\wh d)|^2+\frac{\rho^2}{n^2}\biggl(1+\frac{\rho^2}{n\cL^2\cM}\biggr)|\wh j_0|^2
=2\frac{\rho}{n^{3/2}}\Re(\wh d\,\overline{\wh j_0}).
$$
Now, by using the fact that
$$
2\frac{\rho}{n^{3/2}}\Re(\wh d\:\overline{\wh j_0})\leq\frac1{An}|\wh d|^2+\frac{A\rho^2}{n^2}|\wh j_0|^2,
$$
and by taking $A=3/4,$ we conclude that for $\rho^2\geq\frac{16n}{3(4n^2-1)},$  we have
$$
\frac d{dt}\bigl(\cU_\rho^2+\frac\ep{n^2\cL}|\rho\wh j_0|^2\bigr)
+\frac{\rho^2}{2n^2}|(\wh b,\wh d,\wh j_0)|^2\leq 0,
$$
whence, because  $\frac{16n}{3(4n^2-1)}\leq 1+\frac1n,$ we get
 for some universal positive constants $c_0$ and $C$
$$
|(\rho\wh b,\wh d)(t)|+\sqrt{\frac\ep\cL}|\rho\wh j_0(t)|\leq C e^{-c_0t}
\biggl(|(\rho\wh b,\wh d)(0)|+\sqrt{\frac\ep\cL}|\rho\wh j_0(0)|\biggr)\quad\hbox{for}\quad \rho\geq\sqrt{1+n^{-1}}\cdotp
$$
Resuming to the equation fulfilled by $\wh j_0$ in \eqref{eq:sub1}, 
the above inequality implies (still assuming that $\wh f=\wh g=0$) that
$$
\frac\cL\ep\biggl(1+\frac{\rho^2}{n\cL^2\cM}\biggr)\int_0^t|\rho\wh j_0|\,d\tau
\leq|\rho\wh j_0(0)|+\frac{\sqrt n\cL}{\ep}\int_0^t|\rho b|\,d\tau
\lesssim \sqrt{\frac\cL\ep} |\rho\wh j_0(0)|+\frac\cL{\ep}|(\wh d,\rho\wh b)(0)|,
$$
then plugging this inequality in the equation for $\wh d,$ we get in addition
$$
\rho^2\int_0^t|\wh d|\,d\tau\lesssim |\wh d(0)|+|\rho\wh b(0)|+\sqrt{\frac{\ep}\cL}|\rho\wh j_0(0)|.
$$
Repeating the above computations in the case of general source terms $\wh f$ and $\wh g,$ 
we conclude that the solution $(b,d,j_0)$ to \eqref{eq:sub1} satisfies
for all $\rho\geq\sqrt{1+n^{-1}},$ assuming only that $\cL\gtrsim\ep$
 $$
\displaylines{
|(\rho\wh b,\wh d)(t)|+\sqrt{\frac\ep\cL}|\rho\wh j_0(t)|
+\int_0^t|(\rho\wh b,\rho^2\wh d)|\,d\tau+\frac\cL\ep\biggl(1+\frac{\rho^2}{n\cL^2\cM}\biggr)
\int_0^t|\rho\wh j_0|\,d\tau\hfill\cr\hfill\leq C
\biggl(|(\rho\wh b,\wh d)(0)|+\sqrt{\frac\ep\cL}|\rho\wh j_0(0)|
+\int_0^t\Bigl(|\wh f|+\sqrt{\frac\ep\cL}|\rho\wh g|\Bigr)\,d\tau
\biggr)\cdotp}$$
Resuming to the value of $\wh f$ and $\wh g$ in \eqref{eq:diffumf00}, we thus get
for  $\rho\geq\sqrt{1+n^{-1}}$ and  $\cL\gtrsim\ep$
\begin{multline}\label{eq:diffumf}
|(\rho\wh b,\wh d)(t)|+\sqrt{\frac\ep\cL}|\rho\wh j_0(t)|
+\int_0^t|(\rho\wh b,\rho^2\wh d)|\,d\tau+\frac\cL\ep\biggl(1+\frac{\rho^2}{n\cL^2\cM}\biggr)
\int_0^t|\rho\wh j_0|\,d\tau\\\leq C
\biggl(|(\rho\wh b,\wh d)(0)|+\sqrt{\frac\ep\cL}|\rho\wh j_0(0)|
+\biggl(\cL\cM+\frac{\rho^2}{\sqrt{\cL\ep}}\biggr)
\int_0^t|\wh\zeta_1|\,d\tau\biggr)\cdotp
\end{multline}
Next, it is clear that  the equation for $\wh\zeta_1$ implies that whenever
$\rho\leq\sqrt{\frac n2}\,\cL\cM$
\begin{equation}\label{eq:diff-mf5}
|\wh\zeta_1(t)|+\frac{\cL\cM}{2\ep}\int_0^t|\wh\zeta_1|\,d\tau\leq|\wh\zeta_1(0)|
+\frac1{\sqrt n\ep\cM}\biggl(1+\frac{\rho^2}{n\cL^2\cM}\biggr)\int_0^t|\rho\wh j_0|\,d\tau
+\frac{1}{\ep\cM}\int_0^t|\rho\wh b|\,d\tau.
\end{equation}
Hence, we get if $\cM$ is large enough
and $\rho^2\ll\cL^{3/2}\cM^2\ep^{1/2}$
\begin{equation}\label{eq:diff-mf6}
|\wh\zeta_1(t)|+\frac{\cL\cM}{\ep}\int_0^t|\wh\zeta_1|\,d\tau
\leq C\biggl(|\wh\zeta_1(0)|+\frac{1}{\ep\cM}\Bigl(|(\rho\wh b,\wh d)(0)|+\sqrt{\frac\ep\cL}\,|\rho\wh j_0(0)|\Bigr)\biggr)\cdotp
\end{equation}
Then plugging that inequality in \eqref{eq:diffumf} implies that
\begin{multline}\label{eq:diff-mf7}
|(\rho\wh b,\wh d)(t)|+\sqrt{\frac\ep\cL}|\rho\wh j_0(t)|+\rho^2\!\int_0^t\!|\wh d|\,d\tau
+\int_0^t|\rho\wh b|\,d\tau+\frac\cL\ep\biggl(1+\frac{\rho^2}{n\cL^2\cM}\biggr)
\int_0^t|\rho\wh j_0|\,d\tau\\\leq
C\biggl(|(\rho\wh b,\wh d)(0)|+\sqrt{\frac\ep\cL}|\rho\wh j_0(0)|
+\Bigl(\ep+\frac{\ep^{1/2}\rho^2}{\cM\cL^{3/2}}\Bigr)|\wh\zeta_1(0)|\biggr),
\end{multline}
whenever $1+1/n\leq\rho^2\ll\cL^{3/2}\cM^2\ep^{1/2}.$
\medbreak
Here is another method that gives decay estimates
 in the range  $\cL\sqrt\cM\ll\rho\ll\cL\cM$ if  $\cM$ is large enough. 
{}From the first two equations of \eqref{eq:diff-mf}, we have
\begin{equation}\label{eq:diff-mf1b}
|(\rho\wh b,\wh d)(t)|+\rho^2\int_0^t|\wh d|\,d\tau+\int_0^t|\rho\wh b|\,d\tau
\lesssim |(\rho\wh b,\wh d)(0)|\\+\frac\rho{n^{3/2}}\int_0^t|\wh j_0|\,d\tau+\frac{\cL\cM}n\int_0^t|\wh\zeta_1|\,d\tau.
\end{equation}
The equations for $\wh j_0$ and $\wh\zeta_1$ give, if $\rho\leq\sqrt{\frac n2}\,\cL\cM$
\begin{equation}\label{eq:diff-mf2b}
|\wh j_0(t)|+\frac1\ep\Bigl(\cL+\frac{\rho^2}{n\cL\cM}\Bigr)\int_0^t|\wh j_0|\,d\tau\leq|\wh j_0(0)|
+\frac\rho{\sqrt n\,\ep}\int_0^t|\wh\zeta_1|\,d\tau+\frac{\cL\sqrt n}\ep\int_0^t|\wh b|\,d\tau,
\end{equation}
\begin{equation}\label{eq:diff-mf3b}
 |\wh\zeta_1(t)|\!+\!\frac{\cL\cM}{2\ep}\int_0^t|\wh\zeta_1|\,d\tau\leq|\wh\zeta_1(0)|+\frac{\rho}{\sqrt n\,\cL\cM\ep}
\biggl(\cL\!+\!\frac{\rho^2}{n\cL\cM}\biggr)\int_0^t|\wh j_0|\,d\tau\!+\!\frac\rho{\ep\cM}\int_0^t|\wh b|\,d\tau.
\end{equation}
Plugging \eqref{eq:diff-mf2b} in \eqref{eq:diff-mf3b}, we discover if  $\rho\ll\cL\cM$ that 
\begin{equation}\label{eq:diff-mf2c}
|\wh\zeta_1(t)|+\frac{\cL\cM}{4\ep}\int_0^t|\wh\zeta_1|\,d\tau
\leq|\wh\zeta_1(0)|+\frac{\rho}{\cL\cM}|\wh j_0(0)|+\frac{1}{\ep\cM}\int_0^t\rho|\wh b|\,d\tau.
\end{equation}
Inserting that  inequality in \eqref{eq:diff-mf1b}, 
we conclude  that the last term of \eqref{eq:diff-mf1b} may be absorbed by the
l.h.s. if $\cM$ is large enough. Now, Inequality \eqref{eq:diff-mf2b} guarantees that
$$\displaylines{
\frac\rho{n^{3/2}}\int_0^t|\wh j_0|\,d\tau\leq \frac1{\sqrt n}\biggl(\frac{\rho\ep\cL\cM}{n\cL^2\cM+\rho^2}\biggr)|\wh j_0(0)|
\hfill\cr\hfill+ \biggl(\frac{\rho^2\cL\cM}{n^2\cL^2\cM+\rho^2n}\biggr)\int_0^t|\wh\zeta_1|\,d\tau
+ \biggl(\frac{\cL^2\cM}{n\cL^2\cM+\rho^2}\biggr)\int_0^t\rho|\wh b|\,d\tau.}
$$
Again, resuming to \eqref{eq:diff-mf1b}, we see that the second term in the r.h.s. may be absorbed by the l.h.s.
This is also the case of the last one if $\rho^2\gg \cL^2\cM.$
If all those conditions are fulfilled then we end up if $\cL\sqrt\cM\ll\rho\leq\sqrt\frac n2\,\cL\cM$ with 
\begin{multline}\label{eq:diff-mf4b}
|(\rho\wh b,\wh d)(t)|+\rho^2\int_0^t|\wh d|\,d\tau+\rho\int_0^t|\wh b|\,d\tau
+\rho \int_0^t|\wh j_0|\,d\tau\\+\cL\cM\int_0^t|\wh\zeta_1|\,d\tau
\lesssim |(\rho\wh b,\wh d)(0)|+\frac{\ep\cL\cM}\rho|\wh j_0(0)|
+\ep|\wh\zeta_1(0)|.
\end{multline}
Resuming to \eqref{eq:diff-mf2b} and \eqref{eq:diff-mf2c}, we thus easily deduce 
first that
$$
|\wh \zeta_1(t)|+\frac{\cL\cM}\ep\int_0^t|\wh\zeta_1|\,d\tau
\lesssim |\wh\zeta_1(0)|+|\wh j_0(0)|+\frac1{\ep\cM} |(\rho\wh b,\wh d)(0)|,
$$
and next that
$$
|\wh j_0(t)|+\frac1\ep\Bigl(\cL+\frac{\rho^2}{\cL\cM}\Bigr)\int_0^t|\wh j_0|\,d\tau
\lesssim |\wh j_0(0)|+\frac\rho{\cL\cM}|\wh\zeta_1(0)|+\textstyle\frac\cL{\ep\rho}|(\rho\wh b,\wh d)(0)|.
$$
Our estimates and the definition of $\zeta_1$ allow us to change $\zeta_1$ to $j_1$ in \eqref{eq:diff-mf4b}. 
So finally, in the case  $\ep\cM\gtrsim1$   we get for some large enough $C_1$ and small enough $c$ independent of $\ep$
\begin{multline}\label{eq:diff-hf7}
|(\rho\wh b,\wh d,(1\!+\!\textstyle\frac{\ep\cL\cM}\rho)\wh j_0,\wh j_1)(t)|
+\Int_0^t|\rho\wh b|\,d\tau 
+\rho^2\int_0^t|\wh d|\,d\tau\\+\textstyle\frac\rho{\ep\cL\cM}(1\!+\!\textstyle\frac{\ep\cL\cM}\rho)
\Int_0^t|\rho\wh j_0|\,d\tau+
 \frac{\cL\cM}\ep\int_0^t|\wh\zeta_1|\,d\tau
\lesssim  |(\rho\wh b,\wh d,(1\!+\!\textstyle\frac{\ep\cL\cM}\rho)\wh j_0,\wh j_1)(0)|,
\end{multline}
 whenever  $C_1\cL\sqrt\cM\leq \rho\leq c\cL\cM.$

%%%%%%%%%%%%%%%%%%%%%%%%%%%%%%%%%%%%%%%%%%

\subsubsection*{The case  $\ep^2\cM\leq 1/2$ and $\cL\gg\ep$} 
 Let $\zeta_0:=j_0-\sqrt n\, b.$ We start from 
$$\left\{
\begin{array}{l}
\d_t\wh b+\rho\wh d=0,\\[1ex]
\d_t\wh d+\rho^2\wh d-\rho\wh b=\frac{\cL\cM}n\wh j_1,\\[1ex]
\d_t\wh\zeta_0+\frac\cL\ep\wh\zeta_0+\frac\rho{\sqrt n\,\ep}\wh j_1-\sqrt n\,\rho\wh d=0,\\[1ex]
\d_t\wh j_1+\frac{\cL\cM}\ep\wh j_1-\frac\rho{\sqrt n\,\ep}\wh\zeta_0-\frac\rho\ep\wh b=0.
\end{array}\right.
$$
In order to show the exponential decay, we set
 \begin{equation}\label{eq:Ur}
 \cU_\rho^2:=2|\wh b|^2+2|\wh d|^2+|\rho\wh b|^2-2\rho\Re(\wh b\overline{\wh d})\quad\hbox{and}\quad
\cJ_\rho^2:=|\wh\zeta_0|^2+|\wh j_1|^2.\end{equation}
We easily get for all $K\geq0,$
\begin{multline}\label{eq:diff-mf8}
\frac12\frac d{dt}\Bigl(\cU_\rho^2+K\cJ_\rho^2\Bigr)
+\rho^2|(\wh b,\wh d)|^2+K\frac\cL\ep\biggl(|\wh\zeta_0|^2
+{\cM}|\wh j_1|^2\biggr)\\=
2\frac{\cL\cM}n\Re(\wh j_1\,\overline{\wh d})+\rho\biggl(\frac K\ep-\frac{\cL\cM}n\biggr)\Re(\wh j_1\,\overline{\wh b})
+K\sqrt n\Re (\rho\wh d\,\overline{\wh\zeta_0}).\end{multline}
It is thus natural to take $K=\frac\ep n\cL\cM$ to cancel out the second term of the r.h.s. 
For the first and the last terms, we write that
$$\begin{array}{l}
K\sqrt n\Re\bigl(\rho\wh d\,\overline{\wh\zeta_0}\bigr)\leq \frac14\rho^2|\wh d|^2
+nK^2|\wh j_0-\sqrt n\,\wh b|^2,\\[1ex]
\frac{2\cL\cM}n\Re(\wh j_1\,\overline{\wh d})=\frac{2K}{\ep}\Re(\wh j_1\,\overline{\wh d})\leq\frac23\rho^2|\wh d|^2
+\frac{3K^2}{2\ep^2\rho^2}|\wh j_1|^2.\end{array}$$
Note that the last terms above may be absorbed by the l.h.s. of \eqref{eq:diff-mf8} if, say  
$$
nK\leq\frac\cL{2\ep}\quad\hbox{and}\quad \frac{2K}{\ep}\leq\cL\cM\rho^2.
$$
Given the value of $K,$ the first condition is equivalent to $\ep^2\cM\leq1/2,$ whereas the second one
means that  $\rho^2n\geq 2.$  Under this latter condition, we thus end up with 
\begin{equation}\label{eq:diff-hf1}
\frac d{dt}\Bigl(\cU_\rho^2+\frac{\ep\cL\cM}n\cJ_\rho^2\Bigr)
+\frac16\rho^2|(\wh b,\wh d)|^2+\frac{\cL^2\cM}n\biggl(|\wh\zeta_0|^2
+\frac16{\cM}|\wh j_1|^2\biggr)\leq0,
\end{equation}
which implies, according to  \eqref{eq:stabdiffu}, the following exponential decay estimate for some  small enough $\kappa>0$
\begin{equation}\label{eq:diff-hf2}
\cU_\rho^2(t)+\ep\cL\cM\cJ_\rho^2(t)
\leq e^{-\kappa t} \Bigl(\cU_\rho^2(0)+\frac{\ep\cL\cM}n\cJ_\rho^2(0)\Bigr)\quad\hbox{if }\ \rho\geq\sqrt{2/n}.
\end{equation}
%We note that \eqref{eq:diff-hf1} readily gives 
%$$\frac{\cL^2\cM^2}{6n}\int_0^t|\wh j_1|^2\,d\tau\leq\cU_\rho^2(0)+\frac{\ep\cL\cM}n\cJ_\rho^2(0)\cdotp$$
To exhibit the parabolic decay for $d,$  we  introduce
$\wh\zeta_1:=\wh j_1-(\sqrt n\cL\cM)^{-1}\rho\wh j_0,$ and get
$$\left\{
\begin{array}{l}
\d_t\wh b+\rho\wh d=0,\\[1ex]
\d_t\wh d+\rho^2\wh d-\rho\bigl(1+\frac1n\bigr)\wh b=\frac{\cL\cM}n\wh\zeta_1+\frac\rho{n^{3/2}}\wh\zeta_0,\\[1ex]
\d_t\wh\zeta_0+\frac1\ep\bigl(\cL+\frac{\rho^2}{n\cL\cM}\bigr)\wh\zeta_0+\frac\rho{\sqrt n\,\ep}\wh\zeta_1=\sqrt n\,\rho\wh d-\frac{\rho^2}{\sqrt n\,\ep\cL\cM}\wh b,\\[1ex]
\d_t\wh\zeta_1+\bigl(\frac{\cL\cM}\ep-\frac{\rho^2}{n\ep\cL\cM}\bigr)\wh\zeta_1=\frac{\rho}{\sqrt n\,\cL\cM\ep}
\bigl(\cL+\frac{\rho^2}{n\cL\cM}\bigr)\wh\zeta_0+\frac{\rho^3}{n\ep\cL^2\cM^2}\wh b.
\end{array}\right.
$$
We thus have for $\sqrt{2/n}\leq \rho\leq \sqrt{n/2}\,\cL\cM$
$$
\displaylines{
\int_0^t|\rho\wh b|\,d\tau+\rho^2\int_0^t|\wh d|\,d\tau\lesssim|(\rho\wh b,\wh d)(0)|+\frac{\cL\cM}n\int_0^t|\wh\zeta_1|\,d\tau
+\frac\rho n\int_0^t|\wh\zeta_0|\,d\tau,\cr
\frac1\ep\biggl(\cL+\frac{\rho^2}{n\cL\cM}\biggr)\int_0^t|\wh\zeta_0|\,d\tau\leq|\wh\zeta_0(0)|+\frac\rho{\sqrt n\,\ep}\int_0^t|\wh\zeta_1|\,d\tau
+\sqrt n\,\rho\int_0^t|\wh d|\,d\tau+\frac{\rho}{\sqrt n\,\ep\cL\cM}\int_0^t|\rho\wh b|\,d\tau,\cr
\frac{\cL\cM}{2\ep}\int_0^t|\wh\zeta_1|\,d\tau
\leq|\wh\zeta_1(0)|+\frac{\rho}{\sqrt n\,\ep\cL\cM}\biggl(\cL+\frac{\rho^2}{n\cL\cM}\biggr)\int_0^t|\wh\zeta_0|\,d\tau
+\frac{\rho^2}{n\ep\cL^2\cM^2}\int_0^t|\rho\wh b|\,d\tau.}
$$
Combining the inequalities for $\wh\zeta_0$ and $\wh\zeta_1,$ we easily get if  $\rho\leq c\cL\cM$ with $c$ small enough
$$\displaylines{
\cL\cM\int_0^t|\wh\zeta_1|\,d\tau\lesssim\ep|\wh\zeta_1(0)|
+\frac\rho{\cL\cM}\biggl(\ep|\wh\zeta_0(0)|+\rho\ep\int_0^t|\wh d|\,d\tau
+\frac{\rho}{\cL\cM}\int_0^t|\rho\wh b|\,d\tau\biggr),\cr
\biggl(\cL+\frac{\rho^2}{\cL\cM}\biggr)\int_0^t|\wh\zeta_0|\,d\tau\lesssim\ep|\wh\zeta_0(0)|
+\frac{\rho\ep}{\cL\cM}|\wh\zeta_1(0)|+\rho\ep\int_0^t|\wh d|\,d\tau
+\frac{\rho}{\cL\cM}\int_0^t|\rho\wh b|\,d\tau.}
$$
Now, the exponential decay pointed out in \eqref{eq:diff-hf2} allows to bound the last terms above, and we get
\begin{eqnarray}\label{eq:diff-mf10}
&\cL\cM\Int_0^t|\wh\zeta_1|\,d\tau\lesssim  \cU_\rho(0)+\sqrt{\ep\cL\cM}\cJ_\rho(0)+\frac{\ep}{\cL\cM}\int_0^t|\rho^2\wh d|\,d\tau,\\\label{eq:diff-mf11}
&\quad\biggl(\cL\!+\!\Frac{\rho^2}{\cL\cM}\biggr)\int_0^t|\wh\zeta_0|\,d\tau\lesssim\ep|(\wh\zeta_0,\wh\zeta_1)(0)|
+\frac{\rho}{\cL\cM}\bigl( \cU_\rho(0)+\sqrt{\ep\cL\cM}\cJ_\rho(0)\bigr)
+\rho\ep\int_0^t|\wh d|\,d\tau,
\end{eqnarray}
whereas using $\cU_\rho^2$ allows to get directly
$$
\rho^2\int_0^t|\wh d|\,d\tau\lesssim|(\rho\wh b,\wh d)(0)|+\cL\cM\int_0^t|\wh j_1|\,d\tau.
$$
Using the definition of $\wh j_1$ and, again, Inequality \eqref{eq:diff-hf2}, we may 
replace $\wh j_1$ with $\wh\zeta_1$ as follows
\begin{equation}\label{eq:diff-mf12}
\rho^2\int_0^t|\wh d|\,d\tau\lesssim \cU_\rho(0)+\sqrt{\ep\cL\cM}\cJ_\rho(0) +\cL\cM\int_0^t|\wh\zeta_1|\,d\tau
+\int_0^t|\rho\wh\zeta_0|\,d\tau.
\end{equation}
Then plugging \eqref{eq:diff-mf10} and \eqref{eq:diff-mf11} in \eqref{eq:diff-mf12} and observing that $\ep\ll\cL\cM,$ we get
$$\displaylines{
\rho^2\int_0^t|\wh d|\,d\tau\lesssim \cU_\rho(0)+\sqrt{\ep\cL\cM}\cJ_\rho(0)\hfill\cr\hfill
+\frac{\rho\cL\cM}{\cL^2\cM+\rho^2}\biggl(\ep|(\wh\zeta_0,\wh\zeta_1)(0)|+\frac\rho{\cL\cM}\bigl( \cU_\rho(0)+\sqrt{\ep\cL\cM}\cJ_\rho(0)\bigr)
+\rho\ep\int_0^t|\wh d|\,d\tau\biggr)\cdotp}
$$
Because we assumed that $\cL\gg\ep,$ the last term may be absorbed by the l.h.s. Using in addition that $\ep^2\cM\lesssim1,$
we end up with 
$$
\rho^2\int_0^t|\wh d|\,d\tau\lesssim \cU_\rho(0)+\sqrt{\ep\cL\cM}\cJ_\rho(0).
$$
Then resuming to \eqref{eq:diff-mf10}, \eqref{eq:diff-mf11}, we obtain
\begin{equation}\label{eq:diff-hf4}
\rho\int_0^t|\wh\zeta_0|\,d\tau+\cL\cM\int_0^t|\wh\zeta_1|\,d\tau\lesssim  \cU_\rho(0).
\end{equation}
Obviously, this inequality implies that
\begin{equation}\label{eq:diff-hf3}
\cL\cM\int_0^t|\wh j_1|\,d\tau\lesssim  \cU_\rho(0)
+\sqrt{\ep\cL\cM}\,\cJ_\rho(0)+\ep|(\wh\zeta_0,\wh\zeta_1)(0)|.
\end{equation}
Of course, we get the same inequality if replacing $\zeta_0$ and $\zeta_1$ with $j_0$ and $j_1.$
So one can conclude that  for  $\sqrt{2/n}\leq\rho\leq c\cL\cM,$ we have \eqref{eq:diff-hf2} and
\begin{equation}\label{eq:diff-hf6}
\rho^2\int_0^t|\wh d|\,d\tau+\cL\cM\int_0^t|\wh j_1|\,d\tau+\rho\int_0^t|\wh j_0|\,d\tau\lesssim |(\rho\wh b,\wh d)(0)|+\sqrt{\ep\cL\cM}|(\wh j_0,\wh j_1)(0)|.
\end{equation}

%%%%%%%%%%%%%%%%%%%%%%%%%%%%%%%%%%%%%%%%%%

\subsection{High frequencies} 

We eventually come to the proof of  decay estimates  for $\rho\geq c\cL\cM,$ where $c$ is some given positive constant.
We shall use that fact  that  the systems satisfied by $(b,d)$ and by $(j_0,j_1),$ respectively, 
tend to be uncoupled for $\rho\to+\infty.$ 
As regards $(b,d),$ the classical approach for the barotropic Navier-Stokes equation, based on the study 
of 
$$\cU_\rho^2:=2|(\wh b,\wh d)|^2-2\rho\Re(\wh b\,\overline{\wh d})+|\rho\wh b|^2,$$ 
  guarantees, if $\rho\geq c,$  that
\begin{equation}\label{eq:hfdiff1}
|(\rho\wh b,\wh d)(t)|+\rho^2\int_0^t|\wh d|\,d\tau+\int_0^t\rho|\wh b|\,d\tau
\lesssim |(\rho\wh b,\wh d)(0)|+\frac{\cL\cM}n \int_0^t|\wh j_1|\,d\tau.
\end{equation}
Next, from the system fulfilled by  $(\wh j_0,\wh j_1),$ we get
$$
\displaylines{\frac12\frac d{dt}\Bigl(|\wh j_0|^2+|\wh j_1|^2\Bigr)+\frac\cL\ep|\wh j_0|^2+\frac{\cL{\cM}}{\ep}|\wh j_1|^2=\sqrt n\frac\cL\ep\Re(\wh b\,\overline{\wh j_0}),\cr
\frac d{dt}\Re(\wh j_0\,\overline{\wh j_1})+\frac\cL\ep(1+{\cM})\Re(\wh j_0\,\overline{\wh j_1})+\frac{\rho}{\ep\sqrt n}|\wh j_1|^2-\frac{\rho}{\ep\sqrt n}|\wh j_0|^2=\sqrt n\frac\cL\ep\Re(\wh b\,\overline{\wh j_1}).}
$$
Therefore, for any $\kappa>0$ 
$$\displaylines{
\frac12\frac d{dt}\biggl(|\wh j_0|^2+|\wh j_1|^2-
\frac{2\kappa\cL{\cM}}{\rho}\Re(\wh j_0\,\overline{\wh j_1})\biggr)
+\Bigl(1+\frac\kappa{\sqrt n}{\cM}\Bigr)\frac\cL\ep|\wh j_0|^2+\frac{\cL{\cM}}{\ep}\Bigl(1-\frac\kappa{\sqrt n}\Bigr)|\wh j_1|^2\hfill\cr\hfill=
\frac{\kappa\cL^2{\cM}}{\rho\ep}(1+{\cM})\Re(\wh j_0\,\overline{\wh j_1})
+\sqrt n\,\frac{\cL}\ep\Re\Bigl(\wh b\,\Bigl(\overline{\wh j_0}
-\frac{\kappa\cL{\cM}}{\rho}\,\overline{\wh j_1}\Bigr)\Bigr)\cdotp}
$$
For $\rho\geq c\cL{\cM},$ it is clear that our choosing $\kappa$ small enough implies that
\begin{itemize}
\item  the first term of the r.h.s. may be absorbed by the 
second and third ones of the l.h.s.,\smallbreak
\item we have $|\wh j_0|^2+|\wh j_1|^2-\frac{2\kappa\cL{\cM}}{\rho}\Re(\wh j_0\,\overline{\wh j_1})\approx |\wh j_0|^2+|\wh j_1|^2,$\smallbreak
\item we have $\bigl(1+\frac \kappa{\sqrt n}{\cM}\bigr)\frac\cL\ep|\wh j_0|^2+\bigl(1-\frac{\kappa}{\sqrt n}\bigr)\frac{\cL{\cM}}{\ep}|\wh j_1|^2\geq \frac\kappa2\frac{\cL{\cM}}{\sqrt n\,\ep}(|\wh j_0|^2+|\wh j_1|^2).$
\end{itemize}\medbreak\noindent
Therefore, we end up with the following inequality
\begin{equation}\label{eq:hfdiff2}
\ep|(\wh j_0,\wh j_1)(t)|+\cL\cM\int_0^t|(\wh j_0,\wh j_1)|\,d\tau
\lesssim \ep|(\wh j_0,\wh j_1)(0)|+\cL\int_0^t|\wh b|\,d\tau.
\end{equation}
Combining with  \eqref{eq:hfdiff1},  we conclude that if 
\begin{equation}\label{eq:hfdiff}
\rho\geq c\max(1,\cL{\cM})\quad\hbox{and}\quad \rho\geq C\cL,
\end{equation}
 for a large enough constant $C$ then
\begin{eqnarray}\label{eq:hfdiff3}
|(\rho\wh b,\wh d)(t)|+\rho^2\int_0^t\!|\wh d|\,d\tau+\int_0^t\!\rho|\wh b|\,d\tau
\lesssim |(\rho\wh b,\wh d,\ep\wh j_0,\ep\wh j_1)(0)|,\\\label{eq:hfdiff4}
|(\wh j_0,\wh j_1)(t)|+\frac{\cL\cM}\ep\int_0^t\!|(\wh j_0,\wh j_1)|\,d\tau
\lesssim \frac\cL{\rho\ep}|(\rho\wh b,\wh d)(0)|+|(\wh j_0,\wh j_1)(0)|,
\end{eqnarray}
whence
\begin{equation}\label{eq:hfdiff3a}
|(\rho\wh b,\wh d,\wh j_0,\wh j_1)(t)|+\rho^2\int_0^t\!|\wh d|\,d\tau+\int_0^t\!\rho|\wh b|\,d\tau
+\frac{\cL\cM}\ep\int_0^t\!|(\wh j_0,\wh j_1)|\,d\tau
\lesssim |(\rho\wh b,\wh d,\wh j_0,\wh j_1)(0)|.
\end{equation}

The only case where the condition $\rho\geq C\cL$ may be stronger than $\rho\geq c\cL\cM$ is when  $\cM$ is bounded. {}From our study for small $\rho$'s, we must  assume that $\cL\cM\gtrsim\ep^{-1},$
and thus  $c\cL\cM$ is still much larger than $\sqrt{2/n}.$ 
Therefore, one may take advantage of \eqref{eq:diff-hf2}  
 to bound the r.h.s. of \eqref{eq:hfdiff2}, and combining with \eqref{eq:hfdiff1}, we get for $\rho\approx\cL$
\begin{multline}\label{eq:diff4}
|(\rho\wh b,\wh d,\sqrt{\ep\cL}\wh j_0,\sqrt{\ep\cL}\wh j_1)(t)|+\rho^2\!\int_0^t\!|\wh d|\,d\tau+\!\int_0^t\!\rho|\wh b|\,d\tau
\\+\cL\int_0^t\!|(\wh j_0,\wh j_1)|\,d\tau
\lesssim |(\rho\wh b,\wh d,\sqrt{\ep\cL}\wh j_0,\sqrt{\ep\cL}\wh j_1)(0)|.
\end{multline}

%%%%%%%%%%%%%%%%%%%%%%%%%%%%%%%%%%%%%%%%%%%%%%%

\section{The non-equilibrium diffusion regime}

This section is devoted to the study of the so-called   \emph{non-equilibrium diffusion asymptotics}.
Assuming that for some $\kappa>1$ and $m>0,$ we have
\begin{equation}\label{eq:diffasymp}
\frac\cL\ep\to\frac{\kappa}{n\nu}\quad\hbox{and}\quad\cL^2\cL_s\to\frac m{\nu^2}, 
\end{equation}
we want to prove the convergence of the solutions of \eqref{eq:NSrad} to those 
of \eqref{eq:noneq1} or \eqref{eq:noneq2} if $m<+\infty$ or $m=+\infty,$ respectively, 
when $\ep$ goes to $0.$
\medbreak
The first subsection concerns the proof  of global existence with `uniform' estimates 
for the radiative Navier-Stokes equations \eqref{eq:NSrad}  in the asymptotic \eqref{eq:diffasymp}, 
in the small data case with critical regularity. 
Next,  still for small critical data,  we establish 
 the global  existence for the limit systems \eqref{eq:noneq1} and \eqref{eq:noneq2}. 
In the last part of the present section, we combine the uniform estimates with compactness arguments 
in order to justify the convergence of \eqref{eq:NSrad} to 
\eqref{eq:noneq1} or \eqref{eq:noneq2}.

%%%%%%%%%%%%%%%%%%%%%%%%%%%%%%%%%%%%%%%%%%%%%%%%%%%%%%

\subsection{Global existence and uniform estimates for \eqref{eq:NSrad} with  $\cL\approx\ep$ and $\cL^2\cL_s\gtrsim1$}

In order to get a global-in-time existence statement for \eqref{eq:NSrad} in the non-equilibrium diffusion regime, 
we first put together the estimates that we obtained in the previous section, in the case $\cL\approx\ep$ and $\cL^2\cL_s\gtrsim1$.
Even though localizing the linearized equations by means of Littlewood-Paley operators allows 
to get essentially optimal estimates for the linearized equations of \eqref{eq:NSrad}, it
is not enough for our purpose, owing to the convection term $\vu\cdot\nabla b$ that may cause a loss of one derivative.  
The difficulty may be overcome by paralinearizing  the whole system, as explained below. 
After that, it is easy to prove global in time estimates for the solutions to \eqref{eq:NSrad} 
just by combining the estimates for the paralinearized system, and standard product laws in 
Besov spaces to handle the other nonlinear terms.

\subsubsection{Linear estimates}

Performing the change of variables \eqref{eq:nu} reduces the study to that of  the linear system \eqref{eq:diff}. 

\subsubsection*{Low frequencies estimates} 

 Using \eqref{eq:diffulf2}, the comment that follows, \eqref{eq:diffulf4} and the fact that
$|(\wh b,\wh d,\wh j_0,\wh j_1)|\approx|(\wh\fb,\wh\fd,\wh\fj_0,\wh\fj_1)|,$ we get for the solution $(\wh b,\wh d,\wh j_0,\wh j_1)$
to \eqref{eq:diff0},
\begin{multline}\label{eq:noneqlf}
|(\wh b,\wh d,\wh j_0,\wh j_1)(t)|+\rho^2\int_0^t|(\wh b,\wh d,\wh j_0,\wh j_1)|\,d\tau
+\int_0^t|\wh\fj_0|\,d\tau+\cM\int_0^t|\wh\fj_1|\,d\tau\\
\leq C|(\wh b,\wh d,\wh j_0,\wh j_1)(0)|\quad\hbox{for all }\ 0\leq\rho\leq C_1,
\end{multline}
with\footnote{Note that the last term of $\fj_0$ given by \eqref{eq:chgdiffu2} is negligible
for $\rho\lesssim1$ and may thus be omitted.}
 $\wh\fj_0:=\wh j_0-\sqrt n\,\wh b-\sqrt n\,\frac\ep{\wt\cL}\,\rho\wh d$ and 
$\wh\fj_1:=\wh j_1-\frac\rho{\sqrt n\,\wt\cL\cM}\wh j_0+\frac{\rho\wh b}{\wt\cL\cL_s\cM}\cdotp$

\subsubsection*{Middle frequencies estimates} 

If $\wt\cL^2\cL_s\approx1$ then  using \eqref{eq:diff-hf7},  and the definition of $\wh\zeta_1$ versus that of $\wh\fj_1,$ we get
\begin{multline}\label{eq:noneqmf}
|(\rho\wh b,\wh d,\wh j_0,\wh j_1)(t)|+\rho^2\int_0^t|(\wh d,\wh j_0)|\,d\tau+\rho\int_0^t|\wh b|\,d\tau+\cM\int_0^t|\wh\fj_1|\,d\tau\\
\leq C|(\rho\wh b,\wh d,\wh j_0,\wh j_1)(0)|\quad\hbox{for all }\ C_1\leq\rho\leq c\wt\cL\cM.
\end{multline}
If $\wt\cL^2\cL_s\to+\infty$ then  \eqref{eq:diff-mf6} and \eqref{eq:diff-mf7} ensure that
\begin{multline}\label{eq:noneqmf1}
|(\rho\wh b,\wh d,\rho\wh j_0,\wh j_1)(t)|+\rho^2\int_0^t|\wh d|\,d\tau+\rho\int_0^t|(\wh b,\wh j_0)|\,d\tau+\cM\int_0^t|\wh\fj_1|\,d\tau\\
\leq C|(\rho\wh b,\wh d,\rho\wh j_0,\wh j_1)(0)|\quad\hbox{for all }\ \sqrt{1+n^{-1}}\leq\rho\leq \ep\nu\sqrt\cM,
\end{multline}
and, according to \eqref{eq:diff-hf7}
\begin{multline}\label{eq:noneqmf2}
|(\rho\wh b,\wh d,(1\!+\!\textstyle\frac{\ep^2\nu\cM}\rho)\wh j_0,\wh j_1)(t)|
+\Int_0^t|\rho\wh b|\,d\tau 
+\rho^2\int_0^t|\wh d|\,d\tau\\+(1\!+\!\textstyle\frac{\rho}{\ep^2\nu\cM})
\Int_0^t|\rho\wh j_0|\,d\tau+
 \cM\Int_0^t|\wh\fj_1|\,d\tau
\lesssim  |(\rho\wh b,\wh d,(1\!+\!\textstyle\frac{\ep^2\nu\cM}\rho)\wh j_0,\wh j_1)(0)|,
\end{multline}
 whenever  $C_1\ep\nu\sqrt\cM\leq \rho\leq c\ep\nu\cM.$ 
\medbreak
Hence
\begin{multline}\label{eq:noneqmfbis}
|(\rho\wh b,\wh d, \max(1,\min(\rho,\rho^{-1}\ep^2\nu\cM))\wh j_0,\wh j_1)(t)|+\rho\int_0^t|(\wh b,\min(1,
\textstyle\frac{\rho^2}{\ep^2\nu\cM})\wh j_0)|\,d\tau\\+\rho^2\int_0^t|\wh d|\,d\tau+\cM\int_0^t|\wh\fj_1|\,d\tau
\leq C|(\rho\wh b,\wh d,\max(1,\min(\rho,\rho^{-1}\ep^2\nu\cM))\wh j_0,\wh j_1)(0)|.
\end{multline}

\subsubsection*{High frequencies estimates} 

Using \eqref{eq:hfdiff3a}, we have
\begin{multline}\label{eq:noneqhf}
|(\rho\wh b,\wh d,\wh j_0,\wh j_1)(t)|+\rho^2\int_0^t|\wh d|\,d\tau+\rho\int_0^t|\wh b|\,d\tau+\cM\int_0^t|(\wh j_0,\wh j_1)|\,d\tau\\
\leq C|(\rho\wh b,\wh d,\wh j_0,\wh j_1)(0)|\quad\hbox{for  }\ \rho\geq c\wt\cL\cM.
\end{multline}

For notational simplicity, we shall slightly abusively change $c$ and $C_1$ to $1$ in all the following computations.

\subsubsection*{Optimal estimates in Besov spaces} 

If $\wt\cL^2\cL_s\approx1$ then
 localizing  \eqref{eq:diff} with nonzero source terms $f$ and $\vc g$ according to Littlewood-Paley operator $\ddk,$ 
using \eqref{eq:Pu} and following the computations leading to  \eqref{eq:noneqlf}, \eqref{eq:noneqmf}
and \eqref{eq:noneqhf} (combined with Fourier-Plancherel theorem)
we end up  for all $s\in\R$ with\footnote{Further explanations on the method will be supplied to the reader
in the next paragraph.}
\begin{multline}\label{eq:noneqestlin}
\|(\vu,j_0,\vc j_1)(t)\|_{\dot B^s_{2,1}} +\|b(t)\|_{\dot B^s_{2,1}}^{\ell,1}+\|b(t)\|_{\dot B^{s+1}_{2,1}}^{h,1}
+\int_0^t\|\vu\|_{\dot B^{s+2}_{2,1}}\,d\tau\\
+\int_0^t\|(b,j_0,\vc j_1)\|_{\dot B^{s+2}_{2,1}}^{\ell,1}\,d\tau
+\int_0^t\|b\|^{h,1}_{\dot B^{s+1}_{2,1}}\,d\tau
+\int_0^t\|\fj_0\|_{\dot B^s_{2,1}}^{\ell,1}\,d\tau\\
+\cM\int_0^t\|\vc\fj_1\|_{\dot B^s_{2,1}}^{\ell,\wt\cL\cM}\,d\tau
+\int_0^t\|j_0\|_{\dot B^{s+2}_{2,1}}^{m,1,\wt\cL\cM}\,d\tau
+\cM\int_0^t\|(j_0,\vc j_1)\|_{\dot B^s_{2,1}}^{h,\wt\cL\cM}\,d\tau\\
\lesssim \|(\vu,j_0,\vc j_1)(0)\|_{\dot B^s_{2,1}} +\|b(0)\|_{\dot B^s_{2,1}}^{\ell,1}+\|b(0)\|_{\dot B^{s+1}_{2,1}}^{h,1}
+\int_0^t\bigl(\|f\|_{\dot B^s_{2,1}}^{\ell,1}+\|f\|_{\dot B^{s+1}_{2,1}}^{h,1}+\|\vc g\|_{\dot B^s_{2,1}}\bigr)\,d\tau,
\end{multline}
with 
\begin{equation}\label{eq:J0J1}
\fj_0:=j_0-\sqrt n\,b-\sqrt n\frac\ep{\wt\cL}\div\vu\quad\hbox{and}\quad 
\vc \fj_1:=\vc j_1+\frac1{\sqrt n\,\wt\cL\cM}\nabla j_0-\frac{1}{\wt\cL\cL_s\cM}\nabla b.
\end{equation}
According to our previous work in \cite{DD}, the critical regularity framework corresponds to $s=n/2-1.$
Therefore, the following quantities will play an important role
$$
\|(b,\vu,j_0,\vc j_1)\|_{X}:=
\|b\|_{\dot B^{\frac n2-1}_{2,1}}^{\ell,1}
+\|b\|_{\dot B^{\frac n2}_{2,1}}^{h,1}
+\|(\vu,j_0,\vc j_1)\|_{\dot B^{\frac n2-1}_{2,1}},
$$
and
$$\displaylines{
\|(b,\vu,j_0,\vc j_1)\|_{Y}:=\sup_{t\geq0} \|(b,\vu,j_0,\vc j_1)(t)\|_{X}
+\int_{\R_+}\!\Bigl(\|b\|_{\dot B^{\frac n2+1}_{2,1}}^{\ell,1}+\|b\|_{\dot B^{\frac n2}_{2,1}}^{h,1}
+\|\vu\|_{\dot B^{\frac n2+1}_{2,1}}\Bigr)d\tau\hfill\cr\hfill
+\int_{\R_+}\Bigl(\cM\|\vc\fj_1\|_{\dot B^{\frac n2-1}_{2,1}}^{\ell,\ep\nu\cM}
+\|\fj_0\|_{\dot B^{\frac n2-1}_{2,1}}^{\ell,1}+ \|j_0\|_{\dot B^{\frac n2+1}_{2,1}}^{m,1,\ep\nu\cM}
  +\cM\|(j_0,\vc j_1)\|_{\dot B^{\frac n2-1}_{2,1}}^{h,\ep\nu\cM}\Bigr)d\tau.}
$$
We denote by $X$ and $Y$ the corresponding functional spaces
(where time continuity is imposed rather than just boundedness) and agree that $Y(t)$  stands for
the restriction of $Y$ to  the interval $[0,t].$ 
\medbreak
In the case $\ep^2\cL_s\to+\infty,$ we have to change slightly the definition of the norms $\|\cdot\|_X$ and $\|\cdot\|_Y$ 
as the middle frequencies obey \eqref{eq:noneqmfbis}. 
Consequently, we change $\|\cdot\|_X$ to $\|\cdot\|_{X_\infty}$ with 
$$\displaylines{
\|(b,\vu,j_0,\vc j_1)\|_{X_\infty}:=
\|b\|_{\dot B^{\frac n2-1}_{2,1}}^{\ell,1}
+\|b\|_{\dot B^{\frac n2}_{2,1}}^{h,1}
+\|(\vu,\vc j_1)\|_{\dot B^{\frac n2-1}_{2,1}}\hfill\cr\hfill+\|j_0\|_{\dot B^{\frac n2-1}_{2,1}}^{\ell,1}+\|j_0\|_{\dot B^{\frac n2-1}_{2,1}}^{h,\ep\nu\cM},
+\sum_{1\leq2^k\leq\ep\cM}2^{k\frac n2}\max(1,2^k,\min(2^{-k}\ep^2\nu\cM))\|\ddk j_0\|_{L^2}}
$$
and $\|\cdot\|_Y$ to 
$$\displaylines{
\|(b,\vu,j_0,\vc j_1)\|_{Y_\infty}:=\sup_{t\geq0} \|(b,\vu,j_0,\vc j_1)(t)\|_{X_\infty}
+\int_{\R_+}\!\Bigl(\|b\|_{\dot B^{\frac n2+1}_{2,1}}^{\ell,1}+\|b\|_{\dot B^{\frac n2}_{2,1}}^{h,1}
+\|\vu\|_{\dot B^{\frac n2+1}_{2,1}}\Bigr)d\tau\hfill\cr\hfill
+\int_{\R_+}\Bigl(\cM\|\vc\fj_1\|_{\dot B^{\frac n2-1}_{2,1}}^{\ell,\ep\nu\cM}
+\|\fj_0\|_{\dot B^{\frac n2-1}_{2,1}}^{\ell,1}+\cM\|(j_0,\vc j_1)\|_{\dot B^{\frac n2-1}_{2,1}}^{h,\ep\nu\cM}\Bigr)d\tau\hfill\cr\hfill
+\int_{\R_+}\sum_{1\leq2^k\leq\ep\cM} 2^{k\frac n2}\min(1,2^{2k}\ep^{-2}\nu^{-1}\cM^{-1})\|\ddk j_0\|_{L^2}d\tau.}
$$
 To prove global estimates for the nonlinear system \eqref{eq:NSrad}, the natural  next step would be to take advantage of  \eqref{eq:noneqestlin} with $s=n/2-1$
and all the nonlinear terms in the r.h.s.    Unfortunately, this does not work  for the convection term $\vu\cdot\nabla b$ 
causes a loss of one derivative (indeed, if $b$ is in $\dot B^{\frac n2}_{2,1}$ then $\vu\cdot\nabla b$ cannot be smoother than $\dot B^{\frac n2-1}_{2,1}$).
A nowadays standard way to overcome the difficulty is to paralinearize \eqref{eq:NSrad}, that
is to add to \eqref{eq:diff} the principal parts of the convection terms.
This is the aim of the next paragraph.

%%%%%%%%%%%%%%%%%%%%%%%%%%%%%%

\subsubsection{The paralinearized system}\label{ss:para}

Before introducing the paralinearized system  associated to \eqref{eq:NSrad}, let
us shortly  recall the definition of the paraproduct, according to the pioneering paper \cite{Bony} by J.-M. Bony. 
The (homogeneous)  \emph{paraproduct} between two tempered distributions $U$ and $V$ satisfying \eqref{eq:LFconv} is given by
$$T_UV:=\sum_k \dot S_{k-1} U\ddk V\quad\hbox{with}\quad
\dot S_{k-1}:=\chi(2^{k-1}D).$$
We also introduce
$$
T'_VU:=\sum_k   \dot S_{k+2}V\,\ddk U,
$$
and,  observe that, at least formally 
$$UV=T_UV+T'_VU.$$
To some extent, if $U$ is smooth enough then $T_UV$ may be seen as the \emph{principal part} 
of the product $UV. $ This motivates our  considering the following system 
\begin{equation}\label{eq:par-diff}
\left\{\begin{array}{l}
 \partial_t b+ T_{\vv}\cdot\nabla b+ \div \vu=F,\\[1ex]
  \partial_t \vu+ T_{\vv}\cdot\nabla\vu -\wt\cA\vu + \nabla b
    -\frac{\wt\cL\cM}n \vc j_1=\vec G,\\[1ex]
 \partial_t j_0+ \frac1{\ep\sqrt n}\,\div\vec j_1
+\frac{\wt\cL}\ep(j_0-\sqrt n\,b)=0,\\[1ex]
 \partial_t \vec j_1 + \frac1{\ep\sqrt n}\,{ \nabla j_0}+\frac{\wt\cL\cM}\ep\vec j_1=\vec 0,
 \end{array}\right.
\end{equation}
where   $\cA=\mu\Delta+(\lambda+\mu)\nabla\div,$ $\cM:=1+\cL_s,$ $\vv$ and $\vc G$ are given time dependent vector-fields, 
and $F$ is a given real valued function.

\begin{Proposition}\label{p:para-diff} For any smooth solution $(b,\vu,j_0,\vc j_1)$ we have the following a priori estimate if $0<m<+\infty$
$$\displaylines{
\|(b,\vu,j_0,\vc j_1)\|_{Y(t)}\leq C\biggl(\|(b,\vu,j_0,\vc j_1)(0)\|_{X}
+\int_0^t\|\nabla\vv\|_{L^\infty}\|(b,\vu,j_0,\vc j_1)\|_{X}\,d\tau\hfill\cr\hfill
+\int_0^t\|(\nabla F,\vc G)\|_{\dot B^{\frac n2-1}_{2,1}}^{h,1}\,d\tau
+\int_0^t\|(F-T_{\vv}\cdot\nabla b,\vc G-T_{\vv}\cdot\nabla\vu)\|_{\dot B^{\frac n2-1}_{2,1}}^{\ell,1}\,d\tau\biggr)\cdotp}
$$
A similar inequality holds if $m=+\infty,$ with  $X_\infty(t)$ and $Y_\infty(t)$ 
instead of  $X(t)$ and $Y(t).$ 
\end{Proposition}
\bProof
Localizing System \eqref{eq:par-diff} by means of  $\ddk$ yields
\begin{equation}\label{eq:par-diff-k}
\left\{\begin{array}{l}
 \partial_t \ddk b+ \ddk(T_{\vv}\cdot\nabla b)+ \div \ddk\vu=\ddk F,\\[1ex]
  \partial_t \ddk \vu+ \ddk(T_{\vv}\cdot\nabla\vu) -\wt\cA\ddk\vu + \nabla\ddk b
    -\frac{\wt\cL\cM}{n}\ddk\vc j_1=\ddk\vec G,\\[1ex]
 \partial_t \ddk j_0+  \frac1{\ep\sqrt n}\,{\mbox{div}\,\ddk\vec j_1}
+\frac{\wt\cL}\ep(\ddk j_0-\sqrt n\ddk b)=0,\\[1ex]
 \partial_t \ddk\vec j_1 +  \frac1{\ep\sqrt n}\,\nabla \ddk j_0+\frac{\wt\cL\cM}\ep\ddk \vec j_1=\vec 0.
 \end{array}\right.
\end{equation}
The important point is that in order to obtain all the estimates corresponding to $\rho\geq C_1,$
one only has to resort to combinations between fluid unknowns on one side, and
radiative unknowns, on the other side. 
This will enable us to use exactly the same energy method for \eqref{eq:par-diff} as for \eqref{eq:diff}, 
in the middle and high frequency regimes,  without introducing unwanted 
parts of convection terms in the inequalities.

\noindent\emph{1. Low frequencies: $2^k\leq C_1.$}

Including the para-convection terms   in the  source terms of  \eqref{eq:par-diff-k}
and repeating  the computations leading to \eqref{eq:noneqlf}, we get after  taking $L^2$ norms
and using Fourier-Plancherel theorem 
\begin{multline}\label{eq:lfnoneq1}
\|\ddk(b,\vu,j_0,\vc j_1)(t)\|_{L^2}+2^{2k}\int_0^t
\|\ddk(b,\vu)\|_{L^2}\,d\tau
+\int_0^t\|\ddk\fj_0\|_{L^2}\,d\tau\\+\nu\cL_s\int_0^t\|\ddk\vc\fj_1\|_{L^2}\,d\tau
\lesssim \|\ddk(b,\vu,j_0,\vc j_1)(0)\|_{L^2}\\
+\int_0^t\|\ddk(F-T_{\vu}\cdot\nabla b)\|_{L^2}\,d\tau
+\int_0^t\|\ddk(\vc G-T_{\vu}\cdot\nabla\vu)\|_{L^2}\,d\tau.
\end{multline}

\noindent\emph{2. Medium frequencies: $C_1\leq 2^k\leq c\wt\cL\cL_s.$}

Keeping in mind the proof of \eqref{eq:noneqmf},  we see that it is suitable to  introduce
$$
\vc\zeta_1:=\vc j_1+\frac1{\sqrt n\,\wt\cL\cM}\nabla j_0.
$$
Now, because we have
$$\left\{\begin{array}{l}
\d_t\ddk b+\div\ddk\vu=\ddk F,\\[1ex]
\d_t\ddk\vu-\wt\cA\ddk\vu+\nabla\ddk b=\frac1{n^{3/2}}\nabla\ddk j_0+\frac{\wt\cL\cM}n\ddk\vc\zeta_1+\ddk \vc G,
\end{array}\right.
$$
we easily get by computing 
\begin{equation}\label{eq:lya1}
\frac12\frac d{dt}\Bigl(2\|(\ddk b,\ddk\vu)\|_{L^2}^2+\|\ddk\nabla b\|_{L^2}^2+2(\ddk\nabla b|\ddk\vu)_{L^2}\Bigr),
\end{equation}
and by using  Lemma 4.1 in \cite{DD} to handle the para-convection terms,  
the following inequality for all  $2^{k}\geq C_1$ 
$$\displaylines{
\|(\ddk\nabla b,\ddk\vu)(t)\|_{L^2}+2^{2k}\int_0^t\|\ddk\vu\|_{L^2}\,d\tau
+\int_0^t\|\ddk\nabla b\|_{L^2}\,d\tau\lesssim \|(\ddk\nabla b,\ddk\vu)(0)\|_{L^2}\hfill\cr\hfill
+\int_0^t\|(\ddk\nabla F,\ddk\vc G)\|_{L^2}\,d\tau
+\int_0^t\|\nabla\ddk j_0\|_{L^2}\,d\tau+\wt\cL\cM\int_0^t\|\nabla\ddk\vc j_1\|_{L^2}\,d\tau\hfill\cr\hfill
+\sum_{k'\sim k}\int_0^t\|\nabla\vv\|_{L^\infty}\|(\dot\Delta_{k'}\nabla b,\dot\Delta_{k'}\vu)\|_{L^2}\,d\tau.}
$$
Then looking at the equations satisfied by $\ddk j_0$ and $\ddk\vc\zeta_1$ 
(in the spirit of \eqref{eq:diff-mf}), 
we derive inequalities similar to \eqref{eq:diff-mf2b} and \eqref{eq:diff-mf3b} 
for $\|\ddk j_0\|_{L^2}$ and $\|\ddk\vc\zeta_1\|_{L^2},$  and thus following 
the computations leading to \eqref{eq:diff-hf7}, we end up in the case $m<+\infty$ with 
\begin{multline}\label{eq:mfnoneq1}\|\ddk(\nabla b,\vu,\nabla j_0,\vc j_1)(t)\|_{L^2}+2^{2k}\int_0^t\|(\ddk\vu,\ddk j_0)\|_{L^2}\,d\tau
+\int_0^t\|\ddk \nabla b\|_{L^2}\,d\tau\\
+\nu\cL_s\int_0^t\|\ddk\vc\zeta_1\|_{L^2}\,d\tau
\lesssim \|\ddk(\nabla b,\vu,\nabla j_0,\vc j_1)(0)\|_{L^2}
+\int_0^t\|\ddk(\nabla F,\vc G)\|_{L^2}\,d\tau\\
+\sum_{k'\sim k}\int_0^t\|\nabla\vv\|_{L^\infty}\|\dot\Delta_{k'}(\nabla b,\vu)\|_{L^2}\,d\tau.
\end{multline}
Comparing the definition of $\vc\zeta_1$ and $\vc\fj_1,$ we
see that one may replace $\vc\zeta_1$ with $\vc\fj_1$ above, 
if $C_1\leq2^k\leq c\wt\cL\cL_s.$

The obvious modifications  to be done if $m=+\infty$ are left to the reader.

\noindent\emph{3. High frequencies: $2^k\geq c\wt\cL\cL_s.$}

Again, we compute \eqref{eq:lya1} to bound  the fluid unknowns. In addition, to handle radiative
unknowns, we compute   for some small enough $\kappa$ (see the proof
of \eqref{eq:hfdiff2}) the following quantity 
$$
\frac12\frac d{dt}\bigl(\|\ddk j_0\|_{L^2}^2+\|\ddk\vc j_1\|_{L^2}^2-\kappa\wt\cL\cM\:2^{-2k}(\ddk j_0|\ddk\div\vc j_1)_{L^2}\bigr).
$$
Combining the computations leading to \eqref{eq:noneqhf} with Fourier-Plancherel theorem 
and Lemma 4.1 in \cite{DD} eventually yields
$$
\displaylines{\|\ddk(\nabla b,\vu,j_0,\vc j_1)(t)\|_{L^2}+2^{2k}\int_0^t\|\ddk\vu\|_{L^2}\,d\tau
+2^k\int_0^t\|\ddk b\|_{L^2}\,d\tau+\nu\cL_s\int_0^t\|(\ddk j_0,\ddk\vc j_1)\|_{L^2}\,d\tau\hfill\cr\hfill
\lesssim \|\ddk(\nabla b,\vu,j_0,\vc j_1)(0)\|_{L^2}
+\int_0^t\|\ddk(\nabla F,\vc G)\|_{L^2}\,d\tau
+\sum_{k'\sim k}\int_0^t\|\nabla\vv\|_{L^\infty}\|\dot\Delta_{k'}(\nabla b,\vu)\|_{L^2}\,d\tau.}
$$
Finally, multiplying \eqref{eq:lfnoneq1}, \eqref{eq:mfnoneq1} and the above 
inequality by  $2^{k(\frac n2-1)}$ and summing up over $k$ completes the proof of the proposition.
\qed

%%%%%%%%%%%%%%%%%%%%%%%%%%%%%%%%%%%

\subsubsection{A global existence result}

According to the computations of the previous paragraph and to the change of variables  \eqref{eq:nu},
it is suitable to introduce the following norms for getting  global solutions with uniform estimates in the 
case\footnote{Writing out the corresponding definition if $m=+\infty$ is left to the reader.} $m<+\infty$ 
$$
\|(b,\vu,j_0,\vc j_1)\|_{X^\nu_\ep}:=
\|b\|_{\dot B^{\frac n2-1}_{2,1}}^{\ell,\nu^{-1}}
+\nu\|b\|_{\dot B^{\frac n2}_{2,1}}^{h,\nu^{-1}}
+\|(\vu,j_0,\vc j_1)\|_{\dot B^{\frac n2-1}_{2,1}},
$$
and
$$\displaylines{
\|(b,\vu,j_0,\vc j_1)\|_{Y^\nu_\ep}:=\sup_{t\geq0} \|(b,\vu,j_0,\vc j_1)(t)\|_{X^\nu}
+\int_{\R_+}\!\Bigl(\nu\|b\|_{\dot B^{\frac n2+1}_{2,1}}^{\ell,\nu^{-1}}+\|b\|_{\dot B^{\frac n2}_{2,1}}^{h,\nu^{-1}}
+\nu\|\vu\|_{\dot B^{\frac n2+1}_{2,1}}\Bigr)d\tau\hfill\cr\hfill
+\int_{\R_+}\Bigl(\nu^{-1}\cM\|\vc\fj_1\|_{\dot B^{\frac n2-1}_{2,1}}^{\ell,\ep\cM}
+\nu^{-1}\|\fj_0\|_{\dot B^{\frac n2-1}_{2,1}}^{\ell,\nu^{-1}}+ \nu\|j_0\|_{\dot B^{\frac n2+1}_{2,1}}^{m,\nu^{-1},\ep\cM}
  +\nu^{-1}\cM\|(j_0,\vc j_1)\|_{\dot B^{\frac n2-1}_{2,1}}^{h,\ep\cM}\Bigr)d\tau,}
$$
with $\fj_0:=j_0-b-\frac\ep{\cL}\div\vu$ and $\vc\fj_1:=j_1+\frac1{\cL\cM}\nabla j_0-\frac1{\cL\cL_s\cM}\nabla b.$
\medbreak
Of course, if  $(b',\vu',j_0',\vc j'_1)$ and $(b,\vu,j_0,\vc j_1)$ are interrelated through \eqref{eq:nu}
and $\nu\cL$ is used for  $(b',\vu',j_0',\vc j'_1)$ instead of $\cL,$ then we have 
$$
\|(b',\vu',j_0',\vc j'_1)\|_{X^1_\eps}=\nu^{-1}\|(b,\vu,j_0,\vc j_1)\|_{X^\nu_\eps}
\quad\hbox{and}\quad
\|(b',\vu',j'_0,\vc j'_1)\|_{Y^1_\eps}=\nu^{-1}\|(b,\vu,j_0,\vc j_1)\|_{Y^\nu_\eps}.
$$
 \begin{Theorem}\label{th:diff1}
Assume that $\cL\approx1,$ that $\liminf  \ep^{-1}n\nu\cL>1$ and that $\cL^2\cL_s\approx1.$ 
There exists a positive constant $\eta$ depending only on $\mu/\nu,$ $n$ 
and on the pressure law such that if $\eps$ is small enough and 
the data $(b_0^\eps,\vu_0^\eps,j_{0,0}^\eps,\vc j_{1,0}^\eps)$ satisfy  
\begin{equation}\label{eq:small-diff}
\|(b_0^\eps,\vu_0^\eps,j_{0,0}^\eps,\vc j_{1,0}^\eps)\|_{X_\eps^\nu}\leq \eta\nu,
\end{equation}
then System  \eqref{eq:NSrad} admits a unique global 
solution $(b^\eps,\vu^\eps,j_0^\eps,\vc j_1^\eps)$ in $Y_\eps^\nu.$
In addition, we have
\begin{equation}\label{eq:unif-diff}
\|(b^\eps,\vu^\eps,j_0^\eps,\vc j_1^\eps)\|_{Y_\eps^\nu}\leq C
\|(b_0^\eps,\vu_0^\eps,j_{0,0}^\eps,\vc j_{1,0}^\eps)\|_{X_\eps^\nu}.
\end{equation}
A similar result holds true if $\cL^2\cL_s\to+\infty.$
\end{Theorem}
\bProof 
Performing the change of variables proposed in \eqref{eq:nu}
reduces the proof to the case $\nu=1$ (changing $\cL$ into $\wt\cL:=\nu\cL$). 
Hence we consider a smooth enough solution to \eqref{eq:NSradbis}, and
 show that one may close the estimates globally\footnote{Existence follows from spectral truncation as in e.g. \cite{BCD}, Chap. 10, 
 and is thus omitted. As for uniqueness, we refer to \cite{DD}.} under Assumption \eqref{eq:small-diff}.
\medbreak
Let us  set $U_0:=\|(b_0^\eps,\vu_0^\eps,j_{0,0}^\eps,\vc j_{1,0}^\eps)\|_{X_\eps^1}$
and $U(t):=\|(b^\eps,\vu^\eps,j_0^\eps,\vc j_1^\eps)\|_{Y_\eps^1(t)}.$
In what follows, we drop exponents $\eps$ 
for notational simplicity. Finally, to shorten the presentation, we just treat the case where $\cL^2\cL_s\approx1.$
\medbreak
Now applying Proposition \ref{p:para-diff} with $\vv=\vu,$ $F:=-T'_{\nabla b}\cdot\vu-k_1(b)\div\vu$ and
$$\vc G:=-T'_{\nabla\vu}\cdot\vu+k_2(b)\wt\cA\vu-k_3(b)\nabla b
+\frac{\wt\cL\cM}n\,k_4(b)\vc j_1,$$
yields for all $t\geq0$
\begin{multline}\label{eq:U-diff}
U(t)\leq C\biggl(U_0+\int_0^t\|\nabla\vu\|_{L^\infty}\|(b,\vu,j_0,\vc j_1)\|_{X_\eps^1}\,d\tau
\\+\int_0^t\Bigl(\|F\|_{\dot B^{\frac n2}_{2,1}}+\|F-T_{\vu}\cdot\nabla b\|_{\dot B^{\frac n2-1}_{2,1}}
+\|\vc G\|_{\dot B^{\frac n2-1}_{2,1}}+\|T_{\vu}\cdot\nabla\vu\|_{\dot B^{\frac n2-1}_{2,1}}\Bigr)\,d\tau\biggr)\cdotp
\end{multline}
Using standard continuity results for the paraproduct and remainder, 
and composition estimates leads to 
$$
\begin{aligned}
\|T'_{\nabla b}\cdot\vu\|_{\dot B^{\frac n2}_{2,1}}
&\leq C\|\nabla b\|_{\dot B^{\frac n2-1}_{2,1}}\|\vu\|_{\dot B^{\frac n2+1}_{2,1}},\\[1ex]
\|k_1(b)\div\vu\|_{\dot B^{\frac n2}_{2,1}}&\leq C\|b\|_{\dot B^{\frac n2}_{2,1}}\|\div\vu\|_{\dot B^{\frac n2}_{2,1}}.\end{aligned}
$$
Hence we have
\begin{equation}\label{eq:U-diff1}
\int_0^t\|F\|_{\dot B^{\frac n2}_{2,1}}\,d\tau\leq C U^2(t).
\end{equation}
We also have
$$
\begin{aligned}
\|T_{\vu}\cdot\nabla\vu\|_{\dot B^{\frac n2-1}_{2,1}}&\leq C\|\vu\|_{\dot B^{\frac n2-1}_{2,1}}
\|\nabla\vu\|_{\dot B^{\frac n2}_{2,1}},\\
\|T'_{\nabla\vu}\cdot\vu\|_{\dot B^{\frac n2}_{2,1}}&\leq C\|\nabla\vu\|_{\dot B^{\frac n2}_{2,1}}
\|\vu\|_{\dot B^{\frac n2-1}_{2,1}},\\
\|k_2(b)\wt\cA\vu\|_{\dot B^{\frac n2-1}_{2,1}}&\leq C\|b\|_{\dot B^{\frac n2}_{2,1}}\|\nabla^2\vu\|_{\dot B^{\frac n2-1}_{2,1}}, \\
\|k_3(b)\nabla b\|_{\dot B^{\frac n2-1}_{2,1}}&\leq C\|b\|_{\dot B^{\frac n2}_{2,1}}\|\nabla b\|_{\dot B^{\frac n2-1}_{2,1}}.\end{aligned}
$$
Bounding $\wt\cL\cM k_4(b)\vc j_1$ is slightly more involved as it is not
true that the low frequencies of $\vc j_1$ are bounded in $L^1(\R_+;\dot B^{\frac n2-1}_{2,1}).$
However, one may write that 
$$
\wt\cL\cM \vc j_1=\wt\cL\cM\vc j_1^{h,\wt\cL\cM}+ \wt\cL\cM\vc\fj_1^{\ell,\wt\cL\cM}
+\cL_s^{-1}\nabla b^{\ell,\wt\cL\cM}-n^{-1/2}\nabla j_0^{\ell,\wt\cL\cM}.
$$
Therefore
$$\displaylines{
\wt\cL\cM\| k_4(b)\vc j_1\|_{L^1_t(\dot B^{\frac n2-1}_{2,1})}
\lesssim \wt\cL\cM\bigl(\|\vc j_1\|_{L^1_t(\dot B^{\frac n2-1}_{2,1})}^{h,\wt\cL\cM}
+\|\vc\fj_1\|_{L^1_t(\dot B^{\frac n2-1}_{2,1})}^{\ell,\wt\cL\cM}\bigr)\|b\|_{L^\infty_t(\dot B^{\frac n2}_{2,1})}
\hfill\cr\hfill+\|b\|_{L^2_t(\dot B^{\frac n2}_{2,1})}\bigl(\cL_s^{-1}
\|\nabla b\|^{\ell,\wt\cL\cM}_{L^2_t(\dot B^{\frac n2-1}_{2,1})}
+\|\nabla j_0\|^{\ell,\wt\cL\cM}_{L^2_t(\dot B^{\frac n2-1}_{2,1})}\bigr).}
$$
Hence 
\begin{equation}\label{eq:U-diff2}
\int_0^t\bigl(\|T_{\vu}\cdot\nabla\vu\|_{\dot B^{\frac n2-1}_{2,1}}
+\|\vc G\|_{\dot B^{\frac n2-1}_{2,1}}\bigr)\,d\tau\leq  CU^2(t).
\end{equation}
Finally 
$$
\begin{aligned}
\|\vu\cdot\nabla b\|_{\dot B^{\frac n2-1}_{2,1}}\leq& C\|\vu\|_{\dot B^{\frac n2}_{2,1}}\|\nabla b\|_{\dot B^{\frac n2-1}_{2,1}},\\[1ex]
\|k_1(b)\div\vu\|_{\dot B^{\frac n2-1}_{2,1}}\leq& C\|b\|_{\dot B^{\frac n2}_{2,1}}\|\div\vu\|_{\dot B^{\frac n2-1}_{2,1}}.
\end{aligned}
$$
Therefore, by Cauchy-Schwarz inequality 
\begin{equation}\label{eq:U-diff3}
\int_0^t\|F-T_{\vu}\cdot\nabla b\|_{\dot B^{\frac n2-1}_{2,1}}\,d\tau\leq CU^2(t).
\end{equation}
Inserting   \eqref{eq:U-diff1}, \eqref{eq:U-diff2}, \eqref{eq:U-diff3} in \eqref{eq:U-diff} and remembering that $\dot B^{\frac n2}_{2,1}
\hookrightarrow L^\infty$ (to ensure that, say, $|b|\leq1/2$ if $\|b\|_{\dot B^{\frac n2}_{2,1}}$ is small enough), we end up with 
$$
U(t)\leq C(U_0+U^2(t))\quad\hbox{for all }\ t\geq0.
$$
By a standard bootstrap argument, we easily deduce that 
$$
U(t)\leq 2CU_0\quad\hbox{for all }\ t\geq0,
$$
provided the data have been chosen so that  $4C^2U_0\leq1.$
\qed

%%%%%%%%%%%%%%%%%%%%%%%%%%%%%%%%%%%%%%%%%%%%%%

\subsection{Study of the limit system}

In this paragraph, we prove the existence and  uniqueness of strong (small) solutions with critical regularity
for Systems \eqref{eq:noneq1} and \eqref{eq:noneq2}. 
We shall give a common proof that works for both systems. 

Before giving the global existence statement, let us introduce the solution space 
\begin{itemize}
\item If $m\in(0,+\infty)$ (that is for System \eqref{eq:noneq1}) then 
Initial data will be taken in the space $\cX^\nu$ which is the set of triplets  $(b,\vu,j_0)$ satisfying 
$$
\|(b,\vu,j_0)\|_{\cX^\nu}:=\|b\|_{\dot B^{\frac n2-1}_{2,1}}^{\ell,\nu^{-1}}+
\nu\|b\|_{\dot B^{\frac n2}_{2,1}}^{h,\nu^{-1}}+\|\vu\|_{\dot B^{\frac n2-1}_{2,1}}+
\nu\|j_0\|_{\dot B^{\frac n2}_{2,1}}^{\ell,\nu^{-1}}
+\nu^{-1}\|j_0\|_{\dot B^{\frac n2-2}_{2,1}}^{h,\nu^{-1}}<\infty,
$$
and  the solution space $\cY^\nu$ will be  the set of triplets $(b,\vu,j_0)$
in  $\cC_b(\R_+;\cX^\nu)$ satisfying
$$
\|(b,\vu,j_0)\|_{\cY^\nu}\!:=\sup_{t\geq0} \|(b,\vu,j_0)(t)\|_{\cX^\nu}
+\!\int_{\R^+}\!\!\Bigl(\|j_0-b\|_{\dot B^{\frac n2}_{2,1}}+\nu\|b\|_{\dot B^{\frac n2\!+\!1}_{2,1}}^{\ell,\nu^{-1}}
+\nu\|\vu\|_{\dot B^{\frac n2\!+\!1}_{2,1}}+\|j_0\|_{B^{\frac n2}_{2,1}}^{h,\nu^{-1}}\Bigr)d\tau.
$$
\item 
If  $m=+\infty$  (that is for System \eqref{eq:noneq2}), 
Initial data will be taken in the space $\cX^\nu_\infty$ which is the set of triplets 
$(b,\vu,j_0)$ satisfying 
$$
\|(b,\vu,j_0)\|_{\cX^\nu_\infty}:=\|b\|_{\dot B^{\frac n2-1}_{2,1}}^{\ell,\nu^{-1}}+
\nu\|b\|_{\dot B^{\frac n2}_{2,1}}^{h,\nu^{-1}}+\|\vu\|_{\dot B^{\frac n2-1}_{2,1}}+
\nu\|j_0\|_{\dot B^{\frac n2}_{2,1}}<\infty,
$$
and the solution space $\cY^\nu_\infty$ will be  the set of triplets $(b,\vu,j_0)$
in  $\cC_b(\R_+;\cX^\nu_\infty)$ satisfying
$$\displaylines{
\|(b,\vu,j_0)\|_{\cY^\nu_\infty}:=\sup_{t\geq0} \|(b,\vu,j_0)(t)\|_{\cX^\nu_\infty}
\hfill\cr\hfill+\int_{\R^+}\Bigl(\|j_0-b\|_{\dot B^{\frac n2}_{2,1}}+\nu\|b\|_{\dot B^{\frac n2+1}_{2,1}}^{\ell,\nu^{-1}}
+\nu\|\vu\|_{\dot B^{\frac n2+1}_{2,1}}+\|j_0\|_{B^{\frac n2}_{2,1}}^{h,\nu^{-1}}\Bigr)d\tau<\infty.}
$$
\end{itemize}
\begin{Theorem}\label{th:noneqlim1}
There exist two positive constants $c$ and $C$ so that if 
\begin{eqnarray}\label{eq:smallnoneq2}
&&\quad\|(b_0,\vu_0,j_{0,0})\|_{\cX^\nu}  \leq c\nu
\quad(\hbox{case }\ m<+\infty),\\\label{eq:smallnoneq1}
&\quad\hbox{or}&\quad\|(b_0,\vu_0,j_{0,0})\|_{\cX^\nu_\infty} \leq c\nu\quad(\hbox{case }\ m=+\infty),
\end{eqnarray}
then System \eqref{eq:noneq1} (resp. \eqref{eq:noneq2})   admits a unique solution in the space $\cY^\nu$ (resp. $\cY^\nu_\infty$) 
satisfying in addition,
\begin{eqnarray}\label{eq:estnoneq2} 
&&\|(b,\vu,j_0)\|_{\cY^\nu}\leq C\|(b_0,\vu_0,j_{0,0})\|_{\cX^\nu}\quad\hbox{if }\ m<+\infty,\\
\label{eq:estnoneq1} 
&&\|(b,\vu,j_0)\|_{\cY^\nu_\infty}\leq  C\|(b_0,\vu_0,j_{0,0})\|_{\cX^\nu_\infty}\quad\hbox{if }\ m=+\infty\cdotp
\end{eqnarray}\end{Theorem}
\bProof 
Set $\wt\kappa:=\kappa/n$ and $\wt m:=mn.$
As usual, it is enough to treat  the case $\nu=1$ as performing 
the change of unknowns
$$
(b,u,j_0)(t,x)=(\wt b,\wt u,\wt j_0)(\nu^{-1}t,\nu^{-1}x),
$$
gives Systems \eqref{eq:noneq1} or \eqref{eq:noneq2} for $(\wt b,\wt u,\wt j_0)$ with $\nu=1$ and 
$\wt\cA:=\cA/\nu$ and, obviously
$$
\|(b,u,j_0)(t)\|_{\cX^\nu}=\nu\|(\wt b,\wt u,\wt j_0)(\nu^{-1}t)\|_{\cX^1}
\quad\hbox{and}\quad
\|(b,u,j_0)\|_{\cY^\nu}=\nu\|(\wt b,\wt u,\wt j_0)\|_{\cY^1}.
$$
Let us start with 
the study of the linearized equations with no source term, namely
\begin{equation}%\label{eq:l-diff4}
\left\{\begin{array}{l}
\d_tb+\div\vu=0,\\[1ex]
\d_t\vu-\wt\cA\vu+\nabla b+n^{-1}\nabla j_0=\vc 0,\\[1ex]
\d_t j_0+\wt\kappa(j_0-\wt m^{-1}\Delta j_0-b)=0.\end{array}\right.
\end{equation}
The divergence-free part $\cP\vu$ of the velocity satisfies
$$
\d_t\cP\vu-\mu\Delta\cP\vu=\vc 0,
$$
while the coupling between $b,$ $d:=\Lambda^{-1}\div\vu$ and $j_0$ is described by 
\begin{equation}\label{eq:l-diff5}\left\{\begin{array}{l}
\d_tb+\Lambda d=0,\\[1ex]
\d_td-\Delta d-\Lambda b-n^{-1}\Lambda j_0=0,\\[1ex]
\d_tj_0+\wt\kappa (j_0-\wt m^{-1}\Delta j_0-b)=0.\end{array}\right.
\end{equation}
Note that the stability of a similar  system has already been established 
in the previous section for  $\kappa>1$ (or, equivalently, $\wt\kappa>1/n$).

\subsubsection*{Linear estimates for low frequencies}

We introduce $\zeta_0:=j_0-b-\wt\kappa^{-1}\Lambda d$ and notice that
$$\left\{\begin{array}{l}
\d_t\wh b+\rho\wh d=0,\\[1ex]
\d_t\wh d+\rho^2\bigl(1-\frac1{\wt\kappa n}\bigr)\wh d-\bigl(1+\frac1n\bigr)\rho\wh b=\frac1n\rho\wh\zeta_0,\\[1ex]
\d_t\wh\zeta_0+\bigl(\wt\kappa+(\frac{\wt\kappa}{\wt m}+\frac1{\wt\kappa n})\rho^2\bigr)\wh\zeta_0
=-\Bigl(\bigl(1+\frac1n\bigr)\frac1{\wt\kappa}+\frac{\wt\kappa}{\wt m}\Bigr)\rho^2\wh b
+\Bigl(\bigl(1-\frac1{\wt\kappa n}\bigr)\frac1{\wt\kappa}-\frac1{\wt m}\Bigr)\rho^3\wh d.\end{array}\right.
$$
On one hand, because $\wt\kappa n>1,$ the method described in the appendix 
(see in particular \eqref{eq:ODE7})
allows to write that, omitting the dependence with respect to $\wt\kappa$
$$
|(\wh b,\wh d)(t)|+\rho^2\int_0^t|(\wh b,\wh d)|\,d\tau
\lesssim |(\wh b,\wh d)(0)|+\rho\int_0^t|\wh\zeta_0|\,d\tau.
$$
On the other hand, the last equation directly gives
$$
|\wh\zeta_0(t)|+\Bigl(\wt\kappa+(\frac{\wt\kappa}{\wt m}+\frac1{\wt\kappa n})\rho^2\Bigr)\int_0^t|\wh\zeta_0|\,d\tau
\leq|\wh\zeta_0(0)|+C\Bigl(1+\frac1m\Bigr)\rho^2\int_0^t|(\wh b,\rho\wh d)|\,d\tau.
$$
Hence plugging the second inequality in the first one 
$$
|(\wh b,\wh d)(t)|+\rho^2\int_0^t|(\wh b,\wh d)|\,d\tau
\lesssim |(\wh b,\wh d)(0)|+\frac{\rho}{1\!+\!(1\!+\!m^{-1})\rho^2}\biggl(|\wh\zeta_0(0)|+(1\!+\!m^{-1})\rho^2\int_0^t|(\wh b,\rho\wh d)|\,d\tau\biggr)\cdotp
$$
It is clear that the last term may be absorbed by the integral of the l.h.s. if
$\rho\ll \frac m{1+m}\cdotp$  Hence we eventually get
for some small enough $\rho_\ell>0$
\begin{equation}\label{eq:l-diff6}
|(\wh b,\wh d,\rho\wh\zeta_0)(t)|+\rho^2\int_0^t|(\wh b,\wh d)|\,d\tau
+\int_0^t|\rho\wh\zeta_0|\,d\tau\lesssim |(\wh b,\wh d,\rho\wh\zeta_0)(0)|
\quad\hbox{if}\quad\rho\leq\Bigl(\frac m{1+m}\Bigr)\rho_\ell.
\end{equation}

\subsubsection*{Linear estimates for high frequencies}
We set $\delta:=d-\Lambda^{-1}b$ and notice that
$$\left\{\begin{array}{l}
\d_t\wh b+\wh b=-\rho\wh\delta,\\[1ex]
\d_t\wh\delta+(\rho^2-1)\wh\delta=\rho^{-1}\wh b+n^{-1}\rho\wh j_0,\\[1ex]
\d_t\wh j_0+\wt\kappa\bigl(1+\frac{\rho^2}{\wt m}\bigr)\wh j_0=\wt\kappa\wh b.\end{array}\right.
$$
Therefore
$$
|\wh\delta(t)|+(\rho^2-1)\int_0^t|\wh \delta|\,d\tau\leq|\wh\delta(0)|+\rho^{-1}\int_0^t|\wh b|\,d\tau
+\frac1n\int_0^t\rho|\wh j_0|\,d\tau.
$$
At the same time
$$
\displaylines{
|\wh b(t)|+\int_0^t|\wh b|\,d\tau\leq|\wh b(0)|+\rho\int_0^t|\wh \delta|\,d\tau,\cr
|\wh j_0(t)|+\wt\kappa\Bigl(1+\frac{\rho^2}{\wt m}\Bigr)\int_0^t|\wh j_0|\,d\tau\leq|\wh j_0(0)|+\wt\kappa\int_0^t|\wh b|\,d\tau.}
$$
Hence
$$\displaylines{
|\wh\delta(t)|+(\rho^2-1)\int_0^t\!|\wh \delta|\,d\tau\leq|\wh\delta(0)|+
\Bigl(\frac1\rho+\frac\rho n\Bigr)|\wh b(0)|\hfill\cr\hfill+\Bigl(\frac{n^{-1}\wt\kappa^{-1}}{1\!+\!\wt m^{-1}\rho^2}\Bigr)\rho|\wh j_0(0)|
+\Bigl(1+\frac{\rho^2}{n(1+\wt m^{-1}\rho^2)}\Bigr)\!\int_0^t\!|\wh\delta|\,d\tau.}
$$
Therefore there exists a constant $\rho_h$ depending only on $m$ and 
$n$ (with $n\geq2$ if $m=+\infty$)  such that for $\rho\geq \rho_h,$ we have
\begin{equation}\label{eq:l-diff8}
|(\rho\wh b,\wh\delta)(t)|+\min(\rho, m\rho^{-1})|\wh j_0(t)|
+\rho\int_0^t|(\wh b,\wh j_0)|\,d\tau+\rho^2\int_0^t|\wh \delta|\,d\tau\lesssim|(\rho\wh b,\wh\delta,\rho\wh j_0)(0)|.
\end{equation}
Of course,  one may replace $\delta$ with $d$ in \eqref{eq:l-diff8}.

\subsubsection*{Linear estimates for medium frequencies}

The stability argument used just below  \eqref{eq:diff-mf2} allows to write that there exist two constants
$c$ and $C$ depending continuously on $1/m,$ such that if $\rho\in[\frac m{m+1}\rho_\ell,\rho_h]$ then 
\begin{equation}\label{eq:l-diff9}
|(\wh b,\wh d,\wh j_0)(t)|\leq Ce^{-ct} |(\wh b,\wh d,\wh j_0)(0)|.
\end{equation}

\subsubsection*{Estimates for the paralinearized system}

The previous steps allow us to get handy estimates for the following 
paralinearized version of System \eqref{eq:noneq2}
\begin{equation}\label{eq:l-diff10}\left\{\begin{array}{l}
\d_tb+T_{\vv}\cdot\nabla b+\div\vu=F,\\[1ex]
\d_t\vu+T_{\vv}\cdot\nabla\vu-\wt\cA\vu+\nabla b+n^{-1}\nabla j_0=\vc G,\\[1ex]
\d_t j_0+\wt\kappa(j_0-\wt m^{-1}\Delta j_0-b)=0.\end{array}\right.
\end{equation}
More precisely,   following the steps leading to \eqref{eq:l-diff6}, 
\eqref{eq:l-diff8} and \eqref{eq:l-diff9},  introducing  $\zeta_0:=j_0-b-\wt\kappa^{-1}\div\vu,$ and
 arguing  as in Subsection \ref{ss:para}
we end up with\footnote{Here we do not track the dependency with respect to $m.$} 
\begin{multline}\label{eq:l-diff11}
\|j_0(t)\|_{\dot B^{\frac n2}_{2,1}}^{\ell,1}+\|(b,\vu)(t)\|_{\dot B^{\frac n2-1}_{2,1}}^{\ell,1}
+\int_0^t\bigl(\|j_0-b\|_{\dot B^{\frac n2}_{2,1}}^{\ell,1}
+\|(b,\vu)\|_{\dot B^{\frac n2+1}_{2,1}}^{\ell,1}\bigr)\,d\tau\\
\lesssim \|j_0(0)\|_{\dot B^{\frac n2}_{2,1}}^{\ell,1}+\|(b,\vu)(0)\|_{\dot B^{\frac n2-1}_{2,1}}^{\ell,1}
+\int_0^t\Bigl(\|F-T_{\vv}\cdot\nabla b\|_{\dot B^{\frac n2-1}_{2,1}}^{\ell,1}
+\|\vc G-T_{\vv}\cdot\nabla\vu\|_{\dot B^{\frac n2-1}_{2,1}}^{\ell,1}\Bigr)d\tau.
\end{multline}
For high frequencies, we get, in the case $m=+\infty$
\begin{multline}\label{eq:l-diff12a}
\|(b,j_0)(t)\|_{\dot B^{\frac n2}_{2,1}}^{h,1}+\|\vu(t)\|_{\dot B^{\frac n2-1}_{2,1}}^{h,1}
+\int_0^t\bigl(\|(b,j_0)\|_{\dot B^{\frac n2}_{2,1}}^{h,1}+\|\vu\|_{\dot B^{\frac n2+1}_{2,1}}^{h,1}\bigr)\,d\tau\\
\lesssim \|(b,j_0)(0)\|_{\dot B^{\frac n2}_{2,1}}^{h,1}+\|\vu(0)\|_{\dot B^{\frac n2-1}_{2,1}}^{h,1}
+\int_0^t\Bigl(\|F\|_{\dot B^{\frac n2}_{2,1}}^{h,1}+\|\vc G\|_{\dot B^{\frac n2-1}_{2,1}}^{h,1}\Bigr)d\tau
\\+\int_0^t\|\nabla\vv\|_{L^\infty}\bigl(\|(b,j_0)\|_{\dot B^{\frac n2}_{2,1}}+
\|\vu\|_{\dot B^{\frac n2-1}_{2,1}}\bigr)\,d\tau,\end{multline}
and if $0<m<+\infty$
\begin{multline}\label{eq:l-diff12b}
\|b(t)\|_{\dot B^{\frac n2}_{2,1}}^{h,1}+\|j_0(t)\|_{\dot B^{\frac n2-2}_{2,1}}^{h,1}+\|\vu(t)\|_{\dot B^{\frac n2-1}_{2,1}}^{h,1}
+\int_0^t\bigl(\|(b,j_0)\|_{\dot B^{\frac n2}_{2,1}}^{h,1}+\|\vu\|_{\dot B^{\frac n2+1}_{2,1}}^{h,1}\bigr)\,d\tau\\
\lesssim \|b(0)\|_{\dot B^{\frac n2}_{2,1}}^{h,1}+\|j_0(0)\|_{\dot B^{\frac n2-2}_{2,1}}^{h,1}+\|\vu(0)\|_{\dot B^{\frac n2-1}_{2,1}}^{h,1}
+\int_0^t\Bigl(\|F\|_{\dot B^{\frac n2}_{2,1}}^{h,1}+\|\vc G\|_{\dot B^{\frac n2-1}_{2,1}}^{h,1}\Bigr)d\tau
\\+\int_0^t\|\nabla\vv\|_{L^\infty}\bigl(\|b\|_{\dot B^{\frac n2}_{2,1}}+\|j_0\|_{\dot B^{\frac n2-2}_{2,1}}^h+
\|\vu\|_{\dot B^{\frac n2-1}_{2,1}}\bigr)\,d\tau.\end{multline}

\subsubsection*{Proof of existence}

We only establish  global-in-time a priori bounds in 
the space $\cY^1$ or $\cY^1_\infty$ for the solutions to  \eqref{eq:noneq1} or \eqref{eq:noneq2}
with  data satisfying \eqref{eq:smallnoneq2} or \eqref{eq:smallnoneq1}.
Our proof is based on  \eqref{eq:l-diff11},  \eqref{eq:l-diff12a} and \eqref{eq:l-diff12b}
with $\vv=\vu$ 
$$
F=-T'_{\nabla b}\cdot\vv-k_1(b)\div\vu
\quad\hbox{and}\quad
\vc G=-T'_{\nabla\vu}\cdot\vu+k_2(b)\wt\cA\vu-k_3(b)\nabla b-n^{-1}k_4(b)\nabla j_0.
$$
Bounding $\|F-T_{\vv}\cdot\nabla b\|_{\dot B^{\frac n2-1}_{2,1}}^{\ell,1}$ and $\|F\|_{\dot B^{\frac n2}_{2,1}}^{h,1}$ relies on \eqref{eq:U-diff1} and \eqref{eq:U-diff3}. 
 As regards $\vc G,$ the computations that we did in the proof of Theorem \ref{th:diff1} 
 ensure that the first three terms may  be bounded as in  \eqref{eq:U-diff2}. 
 To handle the last term, $k_4(b)\nabla j_0,$ in the case\footnote{The case $m=+\infty$
 does not require that decomposition.}  $m<+\infty$  we  use the decomposition
 $$
k_4(b)\nabla j_0=k_4(b)\nabla b+k_4(b)\nabla(j_0-b).
$$
The first term may be bounded quadratically exactly as $k_3(b)\nabla b.$
As for the last term, we may write 
$$
\|k_4(b)\nabla(j_0-b)\|_{L^1_t(\dot B^{\frac  n2}_{2,1})}\lesssim \|b\|_{L^\infty_t(\dot B^{\frac n2}_{2,1})}
\|\nabla(j_0-b)\|_{L^1_t(\dot B^{\frac n2-1}_{2,1})},
$$
hence it is also bounded by $C\|(b,\vu,j_0)\|_{\cY^1(t)}^2.$
\medbreak
This enables  to conclude that we do have for all $t\in\R_+$
$$
\|(b,\vu,j_0)\|_{\cY^1(t)}\leq C\bigl(\|(b,\vu,j_0)(0)\|_{\cX^1}+\|(b,\vu,j_0)\|_{\cY^1(t)}^2\bigr).
$$
This obviously yields \eqref{eq:estnoneq1}   if \eqref{eq:smallnoneq1} is fulfilled.

%%%%%%%%%%%%%%%%%%%%%%%%%%%%%%%%%%%%%%%%%%%%%%

\subsubsection*{Proof of uniqueness}

It works the same as for the standard barotropic Navier-Stokes equations: 
we look at the system satisfied by  the difference $(\db,\du,\dj_0)$ between
two solutions $(b^1,\vu^1,j_0^1)$ and $(b^2,\vu^2,j_0^2)$ of \eqref{eq:noneq1}, namely
(denoting $K_i=1+k_i$ for $i=1,2,3,4$)
$$
\left\{\begin{array}{l}\!\d_t\db+\vu^2\cdot\nabla\db=-\du\cdot\nabla b^1+(K_1(b^1)-K_1(b^2))\div\vu^2-K_1(b^1)\div\du,\\[1ex]
\!\d_t\du+\vu^2\!\cdot\!\nabla\du+\du\!\cdot\!\nabla\vu^1-(K_2(b^2)\!-\!K_2(b^1))\cA\vu^2-K_2(b^1)\cA\du
+(K_3(b^2)\!-\!K_3(b^1))\nabla b^2\\
\hspace{3.4cm}+K_3(b^1)\nabla\db +n^{-1}(K_4(b^2)\!-\!K_4(b^1))\nabla j_0^1+n^{-1}K_4(b^2)\nabla\dj_0=\vc 0,\\[1ex]
\!\d_t\dj_0+\wt\kappa\bigl(\dj_0-\db-\frac1{\wt m}\Delta\dj_0\bigr)=0.
\end{array}
\right.
$$
Now, exactly as for the barotropic Navier-Stokes equations, 
it is possible to bound $\db,$ $\du$ and $\dj_0$ just resorting
to basic estimates for the transport and heat equations. 
However, the hyperbolic  nature of the first equation forces us to estimate $(\db,\du,\dj_0)$ with one less derivative, 
namely in 
$$
L^\infty(0,T;\dot B^{\frac n2-1}_{2,1})\times 
\bigl(L^\infty(0,T;\dot B^{\frac n2-2}_{2,1})\cap L^1(0,T;\dot B^{\frac n2}_{2,1})\bigr)^n
\times L^1(0,T;\dot B^{\frac n2-1}_{2,1}).
$$
In dimension $n=3$ combining estimates for the transport and the heat equation allows to get 
uniqueness on a small time interval, then on the whole $\R_+$ by induction.
In dimension $n=2,$ this is slightly more involved as some product laws do not work correctly
if estimating $(\db,\du,\dj_0)$ in the above space (some regularity exponents become too negative). 
Nevertheless this may be overcome by combining logarithmic interpolation and Osgood lemma
(see e.g. \cite{D2} for more details). 
This completes the proof of the theorem.\qed

\begin{Remark}
If $0<m<+\infty$ then one may alternately assume that $j_0$ is in $\dot B^{\frac n2}_{2,1}.$
Taking advantage of the parabolic smoothing given by the equation for $j_0,$ 
it is not difficult to get a solution $(b,\vu,j_0)$ with 
$$
\wt m\int_{\R_+}\|j_0\|_{\dot B^{\frac n2+2}_{2,1}}\leq C\Bigl(\|b_0\|_{\dot B^{\frac n2-1}_{2,1}}
+\nu\|b_0\|_{\dot B^{\frac n2}_{2,1}}
+\|\vu_0\|_{\dot B^{\frac n2-1}_{2,1}}+\|j_0\|_{\dot B^{\frac n2}_{2,1}}\Bigr)\cdotp
$$
\end{Remark}

%%%%%%%%%%%%%%%%%%%%%%%%%%%%%%%%%%%%%%%%

\subsection{Weak convergence}

Here we justify the weak convergence of \eqref{eq:NSrad} to \eqref{eq:noneq1} or  \eqref{eq:noneq2} under the assumption that
$\liminf\eps^2\cL_s>0$ and that $\cL$ tends to $\frac{\kappa\eps}{n\nu}$  for some $\kappa>1.$
  \begin{Theorem}\label{th:diff-lim1}
 Let the family of data $(b_0^\eps,\vu_0^\eps,j_{0,0}^\eps,\vc j_{1,0}^\eps)_{0<\eps<1}$ satisfy
 Condition \eqref{eq:small-diff}.  Assume in addition that
 \begin{equation}\label{eq:l-diff0}\cL^2\cL_s\nu^2\to m \in(0,+\infty]\quad\hbox{and}\quad\frac{n\nu\cL}\ep\to\kappa\in(1,+\infty).
 \end{equation}
 Then  the global solution  $(b^\eps,\vu^\eps,j_0^\eps,\vc j_1^\eps)$ given  by Theorem \ref{th:diff1} satisfies
$$
\vc j_1^\eps\to\vc 0\quad\hbox{in}\quad L^1(\R_+;\dot B^{\frac n2-1}_{2,1}+\dot B^{\frac n2}_{2,1}),
$$
and, if     $(b^\eps_0,\vu^\eps_0,j_{0,0}^\eps)\rightharpoonup(b_0,\vu_0,j_{0,0})$  
then  $(b^\eps,\vu^\eps,j_0^\eps)$  converges  weakly to the unique solution $(b,\vu,j_0)$ of \eqref{eq:noneq1}
supplemented with initial data $(b_0,\vu_0,j_{0,0})$. 
 \end{Theorem}
\bProof 
{}From \eqref{eq:unif-diff}   we gather that
$$
(\vc\fj_1^\eps)^{\ell,\cL\cM}=\cO(\cM^{-1})\quad\hbox{and}\quad
(\vc j_1^\eps)^{h,\cL\cM}=\cO(\cM^{-1})\ \hbox{ in }\  L^1(\R_+;\dot B^{\frac n2-1}_{2,1}).
$$
Therefore, taking advantage of   the boundedness of the low frequencies of
 $\nabla b^\eps$ and $\nabla j_0^\eps$ in $L^1(\R_+;\dot B^{\frac n2}_{2,1})$,  and of the fact that
 $$
\vc j_1^\ep= \vc\fj_1^\eps-\frac{1}{\cL\cM}\nabla j_0^\eps+\frac1{\cL\cL_s\cM}\nabla b^\eps,$$   we get
 \begin{equation}\label{eq:l-diff1}
 \vc j_1^\eps=\cO(\eps)\quad\hbox{in}\quad  L^1(\R_+;\dot B^{\frac n2-1}_{2,1}+\dot B^{\frac n2}_{2,1}).
 \end{equation}
 Next, we observe that \eqref{eq:unif-diff} implies that
 $(b^\eps)$ and $(\vu^\eps)$ are bounded in 
 $L^\infty(\R_+;\dot B^{\frac n2-1}_{2,1}\cap\dot B^{\frac n2}_{2,1})\cap L^1(\R_+;\dot B^{\frac n2+1}_{2,1}+\dot B^{\frac n2}_{2,1})$
 and $L^\infty(\R_+;\dot B^{\frac n2-1}_{2,1})\cap L^1(\R_+;\dot B^{\frac n2+1}_{2,1}),$ respectively. 
 Note that this implies that $\vu^\eps$ is bounded in $L^2(\R_+;\dot B^{\frac n2}_{2,1}).$
 Because
 $$
 \d_t b^\eps=-\vu^\eps\cdot\nabla b^\eps-k_1(b^\eps)\div\vu^\eps,
 $$
 and the product maps $\dot B^{\frac n2-1}_{2,1}\times\dot B^{\frac n2}_{2,1}$ in $\dot B^{\frac n2-1}_{2,1},$ we thus 
 get in addition that $\d_tb^\eps$ is bounded in $L^2(\R_+;\dot B^{\frac n2-1}_{2,1}),$
 and thus $(b^\eps)$ is bounded in $\cC^{\frac12}(\R_+;\dot B^{\frac n2-1}_{2,1}).$
 Interpolating with the bound in $\cC_b(\R_+;\dot B^{\frac n2}_{2,1}),$ we 
 thus have $(b^\eps)$ bounded in $\cC^{\frac\alpha2}(\R_+;\dot B^{\frac n2-\alpha}_{2,1})$ for all $\alpha\in[0,1].$
  Then combining locally compact Besov embeddings and Ascoli theorem allows to conclude that
 there exists $b$ in  $L^\infty(\R_+;\dot B^{\frac n2-1}_{2,1}\cap\dot B^{\frac n2}_{2,1})\cap L^1(\R_+;\dot B^{\frac n2+1}_{2,1}+\dot B^{\frac n2}_{2,1})$
  and a sequence $(\eps_k)_{k\in\N}$ going
 to $0$ so that, for all $\phi\in\cS$ and all $\alpha\in(0,1]$
 \begin{equation}\label{eq:l-diff2}
 \phi\, b^{\eps_k}\longrightarrow \phi\, b\quad\hbox{in}\quad L^\infty_{loc}(\R_+;\dot B^{\frac n2-\alpha}_{2,1}).
 \end{equation}
 {}From \eqref{eq:unif-diff}, we readily get for some sequence $(\eps_k)_{k\in\N}$ tending to $0$
\begin{equation}\label{eq:l-diff4}
\vu^{\eps_k}\rightharpoonup\vu\quad\hbox{in}\quad L^\infty(\R_+;\dot B^{\frac n2-1}_{2,1})\cap L^1(\R_+;\dot B^{\frac n2+1}_{2,1})\quad\hbox{weak *},
\end{equation}
which, combined with  \eqref{eq:l-diff2} is clearly enough to pass to the limit in the mass equation. 
 Next, we see that \eqref{eq:unif-diff} implies  that  $(j_0^\eps)$ is bounded in $L^\infty(\R_+;\dot B^{\frac n2-1}_{2,1}).$
 Hence there exists $j_0\in L^\infty(\R_+;\dot B^{\frac n2-1}_{2,1})$ and a sequence $(\eps_k)_{k\in\N}$ going
 to $0$ so that 
 $$
 j_0^{\eps_k}\rightharpoonup j_0\quad\hbox{in}\quad 
  L^\infty(\R_+;\dot B^{\frac n2-1}_{2,1})\quad\hbox{weak *}.
 $$
 Because 
 \begin{equation}\label{eq:l-diff3}
\frac1\eps\div\vc j_1^\eps=-\frac\cL\ep\biggl(\frac1{\cL^2(1+\cL_s)}\biggr)\Delta j_0^\eps-\frac1{\cL(1+\cL_s)}\d_t\div\vc j_1^\eps,
 \end{equation}
 and \eqref{eq:l-diff1} implies that $\d_t\vc j_1^\ep\to0$ in the sense of distributions, 
 we deduce that 
 $$
 \frac1\eps\div\vc j_1^\eps\to-\frac{\kappa\nu}{nm}\Delta j_0\quad\hbox{in}\quad\cS'.
 $$
Note that the right-hand side is $0$ if $m=+\infty.$
 Therefore $(b,j_0)$ satisfies    the third line of \eqref{eq:noneq1} (case $m<+\infty$) or \eqref{eq:noneq2} (case $m=+\infty$).
 \smallbreak
 Let us finally pass to the limit in the second equation of \eqref{eq:NSrad}. 
The main difficulty is that, owing to the radiative term which is only bounded in a $L^1$-in-time type
space (namely $L^1(\R_+;\dot B^{\frac n2-1}_{2,1})$ or so), one cannot  take advantage of
some suitable bound of $\d_t\vu^\eps$ so as to glean some equicontinuity and then resort to Ascoli theorem. 
To overcome this, we use the fact that, owing to \eqref{eq:l-diff3}
$$\displaylines{
\d_t\biggl(\vu^\eps+\frac\ep n(1+k_4(b^\eps))\,\vc j_1^\eps\biggr)
=-\vu^\eps\cdot\nabla\vu^\eps+(1+k_2(b^\eps))\cA\vu^\eps\hfill\cr\hfill-(1+k_3(b^\eps))\nabla b^\eps-\frac1n\,(1+k_4(b^\eps))\,\nabla j_0^\eps
+\frac\ep n k'_4(b^\eps)\d_tb^\eps\,\vc j_1^\eps.}
$$
Now, because $(\vu^\eps)$ is bounded in the space $L^\infty(\R_+;\dot B^{\frac n2-1}_{2,1})\cap L^2(\R_+;\dot B^{\frac n2}_{2,1}),$
$(j^\eps_0)$ is bounded in $L^\infty(\R_+;\dot B^{\frac n2-1}_{2,1})$ 
and $(b^\eps)$ is bounded in $(L^2\cap L^\infty)(\R_+;\dot B^{\frac n2}_{2,1}),$ 
product laws in Besov spaces ensure that the first four terms of the r.h.s. 
are bounded in $L^2(\R_+;\dot B^{\frac n2-2}_{2,1})$ (only in  $L^2(\R_+;\dot B^{\frac n2-2}_{2,\infty})$ if $n=2$). 
The same property holds true for the last term for  $(\d_t b^\eps)$ is bounded in
$L^2(\R_+;\dot B^{\frac n2-1}_{2,1})$ and $(\vc j_1^\eps)$ is bounded in
$L^\infty(\R_+;\dot B^{\frac n2-1}_{2,1}).$
Using locally compact Besov embedding and Ascoli theorem, one can now conclude that
there exists some $\vc v$ in $L^\infty(\R_+;\dot B^{\frac n2-1}_{2,1})$ so that 
for all $\phi$ in $\cS$ and $\alpha\in(0,1),$ we have, up to extraction
$$
\phi\Bigl(\vu^{\eps}+\frac{\eps}n(1+k_4(b^\eps))\,\vc j_1^{\eps}\Bigr)\longrightarrow\phi\vc v\quad\hbox{in}\quad 
L^\infty_{loc}(\R_+;\dot B^{\frac n2-1-\alpha}_{2,1}).
$$
Of course, combining with \eqref{eq:l-diff1}, we discover  that $\vv=\vu.$ Hence 
we also have 
$$
\phi\vu^{\eps}\longrightarrow\phi\vc v\quad\hbox{in}\quad L^\infty_{loc}(\R_+\dot B^{\frac n2-1-\alpha}_{2,1})\quad\hbox{for all }\ \phi\in\cS.
$$
It is now easy to conclude that the second  line of \eqref{eq:noneq1} is  fulfilled by $(b,\vu,j_0).$
 \medbreak
 Finally, that the whole family $(b^\ep,\vu^\ep,j_0^\ep)$ (and not only subsequences) converges to $(b,\vu,j_0)$ stems from the fact that
 the solution to \eqref{eq:noneq1} or \eqref{eq:noneq2} is unique.  \qed

 \begin{Remark}
 It is also possible to justify the strong convergence of the solutions of 
 \eqref{eq:NSrad} to \eqref{eq:noneq1} or \eqref{eq:noneq2}
 using \eqref{eq:l-diff1} and performing the difference
 between $(b^\ep,\vu^\ep,j_0^\ep)$ and the solution $(b,\vu,j_0)$ to the limit system. 
 Again, taking advantage of the decay properties of $\fj_1^\ep$ is crucial. 
 Note however that, exactly as in the proof of uniqueness, 
 owing to the hyperbolic nature of the density equation, one cannot
 prove the strong convergence in the solution space. 
 There is a loss of one derivative that may be partially compensated by
 combining with uniform estimates. 
 As we do not think  this approach to bring much  compared to weak compactness, we leave the details to the reader. 
 \end{Remark}
% \begin{Remark}Note that in the case of initial data independent of $\ep,$  the smallness condition \eqref{eq:small-diff} is satisfied whenever
%$$\|b_0\|_{\dot B^{\frac n2-1}_{2,1}}+ \nu\|b_0\|_{\dot B^{\frac n2}_{2,1}}+ \|\vu_0\|_{\dot B^{\frac n2-1}_{2,1}}+ \|j_{0,0} \|_{\dot B^{\frac n2-1}_{2,1}}+ \|\vc j_{1,0} \|_{\dot B^{\frac n2-1}_{2,1}}\leq\eta_0\,\nu$$
%for some small enough $\eta_0.$\end{Remark}

 %%%%%%%%%%%%%%%%%%%%%%%%%%%%%%%%%%%%%%%%%%%%%%%

\section{The equilibrium diffusion regime}

This section is devoted to the mathematical justification of 
the \emph{equilibrium diffusion regime} given by \eqref{eq:eq2}. To avoid useless 
technicality, we focus on the case where 
\begin{equation}\label{eq:eqass}
 \cL\to+\infty\quad\hbox{and}\quad \ep\cL\cM\approx1.
 \end{equation}  
 
 \subsection{Linear estimates}

Let us gather the estimates we proved for \eqref{eq:diff0} for the above asymptotics in Section \ref{s:linear}. 

Regarding low frequencies, one may combine  \eqref{eq:diffulf1} and \eqref{eq:diffulf1b}  to get
\begin{multline}\label{eq:eqlf}
|(\wh b,\wh d,\wh j_0,\wh j_1)(t)|+\rho^2\int_0^t|(\wh b,\wh d)|\,d\tau
+\frac{\wt\cL}\ep\int_0^t|\fj_0|\,d\tau+\frac{\wt\cL\cM}\ep\int_0^t|\fj_1|\,d\tau\\
\leq C|(\wh b,\wh d,\wh j_0,\wh j_1)(0)|\quad\hbox{for}\quad 0\leq\rho\leq\sqrt{1+n^{-1}},
\end{multline}
with $\wt\cL:=\nu\cL,$ $
\ \wh\fj_0:=\wh j_0-\sqrt n \wh b-\sqrt n\,\frac\ep{\wt\cL}\rho\wh d\quad\hbox{and}\quad
\wh\fj_1=\wh j_1-\frac\rho{\wt\cL\cM}\wh b.$
\medbreak
For middle frequencies, we have according to \eqref{eq:diff-hf2} and \eqref{eq:diff-hf6} 
 \begin{multline}\label{eq:eqmf}
|(\rho\wh b,\wh d,\wh j_0,\wh j_1)(t)|+\int_0^t|\rho\wh b|\,d\tau+\rho^2\int_0^t|\wh d|\,d\tau+
\rho\int_0^t|\wh j_0|\,d\tau+\wt\cL\cM\int_0^t|\wh j_1|\,d\tau\\
\leq C|(\rho\wh b,\wh d,\wh j_0,\wh j_1)(0)|
\quad\hbox{for}\quad \sqrt{2/n}\leq \rho\leq c\wt\cL\cM,
\end{multline}
and   \eqref{eq:hfdiff3a} gives, if $\cM$ is large enough  
\begin{multline}\label{eq:eqhf}
|(\rho\wh b,\wh d,\wh j_0,\wh j_1)(t)|+\rho^2\int_0^t|\wh d|\,d\tau+\rho\int_0^t|\wh b|\,d\tau+\frac{\wt\cL\cM}\ep\int_0^t|(\wh j_0,\wh j_1)|\,d\tau\\
\leq C|(\rho\wh b,\wh d,\wh j_0,\wh j_1)(0)|\quad\hbox{for  }\ \rho\geq c\wt\cL\cM.
\end{multline}
If $\cM$ is bounded then we must assume that $\rho\geq C_1\wt\cL\cM$ for some $C_1>c.$
However, we have  \eqref{eq:diff4} and $\cM$ bounded implies that $\ep\wt\cL\approx1.$
Therefore \eqref{eq:eqmf} is satisfied up to $\rho\leq C_1\wt\cL\cM.$ 
\medbreak
For the whole system  \eqref{eq:diff} with nonzero source terms $f$ and $\vc g$, we thus 
obtain (taking slightly abusively $c=C_1=1$ for notational simplicity)
\begin{multline}\label{eq:eqestlin}
\|(\vu,j_0,\vc j_1)(t)\|_{\dot B^s_{2,1}} +\|b(t)\|_{\dot B^s_{2,1}}^{\ell,1}+\|b(t)\|_{\dot B^{s+1}_{2,1}}^{h,1}
+\int_0^t\|\vu\|_{\dot B^{s+2}_{2,1}}\,d\tau
+\frac{\wt\cL}\ep\int_0^t\|\fj_0\|_{\dot B^s_{2,1}}^{\ell,1}\,d\tau\\
+\frac{\wt\cL\cM}\ep\int_0^t\|\vc\fj_1\|_{\dot B^s_{2,1}}^{\ell,1}\,d\tau
+\int_0^t\|(b,j_0,\vc j_1)\|_{\dot B^{s+2}_{2,1}}^{\ell,1}\,d\tau
+\int_0^t\|j_0\|_{\dot B^{s+1}_{2,1}}^{m,1,\wt\cL\cM}\,d\tau
+\wt\cL\cM\int_0^t\|\vc j_1\|_{\dot B^s_{2,1}}^{m,1,\wt\cL\cM}\,d\tau\\
+\int_0^t\|b\|^{h,1}_{\dot B^{s+1}_{2,1}}\,d\tau
+\frac{\wt\cL\cM}\ep\int_0^t\|(j_0,\vc j_1)\|_{\dot B^s_{2,1}}^{h,\wt\cL\cM}\,d\tau
\lesssim \|(\vu,j_0,\vc j_1)(0)\|_{\dot B^s_{2,1}} \\+\|b(0)\|_{\dot B^s_{2,1}}^{\ell,1}
+\|b(0)\|_{\dot B^{s+1}_{2,1}}^{h,1}+\int_0^t\bigl(\|f\|_{\dot B^s_{2,1}}^{\ell,1}+\|f\|_{\dot B^{s+1}_{2,1}}^{h,1}+\|\vc g\|_{\dot B^s_{2,1}}\bigr)\,d\tau,
\end{multline}
with 
$$
\fj_0:=j_0-\sqrt n b-\sqrt n\,\frac\ep{\wt\cL}\div\vu\quad\hbox{and}\quad
\vc\fj_1=\vc j_1+\frac1{\wt\cL\cM}\nabla b.
$$
Back to the original variables, that linear analysis induces us to introduce  the following norms 
$$\displaylines{
\|(b,\vu,j_0,\vc j_1)\|_{\wt X^\nu_\eps}:=
\|b\|_{\dot B^{\frac n2-1}_{2,1}}^{\ell,\nu^{-1}}
+\nu\|b\|_{\dot B^{\frac n2}_{2,1}}^{h,\nu^{-1}}
+\|\vu\|_{\dot B^{\frac n2-1}_{2,1}}
+\|(j_0,\vc j_1)\|_{\dot B^{\frac n2-1}_{2,1}}\quad\hbox{and}\quad}
$$
$$\displaylines{
\|(b,\vu,j_0,\vc j_1)\|_{\wt Y^\nu_\eps}:=\sup_{t\geq0} \|(b,\vu,j_0,\vc j_1)(t)\|_{\wt X^\nu_\eps}
+\nu\int_{\R_+}\!\!\biggl(\|(b,j_0,\vc j_1)\|_{\dot B^{\frac n2+1}_{2,1}}^{\ell,\nu^{-1}}
+\|\vu\|_{\dot B^{\frac n2+1}_{2,1}}\biggr)d\tau\hfill\cr\hfill
+\int_{\R_+}\!\biggl(\|b\|_{\dot B^{\frac n2}_{2,1}}^{h,\nu^{-1}}+\frac{\cL\cM}\ep\|\vc\fj_1\|_{\dot B^{\frac n2-1}_{2,1}}^{\ell,\nu^{-1}}+\frac{\cL}\ep\|\fj_0\|_{\dot B^{\frac n2-1}_{2,1}}^{\ell,\nu^{-1}}\biggr)\,d\tau
\hfill\cr\hfill+\int_{\R_+}\biggl(\|j_0\|_{\dot B^{\frac n2}_{2,1}}^{m,\nu^{-1},\cL\cM}+\cL\cM\|\vc j_1\|_{\dot B^{\frac n2-1}_{2,1}}^{m,\nu^{-1},\cL\cM}
  \!+\!\frac{\cL\cM}\ep\|(j_0,\vc j_1)\|_{\dot B^{\frac n2-1}_{2,1}}^{h,\cL\cM}\biggr)d\tau,}
$$
with $\fj_0:=j_0-b-\frac\eps\cL\div\vu$ and $\vc\fj_1:=\vc j_1+\frac1{\cL\cM}\nabla b.$
\medbreak
We denote by $\wt X^\nu_\eps$ and $\wt Y^\nu_\eps$ the corresponding functional spaces
(where time continuity is imposed rather than just boundedness).
Of course, we still have 
$$
\|(b',\vu',j'_0,\vc j'_1)\|_{\wt X^1_\eps}=\nu^{-1}\|(b,\vu,j_0,\vc j_1)\|_{\wt X^\nu_\eps}
\quad\hbox{and}\quad
\|(b',\vu',j'_0,\vc j'_1)\|_{\wt Y^1_\eps}=\nu^{-1}\|(b,\vu,j_0,\vc j_1)\|_{\wt Y^\nu_\eps},
$$
through the change of variables \eqref{eq:nu}, if we replace  $\cL$ by  $\wt\cL$ in the left-hand side. 

%%%%%%%%%%%%%%%%

\subsection{The paralinearized equations}

In the equilibrium diffusion limit case the estimates for the paralinearized system 
\begin{equation}\label{eq:par-diff2}
\left\{\begin{array}{l}
 \partial_t b+ T_{\vv}\cdot\nabla b+ \div \vu=F,\\[1ex]
  \partial_t \vu+ T_{\vv}\cdot\nabla\vu -\cA\vu + \nabla b
    -\frac{\cL(1+\cL_s)}n\vc j_1=\vec G,\\[1ex]
 \partial_t j_0+ \frac{\mbox{div}\,\vec j_1}{n\eps}
+\frac\cL\eps (j_0-b)=0,\\[1ex]
 \partial_t \vec j_1 + \frac{ \nabla j_0}\eps+\frac{\cL(1+\cL_s)}\eps\vec j_1=\vec 0.
 \end{array}\right.
\end{equation}
recast as follows
\begin{Proposition}\label{p:para-diff2} For any 
smooth solution $(b,\vu,j_0,\vc j_1)$ we have the following a priori estimate for \eqref{eq:par-diff}
$$\displaylines{
\|(b,\vu,j_0,\vc j_1)\|_{\wt Y^\nu_\eps(t)}\leq C\biggl(\|(b,\vu,j_0,\vc j_1)(0)\|_{\wt X^1_\eps}
+\int_0^t\|\nabla\vv\|_{L^\infty}\|(b,\vu,j_0,\vc j_1)\|_{\wt X^\nu_\eps}\,d\tau\hfill\cr\hfill
+\int_0^t\|(\nabla F,\vc G)\|_{\dot B^{\frac n2-1}_{2,1}}^{h,1}\,d\tau
+\int_0^t\|(F-T_{\vv}\cdot\nabla b,\vc G-T_{\vv}\cdot\nabla\vu)\|_{\dot B^{\frac n2-1}_{2,1}}^{\ell,1}\,d\tau
+\int_0^t\|T_{\vv}\cdot\nabla b-F\|_{\dot B^{\frac n2-1}_{2,1}}^{m,1,\wt\cL\cM}\,d\tau
\biggr)\cdotp}
$$
\end{Proposition}
\bProof
Except in the middle frequencies range, the proof goes along the lines of the corresponding 
result in the non-equilibrium case. 
Still  assuming  that $\nu=1$ and replacing  $\cL$ with $\wt\cL=\nu\cL$ then, working directly on the localized paralinearized system \eqref{eq:par-diff2},
and  combining Inequalities \eqref{eq:eqlf} to \eqref{eq:eqhf} with  estimates for the para-convection terms gives

\noindent\emph{1. Low frequencies: $2^k\leq C_1.$}

$$
\displaylines{\|\ddk(b,\vu,j_0,\vc j_1)(t)\|_{L^2}+2^{2k}\int_0^t
\|\ddk(b,\vu,j_0,\vc j_1)\|_{L^2}\,d\tau
+\frac{\wt\cL\cM}\ep\int_0^t\|\ddk\vc\fj_1\|_{L^2}\,d\tau\hfill\cr\hfill
+\frac{\wt\cL}\ep\int_0^t\|\ddk\fj_0\|_{L^2}\,d\tau
\lesssim \|\ddk(b,\vu,j_0,\vc j_1)(0)\|_{L^2}
\hfill\cr\hfill+\int_0^t\|\ddk(F-T_{\vv}\cdot\nabla b)\|_{L^2}\,d\tau
+\int_0^t\|\ddk(\vc G-T_{\vv}\cdot\nabla\vu)\|_{L^2}\,d\tau.}
$$

\noindent\emph{2. Medium frequencies: $C_1\leq 2^k\leq c\wt\cL\cM.$}

One has to keep in mind that in order to derive \eqref{eq:eqmf} from \eqref{eq:diff-hf2} and \eqref{eq:diff-hf6}, one has
to consider the system that is fulfilled by $(b,\vu,\zeta_0,\vc j_1)$ with $\zeta_0:=j_0-\sqrt n\, b.$
In particular, a part of the the paraconvection term of $b$ enters in the equation for $\zeta_0$ as we have
$$
\d_t\zeta_0+\frac{\wt\cL}\ep\zeta_0+\frac1{\ep\sqrt n}\div\vc j_1-\sqrt n\div  \vu=\sqrt n(T_\vv\cdot\nabla b-F).
$$ 
Therefore, following  the computations  leading to  \eqref{eq:diff-hf2} and \eqref{eq:diff-hf6}, and using Lemma 4.1 in \cite{DD} 
to bound the convection terms coming from the equations for $b$ and $\vu,$ we end up with 
$$
\displaylines{\|\ddk(\nabla b,\vu,j_0,\vc j_1)(t)\|_{L^2}+2^{2k}\int_0^t\|\ddk\vu\|_{L^2}\,d\tau
+2^k\int_0^t\|(\ddk b,\ddk j_0)\|_{L^2}\,d\tau\hfill\cr\hfill
+\cL(1+\cL_s)\int_0^t\|\ddk\vc j_1\|_{L^2}\,d\tau
\lesssim \|\ddk(\nabla b,\vu,j_0,\vc j_1)(0)\|_{L^2}
+\int_0^t\|\ddk(\nabla F,\vc G)\|_{L^2}\,d\tau
\hfill\cr\hfill+\int_0^t\|\ddk(T_\vv\cdot\nabla b- F)\|_{L^2}\,d\tau+\sum_{k'\sim k}\int_0^t\|\nabla\vv\|_{L^\infty}\|\dot\Delta_{k'}(\nabla b,\vu)\|_{L^2}\,d\tau.}
$$

\noindent\emph{3. High frequencies: $2^k\geq c\wt\cL\cM.$}
We get
$$
\displaylines{\|\ddk(\nabla b,\vu,j_0,\vc j_1)(t)\|_{L^2}+\int_0^t\bigl(2^{2k}\ddk\vu\|_{L^2}
+2^k\|\ddk b\|_{L^2}\bigr)\,d\tau+{\wt\cL\cM}\ep\int_0^t\|\ddk(j_0,\vc j_1)\|_{L^2}\,d\tau\hfill\cr\hfill
\lesssim \|\ddk(\nabla b,\vu, j_0, \vc j_1)(0)\|_{L^2}
+\int_0^t\|\ddk(\nabla F,\vc G)\|_{L^2}\,d\tau
\hfill\cr\hfill+\sum_{k'\sim k}\int_0^t\|\nabla\vv\|_{L^\infty}\|\dot\Delta_{k'}(\nabla b,\vu)\|_{L^2}\,d\tau.}
$$
Putting together all those inequalities completes the proof.
\qed

%%%%%%%%%%%%%%%%%%%%%%%%%%%%%%%%%%%%

\subsection{A global existence result}

Our global existence result  with uniform estimates reads
\begin{Theorem}\label{th:diff2}
There exists a positive constant $\eta$ depending only on $\mu/\nu,$ $n$ 
and on the pressure law such that if $\eps\in(0,1)$ and 
the data $(b_0^\eps,\vu_0^\eps,j_{0,0}^\eps,\vc j_{1,0}^\eps)$ satisfy 
\begin{equation}\label{eq:small-diff2}
\|(b_0^\eps,\vu_0^\eps,j_{0,0}^\eps,\vc j_{1,0}^\eps)\|_{\wt X_\eps^\nu}\leq \eta\nu,
\end{equation}
then System  \eqref{eq:NSrad} admits a unique global 
solution $(b^\eps,\vu^\eps,j_0^\eps,\vc j_1^\eps)$ in $\wt Y_\eps^\nu.$
In addition, we have
\begin{equation}\label{eq:unif-diff2}
\|(b^\eps,\vu^\eps,j_0^\eps,\vc j_1^\eps)\|_{\wt Y_\eps^\nu}\leq C
\|(b_0^\eps,\vu_0^\eps,j_{0,0}^\eps,\vc j_{1,0}^\eps)\|_{\wt X_\eps^\nu}.
\end{equation}
\end{Theorem}
 The proof relies on Proposition \ref{p:para-diff2}. 
 Note in particular that the `new' last term in the estimate of \eqref{p:para-diff2} does not entail a loss of derivative as we simply have
 $$
  \|T_{\vv}\cdot\nabla b-F\|_{L_t^1(\dot B^{\frac n2-1}_{2,1})}^{m,1,\wt\cL\cM}\leq 
  \|\vv\cdot\nabla b\|_{L^1_t(\dot B^{\frac n2-1}_{2,1})}\lesssim \|\vv\|_{L^2_t(\dot B^{\frac n2}_{2,1})}\|b\|_{L^2_t(\dot B^{\frac n2}_{2,1})}.
 $$
 The rest of the proof  works exactly the same as in the non-equilibrium case.\qed

%%%%%%%%%%%%%%%%%%%%%%%%%%%%%%%%%%%%%%

\subsection{Weak convergence}

Here we justify weak convergence to \eqref{eq:eq2}  when  assumption \eqref{eq:eqass} is fulfilled.
 \begin{Theorem}\label{th:diff-lim2}
 Let the family of data $(b_0^\eps,\vu_0^\eps,j_{0,0}^\eps,\vc j_{1,0}^\eps)_{0<\eps<1}$ satisfy
 \eqref{eq:small-diff2}. 
Then  the global solution  $(b^\eps,\vu^\eps,j_0^\eps,\vc j_1^\eps)$ in $\wt Y^\nu_\eps$ given 
by Theorem \ref{th:diff2} satisfies
$$
\vc j_1^\eps\to\vc 0\quad\hbox{in}\quad L^1(\R_+;\dot B^{\frac n2-1}_{2,1}+\dot B^{\frac n2}_{2,1}),
$$
and if $(b_0^\ep,\vu_0^\ep)\rightharpoonup(b_0,\vu_0)$ then 
$(b^\eps,\vu^\eps,j_0^\eps)$  converges weakly to $(b,\vu,b)$ where $(b,\vu)$ stands for the unique solution of
\begin{equation}\label{eq:eqlim}\left\{
\begin{array}{l}
 \partial_t b+ \vu\cdot\nabla b+ k_1(b)\div \vu=0,\\[1ex]
\partial_t \vu + \vu\cdot\nabla\vu -k_2(b)\cA\vu +  \bigl(k_3(b)+n^{-1}k_4(b)\bigr)\nabla b=\vec 0.
\end{array}\right.
\end{equation}supplemented with initial data $(b_0,\vu_0).$
 \end{Theorem}
\bProof 
Let  $\vc\fj_1^\eps=\vc j_1^\eps+\frac{\nabla b^\eps}{\cL\cM}\cdotp$
{}From \eqref{eq:unif-diff2}, we have
$$
\frac{\cL\cM}\ep\|\vc\fj_1^\eps\|^{\ell,\nu^{-1}}_{L^1(\R_+;\dot B^{\frac n2-1}_{2,1})}
+\nu\|\nabla b^\eps\|_{L^1(\R_+;\dot B^{\frac n2}_{2,1})}^{\ell,\nu^{-1}}
+\cL\cM\|\vc j_1^\eps\|_{L^1(\R_+;\dot B^{\frac n2-1}_{2,1})}^{h,\nu^{-1}}\leq C\eta\nu.
$$
Hence, given \eqref{eq:eqass}, we deduce that 
 \begin{equation}\label{eq:l-diff4b}
\vc j_1^\eps=\cO(\ep)\quad\hbox{in}\quad L^1(\R_+;\dot B^{\frac n2-1}_{2,1}\!+\!\dot B^{\frac n2}_{2,1}).
\end{equation}
Using the equation of $\vc j_1^\ep,$ this gives 
\begin{equation}\label{eq:l-diff4a}
\nabla j_0^\eps+\cL(1+\cL_s)\vc j_1^\eps\to 0
\quad\hbox{in the sense of distributions.}
\end{equation}
 As in the non-equilibrium case,  \eqref{eq:unif-diff2} implies that the families 
 $(b^\eps)$ and $(\vu^\eps)$ are bounded in 
 $L^\infty(\R_+;\dot B^{\frac n2-1}_{2,1}\cap\dot B^{\frac n2}_{2,1})\cap L^1(\R_+;\dot B^{\frac n2+1}_{2,1}+\dot B^{\frac n2}_{2,1})$
 and $L^\infty(\R_+;\dot B^{\frac n2-1}_{2,1})\cap L^1(\R_+;\dot B^{\frac n2+1}_{2,1}),$ respectively.
 Hence $(\d_tb^\eps)$ is bounded in $L^2(\R_+;\dot B^{\frac n2-1}_{2,1})$ 
 and we can thus deduce  that there exists $b$ in  $L^\infty(\R_+;\dot B^{\frac n2-1}_{2,1}\cap\dot B^{\frac n2}_{2,1})\cap L^1(\R_+;\dot B^{\frac n2+1}_{2,1}+\dot B^{\frac n2}_{2,1})$
  and a sequence $(\eps_k)_{k\in\N}$ going
 to $0$ so that, for all $\phi\in\cS$ and all $\alpha\in(0,1]$
 \begin{equation}\label{eq:l-diff5b}
 \phi\, b^{\eps_k}\longrightarrow \phi\, b\quad\hbox{in}\quad L^\infty(\R_+;\dot B^{\frac n2-\alpha}_{2,1}).
 \end{equation}
For $(\vu^\eps),$ we still  have the weak convergence result given by \eqref{eq:l-diff4}, 
which suffices to pass to the limit in the mass equation. 
 \medbreak
 Next, we observe that \eqref{eq:unif-diff2} implies that $(j_0^\eps)$ is bounded in $L^1(\R_+;\dot B^{\frac n2-1}_{2,1}\!+\!\dot B^{\frac n2+1}_{2,1}).$ Hence, there exists 
 a sequence $(\eps_k)_{k\in\N}$ going
 to $0$ so that  $j_0^{\eps_k}\rightharpoonup j_0$ in the sense of distributions. 
 Moreover, we have 
  $$
 j_0^\eps-b^\eps=-\cL^{-1}n^{-1}\div\vc j_1^\eps-\eps\cL^{-1}\d_tj_0^\eps.
 $$
 Remembering \eqref{eq:l-diff4b} and \eqref{eq:eqass}, we see that the first
 term of the r.h.s. is $\cO(\ep)$ in a suitable space.  The second one also tends to $0$ in the sense
 of distributions as   $\eps\cL^{-1}\to0.$  Hence
 \begin{equation}\label{eq:diff6}
 j_0=b.
 \end{equation}
 In order to pass to the limit in the velocity equation, we proceed as in the non-equilibrium 
 case. First we use \eqref{eq:l-diff4} and next, the fact that 
 $$
\d_t\biggl(\vu^\eps+\frac\eps n\, k_4(b^\eps)\,\vc j_1^\eps\biggr)
=-\vu^\eps\cdot\nabla\vu^\eps+k_2(b^\eps)\cA\vu^\eps-k_3(b^\eps)\nabla b^\eps
+\frac\eps nk'_4(b^\eps)\d_tb^\eps\,\vc j_1^\eps-\frac1n\,k_4(b^\eps)\,\nabla j_0^\eps.
$$
Taking advantage of \eqref{eq:unif-diff2} and of product laws in Besov spaces, 
we readily obtain that the four first terms of the r.h.s. are 
 bounded in $L^2(\R_+;\dot B^{\frac n2-2}_{2,1})$ (or only in  $L^2(\R_+;\dot B^{\frac n2-2}_{2,\infty})$ if $n=2$).
 To handle the last term, we observe that according to \eqref{eq:unif-diff2}, 
 $\nabla j_0^\eps$ is bounded in $L^\infty(\R_+;\dot B^{\frac n2-2}_{2,1})$
 hence  the term $k_4(b^\eps)\nabla j_0^\eps$ is bounded 
in $L^\infty(\R_+;\dot B^{\frac n2-2}_{2,1})$ (or $L^\infty(\R_+;\dot B^{\frac n2-2}_{2,\infty})$ if $n=2$).
\medbreak
As in the non-equilibrium case, it is now easy to conclude that  
there exists some $\vc v$ in $L^\infty(\R_+;\dot B^{\frac n2-1}_{2,1})$ so that 
for all $\phi$ in $\cS$ and $\alpha\in(0,1),$ we have
$$
\phi\Bigl(\vu^{\eps_k}+\frac{\eps^k}n\vc j_1^{\eps_k}\Bigr)\longrightarrow\phi\vc v\quad\hbox{in}\quad L^\infty_{loc}(\R_+;\dot B^{\frac n2-1-\alpha}_{2,1}).
$$
Of course, combining with \eqref{eq:l-diff4b}, this implies that $\vv=\vu,$
and  $(b,\vu)$  thus satisfies the second line of \eqref{eq:eqlim}.
 \medbreak
 Finally, that the whole family $(b^\eps,\vu^\eps,j_0^\eps,\vc j_1^\eps)$ converges to $(b,\vu,b,\vc 0)$ 
 stems from the fact that
 the solution to \eqref{eq:eqlim} is unique (note that it is just the standard barotropic Navier-Stokes equations
 with a \emph{modified} but still stable pressure law).  
 \qed

%%%%%%%%%%%%%%%%%%%%%%%%%%%%%%%%%%%%%%%%%%%%%%%

\section{The Poisson diffusion regime}

This section is devoted to the study of the asymptotics regime where
\begin{equation}\label{eq:poissonsas}
 \ep\ll\cL\lesssim\ep^{1/2}\quad\hbox{and}\quad \cL^2\cL_s\approx1.\end{equation}
 According to the formal computations of Section \ref{s:formal}, we expect the
 solutions of \eqref{eq:NSrad} to tend to those of the Navier-Stokes-Poisson 
 sytem \eqref{eq:eq1}. 
 
 The general scheme of the proof that we here propose is the same as in the study of the other asymptotics:
 we first perform a fine analysis of the linearized equations so as to check the long-time stability 
 and exhibit the quantities that are likely to be bounded uniformly when $\ep\to0,$ 
 then tackle the proof of the global existence. We rapidly justify that the limit system 
 is globally well-posed in a functional framework that is consistent with the analysis we  used
 for   \eqref{eq:NSrad}, and eventually take advantage of compactness 
 arguments so as to prove the expected convergence result. 
 As in the other  regimes, the fact that the limit system has a unique solution will 
 guarantee that the whole family of solutions to \eqref{eq:NSrad} converges to 
 the solution to \eqref{eq:eq1}.  
 
 \subsection{Linear analysis of \eqref{eq:NSrad} in the Poisson regime \eqref{eq:poissonsas}}
We here gather the estimates  for \eqref{eq:diff0} that have been obtained in Section \ref{s:linear}.
Recall that $\wt\cL:=\nu\cL.$ 
\subsubsection*{Small frequencies}
Using \eqref{eq:diffulf2}, \eqref{eq:diffulf9}, \eqref{eq:diffulf10}  and the fact that 
 $|(\wh b,\wh d,\wh j_0,\wh j_1)|\approx|(\wh\fb,\wh\fd,\wh\fj_0,\wh\fj_1)|$ 
 and that the last term in the original definition of $\wh\fj_0$  in \eqref{eq:chgdiffu2}
 has a negligible contribution with respect to $\wh\fj_1,$ we get 
\begin{multline}\label{eq:noneqPlf}
|(\wh b,\wh d,\frac\ep{\wt\cL}\wh j_0,\rho\wh j_0,\wh j_1)(t)|+\rho^2\int_0^t|(\wh b,\wh d,\wh j_0,\wh j_1)|\,d\tau
+\int_0^t|\wh\fj_0|\,d\tau+\frac{\wt\cL}\ep\int_0^t|\rho\wh\zeta_0|\\+\frac{\wt\cL\cM}{\ep}\int_0^t|\wh\fj_1|\,d\tau
\leq C|(\wh b,\wh d,\frac\ep{\wt\cL}\wh j_0,\rho\wh j_0,\wh j_1)(0)|\quad\hbox{for all }\ 0\leq\rho\leq C_1,
\end{multline}
with $\:\wh\fj_0:=\wh j_0-\sqrt n\,\wh b-\sqrt n\frac\ep{\wt\cL}\rho\wh d,$
$\:\wh\zeta_0:=\wh j_0-\frac{\sqrt n}{1+\frac{\rho^2}{n\wt\cL^2\cM}}\wh b\,$ and 
$\:\wh\fj_1:=\wh j_1-\frac\rho{\sqrt n\,\wt\cL\cM}\wh j_0+\frac{\rho\wh b}{\wt\cL\cL_s\cM}\cdotp$

\subsubsection*{Middle frequencies}

Combining \eqref{eq:diff-hf7} and the definition of $\wh\zeta_1$ versus that of $\wh\fj_1,$ we get
\begin{multline}\label{eq:noneqPmf}
|(\rho\wh b,\wh d,\wh j_0,\wh j_1)(t)|+\rho\int_0^t|\wh b|\,d\tau+\rho^2\Int_0^t|\wh d|\,d\tau
+\rho^2\frac{\wt\cL}\ep\Int_0^t|\wh j_0|\,d\tau\\+\frac{\wt\cL\cM}\ep\int_0^t|\wh\fj_1|\,d\tau
\lesssim  |(\rho\wh b,\wh d,\wh j_0,\wh j_1)(0)|
\quad\hbox{for }\ C_1\leq\rho\leq c\wt\cL\cM.
\end{multline}

\subsubsection*{Large frequencies}

Finally, using \eqref{eq:hfdiff3a}, we have
\begin{multline}\label{eq:noneqPhf}
|(\rho\wh b,\wh d,\wh j_0,\wh j_1)(t)|+\rho^2\int_0^t|\wh d|\,d\tau+\rho\int_0^t|\wh b|\,d\tau
+\frac{\wt\cL\cM}\ep\int_0^t|(\wh j_0,\wh j_1)|\,d\tau\\
\leq C|(\rho\wh b,\wh d,\wh j_0,\wh j_1)(0)|\quad\hbox{for  }\ \rho\geq c\wt\cL\cM.
\end{multline}

Therefore, localizing  \eqref{eq:diff} (with nonzero source terms $f$ and $\vc g$) according to Littlewood-Paley operator $\ddk,$ 
using \eqref{eq:Pu},  following the computations leading to the above three inequalities and using Fourier-Plancherel theorem, 
we end up with the following inequality for all $s\in\R$
\begin{multline}\label{eq:noneqPestlin}
\|(\vu,\vc j_1)(t)\|_{\dot B^s_{2,1}} +\|b(t)\|_{\dot B^s_{2,1}}^{\ell,1}+\|b(t)\|_{\dot B^{s+1}_{2,1}}^{h,1}
+\frac\ep{\wt\cL}\|j_0(t)\|_{\dot B^s_{2,1}}^{\ell,1}+\|j_0(t)\|_{\dot B^{s+1}_{2,1}}^{\ell,1}+\|j_0(t)\|_{\dot B^s_{2,1}}^{h,1}\\
+\int_0^t\|\vu\|_{\dot B^{s+2}_{2,1}}\,d\tau
+\int_0^t\|(b,j_0,\vc j_1)\|_{\dot B^{s+2}_{2,1}}^{\ell,1}\,d\tau
+\frac{\wt\cL}\ep\int_0^t\|\zeta_0\|_{\dot B^{s+1}_{2,1}}^{\ell,1}\,d\tau
+\int_0^t\|\fj_0\|_{\dot B^s_{2,1}}^{\ell,1}\,d\tau\\
+\frac{\wt\cL\cM}\ep\int_0^t\|\vc\fj_1\|_{\dot B^s_{2,1}}^{\ell,\wt\cL\cM}\,d\tau
+\frac{\wt\cL}\ep\int_0^t\|j_0\|_{\dot B^{s+2}_{2,1}}^{m,1,\wt\cL\cM}
+\int_0^t\|b\|^{h,1}_{\dot B^{s+1}_{2,1}}\,d\tau
+\frac{\wt\cL\cM}\ep\int_0^t\|(j_0,\vc j_1)\|_{\dot B^s_{2,1}}^{h,\wt\cL\cM}\,d\tau\\
\lesssim \|(\vu,\vc j_1)(0)\|_{\dot B^s_{2,1}} +\|b(0)\|_{\dot B^s_{2,1}}^{\ell,1}+\|b(0)\|_{\dot B^{s+1}_{2,1}}^{h,1}
+\frac\ep{\wt\cL}\|j_0(0)\|_{\dot B^s_{2,1}}^{\ell,1}+\|j_0(0)\|_{\dot B^{s+1}_{2,1}}^{\ell,1}\\
+\|j_0(0)\|_{\dot B^s_{2,1}}^{h,1}
+\int_0^t\bigl(\|f\|_{\dot B^s_{2,1}}^{\ell,1}+\|f\|_{\dot B^{s+1}_{2,1}}^{h,1}+\|\vc g\|_{\dot B^s_{2,1}}\bigr)\,d\tau,
\end{multline}
with  $\fj_0:=j_0-\sqrt n\,b-\sqrt n\,\frac\ep{\wt\cL}\,\div\vu,$
$$
\zeta_0:=j_0-\sqrt n\Bigl({\rm Id}-\frac1{n\wt\cL^2\cM}\Delta\Bigr)^{-1}b\quad\hbox{and}\quad 
\vc \fj_1:=\vc j_1+\frac1{\sqrt n\,\wt\cL\cM}\nabla j_0-\frac{1}{\wt\cL\cL_s\cM}\nabla b.$$

As in the previous sections, owing to the convection term in the equation for $b,$
the above inequality does not allow to prove the global existence 
for \eqref{eq:NSrad}, and one has to consider the paralinearized system \eqref{eq:par-diff}. 
Adapting the proof of Proposition \ref{p:para-diff}, we get the following
\begin{Proposition}\label{p:para-diff-poisson}  If the coefficients $\cL,$ $\cL_s$ and $\ep$ fulfill 
\eqref{eq:poissonsas} then for any smooth solution $(b,\vu,j_0,\vc j_1)$ 
to \eqref{eq:par-diff}, one has the following inequality
$$
\displaylines{\|(b,\vu,j_0,\vc j_1)(t)\|_{\check Y^\nu_\ep(t)}
\leq  C\biggl(\|(b,\vu,j_0,\vc j_1)(0)\|_{\check X^\nu_\eps}
+\int_0^t\|\nabla\vv\|_{L^\infty}\|(b,\vu,j_0,\vc j_1)\|_{\check X^\nu_\eps}\,d\tau\hfill\cr\hfill
+\int_0^t\|(\nabla F,\vc G)\|_{\dot B^{\frac n2-1}_{2,1}}^{h,1}\,d\tau
+\int_0^t\|(F-T_{\vv}\cdot\nabla b,\vc G-T_{\vv}\cdot\nabla\vu)\|_{\dot B^{\frac n2-1}_{2,1}}^{\ell,1}\,d\tau\biggr),}
$$
with $$\|(b,\vu,j_0,\vc j_1)\|_{\check X^\nu_\eps}:=\|(\vu,\vc j_1)\|_{\dot B^{\frac n2-1}_{2,1}}
 +\|b\|_{\dot B^{\frac n2-1}_{2,1}}^{\ell,\nu^{-1}}+\nu\|b\|_{\dot B^{\frac n2}_{2,1}}^{h,\nu^{-1}}
+\frac\ep{\cL\nu}\|j_0\|_{\dot B^{\frac n2-1}_{2,1}}^{\ell,\nu^{-1}}+
\nu\|j_0\|_{\dot B^{\frac n2}_{2,1}}^{\ell,\nu^{-1}}+\|j_0\|_{\dot B^{\frac n2-1}_{2,1}}^{h,\nu^{-1}},$$
and 
$$\displaylines{
\|(b,\vu,j_0,\vc j_1)(t)\|_{\check Y^\nu_\ep(t)}
:=\sup_{0\leq\tau\leq t}\|(b,\vu,j_0,\vc j_1)(\tau)\|_{\check X^\nu_\eps}\hfill\cr\hfill+\nu\int_0^t\bigl(\|\vu\|_{\dot B^{\frac n2+1}_{2,1}}
+\|(b,\vc j_1)\|_{\dot B^{\frac n2+1}_{2,1}}^{\ell,\nu^{-1}}
+{\textstyle\frac{\cL\nu}\ep}\|j_0\|_{\dot B^{\frac n2+1}_{2,1}}^{m,\nu^{-1},\cL\cM}\bigr)\,d\tau
+\frac\cL\ep\int_0^t\|\zeta_0\|_{\dot B^{\frac n2}_{2,1}}^{\ell,\nu^{-1}}\,d\tau\hfill\cr\hfill
+\nu^{-1}\int_0^t\|\fj_0\|_{\dot B^{\frac n2-1}_{2,1}}^{\ell,\nu^{-1}}\,d\tau
+\frac{\cL\cM}\ep\int_0^t\|\vc\fj_1\|_{\dot B^{\frac n2-1}_{2,1}}^{\ell,\cL\cM}\,d\tau
+\int_0^t\|b\|^{h,1}_{\dot B^{\frac n2}_{2,1}}\,d\tau
+\frac{\cL\cM}\ep\int_0^t\|(j_0,\vc j_1)\|_{\dot B^{\frac n2-1}_{2,1}}^{h,\cL\cM}\,d\tau.}
$$
Above, we set $$\fj_0:=j_0-\,b-\frac\ep\cL\div\vu,\quad
\zeta_0:=j_0-\biggl({\rm Id}-\frac1{n\cL^2\cM}\Delta\biggr)^{-1}b
\!\quad\hbox{and}\quad\! 
\vc \fj_1:=\vc j_1+\frac1{\cL\cM}\nabla j_0-\frac{1}{\cL\cL_s\cM}\nabla b.$$
\end{Proposition}

%%%%%%%%%%%%%%%%%%%%%%%%%%%%%

\subsection{Uniform global well-posedness in the Poisson regime}

In this paragraph, we sketch  the proof of the following global
existence result.
\begin{Theorem}\label{th:diff3}
There exists a positive constant $\eta$ depending only on $\mu/\nu,$ $n$ 
and on the pressure law such that if $\eps\in(0,1)$ and  if the coefficients
$\cL$ and $\cL_s$ fulfill \eqref{eq:poissonsas}  then any 
data $(b_0^\eps,\vu_0^\eps,j_{0,0}^\eps,\vc j_{1,0}^\eps)$ satisfying 
\begin{equation}\label{eq:small-diff3}
\|(b_0^\eps,\vu_0^\eps,j_{0,0}^\eps,\vc j_{1,0}^\eps)\|_{\check X_\eps^\nu}\leq \eta\nu,
\end{equation}
generates a unique global solution  $(b^\eps,\vu^\eps,j_0^\eps,\vc j_1^\eps)$ in $\check Y_\eps^\nu$
to  System  \eqref{eq:NSrad}.

Furthermore, we have
\begin{equation}\label{eq:unif-diff3}
\|(b^\eps,\vu^\eps,j_0^\eps,\vc j_1^\eps)\|_{\check Y_\eps^\nu}\leq C
\|(b_0^\eps,\vu_0^\eps,j_{0,0}^\eps,\vc j_{1,0}^\eps)\|_{\check X_\eps^\nu}.
\end{equation}
\end{Theorem}
\begin{proof}
 Assuming with no loss of generality that $\nu=1,$ the proof relies on Proposition \ref{p:para-diff-poisson} 
 with, dropping the indices $\ep$ for better readability, $\vv=\vu$  
$$
F:=-T'_{\nabla b}\cdot\vu-k_1(b)\div\vu\quad\hbox{and}\quad
\vc G:=-T'_{\nabla\vu}\cdot\vu+k_2(b)\wt\cA\vu-k_3(b)\nabla b+\frac{\cL\cM}nk_4(b)\vc j_1.
$$ 
Let us just explain how to handle the last term, as it cannot be bounded exactly as 
in the proof of Theorems \ref{th:diff1} or \ref{th:diff2} due to the difference between the spaces $\check Y_\eps^\nu$
and $\wt Y_\eps^\nu.$
We use the fact that 
$$
\cL\cM\vc j_1=\cL\cM\vc\fj_1+\cL_s^{-1}\nabla b-\nabla j_0,
$$
and thus 
$$ 
{\cL\cM}k_4(b)\vc j_1= \cL\cM k_4(b)(\vc j_1^{h,\cL\cM}+\vc\fj_1^{\ell,\cL\cM})
+\cL_s^{-1} k_4(b)\nabla b^{\ell,\cL\cM}-k_4(b)\nabla j_0^{\ell,\cL\cM}.
$$
It is clear that
$$
\| k_4(b)\vc j_1^{h,\cL\cM}\|_{L^1(\dot B^{\frac n2-1}_{2,1})}\lesssim
\|b\|_{L^\infty(\dot B^{\frac n2}_{2,1})}\|\vc j_1^{h,\cL\cM}\|_{L^1(\dot B^{\frac n2-1}_{2,1})}
\lesssim(\cL\cM)^{-1}\|(b,\vu,j_0,\vc j_1)\|_{\check Y_\ep^1},
$$
that the second term in the r.h.s. may be bounded in the same way,
and that the third one can   be bounded as the pressure term $k_3(b)\nabla b.$
For the last term, one just has to observe that the definition of $\|\cdot\|_{\check Y_\ep^1}$
guarantees that 
\begin{equation}\label{eq:poisson3}
\|\nabla j_0\|^{\ell,\cL\cM}_{L^1(\dot B^{\frac n2}_{2,1})}\lesssim \|(b,\vu,j_0,\vc j_1)\|_{\check Y_\ep^1},
\end{equation}
and that 
$$
\|k_4(b)\|_{L^\infty(\dot B^{\frac n2}_{2,1})}\lesssim \|b\|_{L^\infty(\dot B^{\frac n2}_{2,1})}.
$$
The rest of the proof is standard, and thus left to the reader.
\end{proof}

 %%%%%%%%%%%%%%%%%%%%%%%
 
 \subsection{Study of the limit system}

We introduce the following norms
$$
\|(b,\vu,j_0)\|_{\check\cX^\nu}:=
\|b\|_{\dot B^{\frac n2-1}_{2,1}}^{\ell,\nu^{-1}}
+\nu\|b\|_{\dot B^{\frac n2}_{2,1}}^{h,\nu^{-1}}
+\|\vu\|_{\dot B^{\frac n2-1}_{2,1}}+\|j_0\|_{\dot B^{\frac n2-1}_{2,1}}^{\ell,\nu^{-1}}
+\nu^3\|j_0\|_{\dot B^{\frac n2+2}_{2,1}}^{h,\nu^{-1}},
$$
and
$$\displaylines{
\|(b,\vu,j_0)\|_{\check\cY^\nu}:=\sup_{t\geq0}\|(b,\vu,j_0)(t)\|_{\check\cX^\nu}
\hfill\cr\hfill+\nu\int_{\R_+}\bigl(\|(b,j_0)\|_{\dot B^{\frac n2+1}_{2,1}}^{\ell,\nu^{-1}}
+\|\vu\|_{\dot B^{\frac n2+1}_{2,1}}\bigr)\,dt
+\int_{\R_+}\bigl(\|b\|_{\dot B^{\frac n2}_{2,1}}^{h,\nu^{-1}}
+\nu^2\|j_0\|_{\dot B^{\frac n2+2}_{2,1}}^{h,\nu^{-1}}\bigr)\,dt.}
$$
 \begin{Theorem}\label{th:poisson}
 Let the data $(b_0,\vu_0,j_{0,0})$ satisfy for a small enough constant $c>0$
 \begin{equation}\label{eq:poisson1}
  \|b_0\|_{\dot B^{\frac n2-1}_{2,1}}^{\ell,\nu^{-1}}
+\nu\|b_0\|_{\dot B^{\frac n2}_{2,1}}^{h,\nu^{-1}}
+\|\vu_0\|_{\dot B^{\frac n2-1}_{2,1}}\leq c\nu,
\end{equation} 
and the compatibility condition
$$
j_{0,0}-\frac{\nu^2}{nm}\Delta j_{0,0}=b_0.
$$
Then System \eqref{eq:eq1} admits a unique global solution $(b,\vu,j_0)$ in the space
$\check\cY^\nu,$ satisfying in addition for a large enough constant $C$ independent of $\nu$
\begin{equation}\label{eq:poisson2}
\|(b,\vu,j_0)\|_{\check\cY^\nu}\leq C\bigl(  \|b_0\|_{\dot B^{\frac n2-1}_{2,1}}^{\ell,\nu^{-1}}
+\nu\|b_0\|_{\dot B^{\frac n2}_{2,1}}^{h,\nu^{-1}}
+\|\vu_0\|_{\dot B^{\frac n2-1}_{2,1}}\bigr).
\end{equation}
 \end{Theorem}
 \begin{proof}
 We just sketch the proof as it is very similar to the standard one for the barotropic 
 Navier-Stokes equations. As usual, it suffices to treat the case $\nu=1.$ 
 
  The first step is to analyse the linearized system
 \begin{equation}\label{eq:poissonl}
 \left\{\begin{array}{l}
 \d_tb+\div\vu=f,\\[1ex]
 \d_t\vu-\wt\cA\vu+\nabla b+\frac1n\nabla j_0=g,\\[1ex]
\bigl({\rm Id}-\frac{1}{nm}\Delta\bigr) j_0=b.
 \end{array}\right.
 \end{equation}
To this end, we set $d=(-\Delta)^{-1/2}\div\vu$ and observe that in the Fourier space, 
$(\wh b,\wh d)$ fulfills the following ODE if $f=g=0$
$$
\left\{\begin{array}{l}
\d_t\wh b+\rho\wh d=0,\\[1ex]
\d_t\wh d+\rho^2\wh d-\rho a_\rho\wh b=0
\end{array}\right.
$$
with $\rho:=|\xi|$ and $a_\rho:=1+\frac{m}{\rho^2+nm}\cdotp$
Of course, $j_0$ may be computed from $b$ by the relation
$$
\wh j_0=\frac{nm}{\rho^2+nm}\,\wh b.
$$
Introducing the following Lyapunov and diffusion functionals
$$
\cL_\rho^2=2a_\rho|\wh b|^2+2|\wh d|^2+|\rho\wh b|^2-2\Re(\rho\wh b\overline{\wh d})
\quad\hbox{and}\quad
\cH_\rho^2=\rho^2(a_\rho|\wh b|^2+|\wh d|^2),
$$
we see that 
$$
\frac 12\frac d{dt}\cL_\rho^2+\cH^2_\rho=0.
$$
Because we have $a_\rho^{\pm1}\leq c_0$ for some $c_0$ independent of $\rho,$
one can thus conclude exactly as in the standard barotropic case that for all $t\geq0$ 
and $\rho\geq0$
$$
|(\wh b,\rho\wh b,\wh d)(t)|+\min(1,\rho)\int_0^t|\rho\wh b|\,d\tau+\rho^2\int_0^t|\wh d|\,d\tau
\lesssim |(\wh b,\rho\wh b,\wh d)(0)|.
$$
Back to \eqref{eq:poissonl}, one may combine Fourier-Plancherel theorem 
and Duhamel formula to get the following estimate for all $s\in\R$
$$\displaylines{
\|(b,\nabla b,\vu)(t)\|_{\dot B^s_{2,1}}+\|j_0(t)\|_{\dot B^s_{2,1}\cap \dot B^{s+3}_{2,1}}
+\int_0^t\bigl(\|\vu\|_{\dot B^{s+2}_{2,1}}+\|(b,j_0)\|_{\dot B^{s+2}_{2,1}}^{\ell,1}\bigr)\,d\tau
\hfill\cr\hfill+\int_0^t\bigl(\|b\|_{\dot B^{s+1}_{2,1}}^{h,1}+\|j_0\|_{\dot B^{s+3}_{2,1}}^{h,1}\bigr)\,d\tau
\leq C\biggl(\|(b_0,\nabla b_0,\vu_0)\|_{\dot B^s_{2,1}}
+\int_0^t\|(f,\nabla f,g)\|_{\dot B^s_{2,1}}\,d\tau\biggr).}
$$
However, because of the convection term in the equation for $b,$ 
this does not allow to prove estimates for the nonlinear system \eqref{eq:eq1}.
Therefore, mimicking the standard approach  for the compressible Navier-Stokes equation
we `paralinearize' the system and  get the following result
\begin{Proposition}
The solutions to the following paralinearized system
 $$
 \left\{\begin{array}{l}
 \d_tb+T_{\vv}\cdot\nabla\vu+\div\vu=f,\\[1ex]
 \d_t\vu+T_{\vv}\cdot\nabla\vu-\wt\cA\vu+\nabla b+\frac1n\nabla j_0=\vc g,\\[1ex]
\bigl({\rm Id}-\frac{1}{nm}\Delta\bigr) j_0=b.
 \end{array}\right.
 $$
fulfill the following a priori estimate
$$\displaylines{
\|(b,\nabla b,\vu)(t)\|_{\dot B^s_{2,1}}+\|j_0(t)\|_{\dot B^s_{2,1}\cap \dot B^{s+3}_{2,1}}
+\int_0^t\bigl(\|\vu\|_{\dot B^{s+2}_{2,1}}+\|(b,j_0)\|_{\dot B^{s+2}_{2,1}}^{\ell,1}\bigr)\,d\tau
\hfill\cr\hfill+\int_0^t\bigl(\|b\|_{\dot B^{s+1}_{2,1}}^{h,1}+\|j_0\|_{\dot B^{s+3}_{2,1}}^{h,1}\bigr)\,d\tau
\leq C\biggl(\|(b_0,\nabla b_0,\vu_0)\|_{\dot B^s_{2,1}}
+\int_0^t\|(f,\nabla f,g)\|_{\dot B^s_{2,1}}\,d\tau\hfill\cr\hfill
+\int_0^t\|\nabla\vv\|_{L^\infty}\|(b,\nabla b,\vu)\|_{\dot B^s_{2,1}}\,d\tau\biggr).}
$$
\end{Proposition}
Now, in order to estimate the solutions of the nonlinear system \eqref{eq:eq1}, 
it suffices to apply the above proposition with $\vv=\vu$
$$
f=-T'_{\vu}\cdot\nabla b- k_1(b)\div\vu\quad\hbox{and}\quad
\vc g=-T'_{\vu}\cdot\nabla \vu+k_2(b)\wt\cA\vu-k_3(b)\nabla b-n^{-1} k_4(b)\nabla j_0.
$$
All the terms but the last one of $\vc g$ are already present in the barotropic Navier-Stokes equations,
and may be bounded quadratically in terms of $\|(b,\vu,j_0)\|_{\check\cY^1}.$
Now, we have 
$$
\|k_4(b)\nabla j_0\|_{\dot B^{\frac n2-1}_{2,1}}\lesssim \|b\|_{\dot B^{\frac n2}_{2,1}}
\|j_0\|_{\dot B^{\frac n2}_{2,1}},$$
hence
$$
\|k_4(b)\nabla j_0\|_{L^1_t(\dot B^{\frac n2-1}_{2,1})}\lesssim \|b\|_{L^2_t(\dot B^{\frac n2}_{2,1})}
\bigl(\|j_0\|_{L^2_t(\dot B^{\frac n2}_{2,1})}^{\ell,1}+\|j_0\|_{L^2_t(\dot B^{\frac n2+2}_{2,1})}^{h,1}\bigr),
$$
and one can thus conclude that whenever the solution $(b,\vu,j_0)$ exists we have
$$
\|(b,\vu,j_0)\|_{\check\cY^1(t)}\leq C\bigl( \|b_0\|_{\dot B^{\frac n2-1}_{2,1}\cap\dot B^{\frac n2}_{2,1}}
+\|\vu_0\|_{\dot B^{\frac n2-1}_{2,1}}+\|(b,\vu,j_0)\|_{\check\cY^1(t)}^2\bigr),
$$ which allows to get \eqref{eq:poisson2} if \eqref{eq:poisson1} is fulfilled 
with a small enough $c.$
\end{proof}
 
%%%%%%%%%%%%%%%%%%%%%%%%%%%%%%%%%%%%%%

\subsection{Weak convergence}

Here we justify weak convergence to \eqref{eq:eq1}  when  assumption \eqref{eq:poissonsas} is fulfilled
and, in addition
\begin{equation}\label{eq:poissonsas1}
\nu^2\cL^2\cL_s\to m\in(0,+\infty).
\end{equation}
 \begin{Theorem}\label{th:diff-lim3}
 Let the family of data $(b_0^\eps,\vu_0^\eps,j_{0,0}^\eps,\vc j_{1,0}^\eps)_{0<\eps<1}$ satisfy
 \eqref{eq:small-diff3}. 
Then  the global solution  $(b^\eps,\vu^\eps,j_0^\eps,\vc j_1^\eps)$ in $\check Y^\nu_\eps$ given 
by Theorem \ref{th:diff3} satisfies
\begin{equation}\label{eq:poisson4}
\vc j_1^\eps=\cO(\cL)\quad\hbox{in}\quad L^1(\R_+;\dot B^{\frac n2-1}_{2,1}+\dot B^{\frac n2}_{2,1}),
\end{equation}
and, up to  extraction,  $(b^\eps,\vu^\eps,j_0^\eps)$  converges weakly to some solution $(b,\vu,j_0)$ in $\check\cY^\nu$ of System \eqref{eq:eq1}
when $\ep$ goes to $0.$ 
\medbreak
If in addition   \begin{equation}\label{eq:poisson7}
(b^\eps_0,\vu^\eps_0,j_{0,0}^\eps)\rightharpoonup(b_0,\vu_0,j_{0,0})\quad\hbox{with}\quad 
-\nu^2\Delta j_{0,0}+ nm(j_{0,0}-b_0)=0,
\end{equation}
 then the whole family $(b^\ep,\vu^\ep,j_0^\eps)$ converges to the unique solution 
 $(b,\vu,j_0)$ corresponding to the initial data $(b_0,\vu_0,j_{0,0}),$ given by Theorem \ref{th:poisson}. 
 \end{Theorem}
\bProof 
Let us first prove \eqref{eq:poisson4}.  
From \eqref{eq:unif-diff3}, we already know that 
$(\vc\fj_1^\ep)^{\ell,\cL\cM}$ and $(\vc j_1^\ep)^{h,\cL\cM}$ are $\cO(\ep\cL)$ in $L^1(\R_+;\dot B^{\frac n2-1}_{2,1}).$
Now, we have
$$
\vc j_1^\ep=\vc\fj_1^\ep-\frac1{\cL\cL_s}\nabla j_0^\ep+\frac1{\cL\cL_s\cM}\nabla b^\ep.
$$
It is easy to see that the last term is $\cO(\cL^3)$ in $L^1(\R_+;\dot B^{\frac n2}_{2,1}
+\dot B^{\frac n2-1}_{2,1})$, and that, according to \eqref{eq:poisson3} and $\cL^2\cL_s\approx1,$
 the last but one term is $\cO(\cL)$ in $L^1(\R_+;\dot B^{\frac n2}_{2,1}),$
 which completes the proof of \eqref{eq:poisson4}.
 \smallbreak
 Next, let us turn our attention to the convergence of $j_0^\ep.$ First, \eqref{eq:unif-diff3}
 and the definition of $\|\cdot\|_{\check Y^\nu_\ep}$ ensure that 
 $(j_0^\ep)^{h,\cL\cM}$ is bounded in, say,  $L^2(\R_+;\dot B^{\frac n2-1}_{2,1}).$
 Next, using the bound for the middle frequencies of $j_0$ and for the low frequencies of $\zeta_0,$
 we discover that  $(j_0^\ep)^{\ell,\cL\cM}$ is bounded in $L^2(\R_+;\dot B^{\frac n2}_{2,1}).$
  Hence, up to an omitted extraction 
  \begin{equation}\label{eq:poisson4a}
  j_0^\ep\rightharpoonup j_0 \quad\hbox{weak } *\quad\hbox{in }\ 
  L^2(\R_+;\dot B^{\frac n2-1}_{2,1}+\dot B^{\frac n2}_{2,1}).
  \end{equation}
    Now,  taking the divergence of  the equation of $\vc j_1^\ep,$ 
 then using the equation of $j_0^\ep$ gives
 \begin{equation}\label{eq:poisson5}
 \Delta j_0^\ep=-\cL(1+\cL_s)\bigl(n\cL(b^\ep-j_0^\ep)-\ep n\d_t j_0^\ep\bigr) -\ep\d_t \div j_1^\ep.
 \end{equation}
 Given \eqref{eq:poisson4}, one can assert that the last term tends to $0$ in the sense of distributions. 
 We also know that, up to an omitted extraction,  $j_0^\ep\to j_0$ in the sense of distributions, 
 hence given that $\cL(1+\cL_s)\ep\to0,$ the term with $\d_t j_0^\ep$ also 
 tends to $0.$ 
 Finally, exactly as in the cases treated before, 
  $(b^\eps)$ is  bounded in 
 $L^\infty(\R_+;\dot B^{\frac n2-1}_{2,1}\cap\dot B^{\frac n2}_{2,1})$ hence weakly converges 
 to some $b\in L^\infty(\R_+;\dot B^{\frac n2-1}_{2,1}\cap\dot B^{\frac n2}_{2,1}).$
As \eqref{eq:poissonsas1} has been assumed, passing to the limit in \eqref{eq:poisson5} gives
 $$ \nu^2\Delta j_0=-nm\bigl(b-j_0).$$

 Passing to the limit in the equation of $b$ goes along the lines of the  non-equilibrium case
 we notice that  $(\d_tb^\eps)$ is bounded in $L^2(\R_+;\dot B^{\frac n2-1}_{2,1})$ 
 and we thus have, up to an omitted extraction
  \begin{equation}\label{eq:poisson6}
 \phi\, b^{\eps}\longrightarrow \phi\, b\quad\hbox{in}\quad L^\infty(\R_+;\dot B^{\frac n2-\alpha}_{2,1})
 \quad\hbox{for all }\ \alpha\in(0,1).
 \end{equation}
As    \eqref{eq:unif-diff3} also implies that  $(\vu^\eps)$ is  bounded in 
$L^\infty(\R_+;\dot B^{\frac n2-1}_{2,1})\cap L^1(\R_+;\dot B^{\frac n2+1}_{2,1}),$ 
we have $\vu^\eps\rightharpoonup \vu$ weakly * in that space, 
which is enough to justify the first equation of \eqref{eq:eq1}. 
 \medbreak
  In order to pass to the limit in the velocity equation, we use again the fact that
 $$
\d_t\biggl(\vu^\eps+\frac\eps n\, k_4(b^\eps)\,\vc j_1^\eps\biggr)
=-\vu^\eps\cdot\nabla\vu^\eps+k_2(b^\eps)\cA\vu^\eps-k_3(b^\eps)\nabla b^\eps
+\frac\eps nk'_4(b^\eps)\d_tb^\eps\,\vc j_1^\eps-\frac1n\,k_4(b^\eps)\,\nabla j_0^\eps.
$$
As in the other asymptotic regimes, the  first four terms of the r.h.s. are 
 bounded in $L^2(\R_+;\dot B^{\frac n2-2}_{2,1})$ (or in  $L^2(\R_+;\dot B^{\frac n2-2}_{2,\infty})$ if $n=2$).
 To handle the last term, we observe that according to \eqref{eq:unif-diff3} and \eqref{eq:poisson4a}, 
 $(\nabla j_0^\eps)$ is bounded in $L^2(\R_+;\dot B^{\frac n2-2}_{2,1}+\dot B^{\frac n2-1}_{2,1}).$
 Because $(b^\eps)$ is bounded in $L^\infty(\R_+;\dot B^{\frac n2-1}_{2,1}\cap\dot B^{\frac n2}_{2,1}),$
 this implies  that  $k_4(b^\eps)\nabla j_0^\eps$ is bounded 
in $L^2(\R_+;\dot B^{\frac n2-2}_{2,1})$ (or $L^2(\R_+;\dot B^{\frac n2-2}_{2,\infty})$ if $n=2$), 
and thus $\d_t\bigl(\vu^\eps+\frac\eps n\, k_4(b^\eps)\,\vc j_1^\eps\bigr)$ is bounded in the same space.
\medbreak
As in the already studied cases, we conclude that  
there exists some $\vu$ in $L^\infty(\R_+;\dot B^{\frac n2-1}_{2,1})$ so that 
for all $\phi$ in $\cS$ and $\alpha\in(0,1),$ we have
$$
\phi\Bigl(\vu^{\eps}+\frac{\eps}n\vc j_1^{\eps}\Bigr)\longrightarrow\phi\vu\quad\hbox{in}\quad L^\infty_{loc}(\R_+;\dot B^{\frac n2-1-\alpha}_{2,1}).
$$
Finally, in the case where \eqref{eq:poisson7} is fulfilled, the limit system \eqref{eq:eq1} 
supplemented with  initial data  $(b_0,\vu_0,j_{0,0})$  possesses a unique solution
 $(b,\vu,j_0)$ given by Theorem \ref{th:poisson},
 and the whole family $(b^\ep,\vu^\ep,j_0^\eps)$ thus converges to $(b,\vu,j_0).$
 \qed
 
 %%%%%%%%%%%%%%%%%%%%%%%%%%%%%%%%%%%%%%%%%%%%%%%%%%
 
 \appendix\section{Estimates for a toy linear differential equation}\label{s:A}
 
 The appendix is devoted to the proof of decay estimates for the solutions to systems of ODEs of the form
$$
\partial_tU+A_0U+\rho\left(A_1+B_1\right)U+\rho^2A_2U=0,\leqno(E)
$$
where $\rho$ is a nonnegative parameter, 
 and $A_0,$ $A_1,$ $B_1$ and $A_2$ are given $N\times N$ matrices.
We have in mind System \eqref{eq:diff0} 
  in which case, after suitable change of unknowns  (see \eqref{eq:U}), 
$A_0$ is a degenerate nonnegative diagonal matrix, $A_2$ has nonnegative eigenvalues
and $A_1$ is skewsymmetric up to some positive diagonal symmetrizer.
 
\subsection{A general approach}
 
The basic idea is to set $V:=(I+\rho P)U$ where $P$ is a suitable matrix, so as 
 to  eliminate the bad first order term $\rho B_1U.$
Now, whenever $(I+\rho P)$ is invertible,  the equation for $V$ reads
$$\displaylines{
\partial_tV+A_0V+\rho\bigl(A_1+B_1+[P,A_0]\bigr)V+\rho^2\bigl([A_0,P]P
+[P,A_1]+[P,B_1]+ A_2\bigr)V\hfill\cr\hfill
+\rho^3(I+\rho P)\bigl((A_1+B_1)P^2-A_0P^3- A_2P\bigr)(I+\rho P)^{-1} V=0.}
$$
Therefore, if one can find some matrix $P$ so that
\begin{equation}\label{eq:Pcom}
[A_0,P]=B_1,
\end{equation}
then we have 
\begin{equation}\label{eq:V}
\partial_tV+A_0V+\rho A_1 V+\rho^2\left( A_2+PB_1+[P,A_1]\right)V\\
=\rho^3(I+\rho P)\,A_3\, (I+\rho P)^{-1} V,
\end{equation}
where $A_3:=(PA_0-A_1)P^2+ A_2P.$
\medbreak
The gain is clear as the matrix $B_1$ now appears at order $2$ instead of order $1$. 
Hence the system for $V$ is more likely to be tractable  for small enough $\rho$ as we shall see below. 
 
%%%%%%%%%%%%%%%%%%%%%%%%%%%%%%%%%%%

\subsection{Application to the linearized system for barotropic radiative flows}

The system we are interested in reads 
\begin{equation}\label{eq:class}
\frac d{dt}\left(
\begin{array}{c}\wh a\\ \wh d\\\wh j_0\\\wh j_1\end{array}\right)
+\left(\begin{array}{cccc}0&\rho&0&0\\-\rho&\rho^2&0&-\varsigma\\
-\eta&0&\beta&\alpha\rho\\0&0&-\alpha\rho&\gamma\end{array}\right)
\left(\begin{array}{c}\wh a\\\wh d\\\wh j_0\\\wh j_1\end{array}\right)=
\left(\begin{array}{c}0\\0\\0\\0\end{array}\right),
\end{equation}
where all the coefficients of the matrix are positive.
\smallbreak
To bound the solutions of \eqref{eq:class} for small enough $\rho$ (under some stability condition that we will discover below), 
we propose two  different approaches, the first 
one being appropriate to handle the case  where $\beta$ and $\gamma$ are of the same order of magnitude, 
and the second one, to the case where $\beta/\gamma\ll1$ or $\gamma/\beta\ll1$
(of course only $\gamma\geq\beta$ is relevant as far as \eqref{eq:diff0} is concerned).

\subsubsection{First approach}

 Making the change of unknown 
\begin{equation}\label{eq:U}
U:=\left(\begin{array}{cccc}1&0&0&0\\0&1&0&\frac\varsigma\gamma\\
-\frac\eta\beta&0&1&0\\0&0&0&1\end{array}\right)
\left(\begin{array}{c}\wh a\\\wh d\\\wh j_0\\\wh j_1\end{array}\right),
\end{equation}
and setting $\tilde\alpha:=\alpha+\frac{\varsigma\eta}{\beta\gamma},$ 
we see that $U$ satisfies a  system  of type $(E)$ with 
$$
A_0:=\left(\begin{array}{cccc}0&0&0&0\\0&0&0&0\\
0&0&\beta&0\\0&0&0&\gamma\end{array}\right),
\qquad
A_1:=\left(\begin{array}{cccc}0&1&0&0\\-1-\frac{\alpha\varsigma\eta}{\beta\gamma}&0&0&0\\
0&0&0&\tilde\alpha\\0&0&-\alpha&0\end{array}\right),
$$
$$
B_1:=-
\left(\begin{array}{cccc}0&0&0&\frac\varsigma\gamma\\0&0&\frac{\alpha\varsigma}\gamma&0\\
0&\frac\eta\beta&0&0\\\frac{\alpha\eta}\beta&0&0&0\end{array}\right)
\quad\hbox{and}\quad
A_2:=\left(\begin{array}{cccc}0&0&0&0\\0&1&0&-\frac\varsigma\gamma\\
0&0&0&0\\0&0&0&0\end{array}\right)\cdotp
$$
Note that the above  matrices may be written in block form  as follows
\[
B_1=
\left(
\begin{array}{cc}
0&B_1^1\\
B_1^2&0
\end{array}
\right),\ \ \
A_0=
\left(
\begin{array}{cc}
0&0\\
0&\Delta
\end{array}
\right),\ \ \
A_1=
\left(
\begin{array}{cc}
A^{1}_1&0\\
0&A^{2}_1
\end{array}
\right),\ \ \
P=
\left(
\begin{array}{cc}
P^{11}&P^{12}\\
P^{21}&P^{22}
\end{array}
\right)\cdotp
\]
Computing the commutator
\begin{equation}\label{eq:A0P}
[A_0,P]=
\left(
\begin{array}{cc}
0&-P^{12}\Delta\\
\Delta P^{21}&[\Delta,P^{22}]
\end{array}
\right),
\end{equation}
we see \eqref{eq:Pcom} is satisfied if 
 $$P^{11}:=0,\quad P^{22}:=0,\quad P^{12}:=-B_1^1\Delta^{-1}, \quad P^{21}:=\Delta^{-1}B_1^2.$$
 In other words 
 \begin{equation}\label{eq:P}
 P=\left(\begin{array}{cccc}0&0&0&\frac\varsigma{\gamma^2}\\0&0&\frac{\alpha\varsigma}{\beta\gamma}&0\\
0&-\frac\eta{\beta^2}&0&0\\-\frac{\alpha\eta}{\beta\gamma}&0&0&0\end{array}\right),
\end{equation}
which, remembering \eqref{eq:U}, corresponds to  the following  change of unknowns
\begin{equation}\label{eq:Vc}
V=\left(\begin{array}{c}\wh\fb\\ \wh\fd\\\wh\fj_0\\\wh\fj_1\end{array}\right)
:=\left(\begin{array}{cccc}1&0&0&\frac\varsigma{\gamma^2}\rho\\
-\frac{\alpha\varsigma\eta}{\beta^2\gamma}\rho&1&\frac{\alpha\varsigma}{\beta\gamma}\rho&\frac\varsigma\gamma\\
-\frac\eta\beta&-\frac\eta{\beta^2}\rho&1&-\frac{\varsigma\eta}{\beta^2\gamma}\rho\\
-\frac{\alpha\eta}{\beta\gamma}\rho&0&0&1\end{array}\right)
\left(\begin{array}{c}\wh a\\ \wh d\\\wh j_0\\\wh j_1\end{array}\right)\cdotp
\end{equation}
Note that the determinant of the matrix $(I+\rho P)$  is
$$
\biggl(1+\frac{\alpha\varsigma\eta}{\beta^3\gamma}\rho^2\biggr)
\biggl(1+\frac{\alpha\varsigma\eta}{\beta\gamma^3}\rho^2\biggr),
$$
and is thus of order $1$ whenever $\rho$ satisfies the smallness condition
\begin{equation}\label{eq:smallrho1}
\rho^2\lesssim \frac{\beta\gamma}{\alpha\varsigma\eta}\,\min(\beta^2,\gamma^2).
\end{equation}

In order to go further in the estimates of $V,$  we  compute
 $$PB_1=\left(
\begin{array}{cc}
-B_1^1\Delta^{-1}B_1^2&0\\
0&\Delta^{-1}B_1^2B_1^1
\end{array}
\right)=\left(\begin{array}{cccc}-\frac{\alpha\varsigma\eta}{\beta\gamma^2}&0&0&0\\0&-\frac{\alpha\varsigma\eta}{\beta^2\gamma}&0&0\\
0&0&\frac{\alpha\varsigma\eta}{\beta^2\gamma}&0\\0&0&0&\frac{\alpha\varsigma\eta}{\beta\gamma^2}\end{array}\right)\quad\hbox{and}\quad$$
$$\begin{array}{lll} [P,A_1]&\!\!\!\!=\!\!\!\!&
\left(
\begin{array}{cc}
0&-B_1^1\Delta^{-1}A_1^2+A_1^1B_1^1\Delta^{-1}\\
\Delta^{-1}B_1^2A_1^1-A_1^2\Delta^{-1}B_1^2&0
\end{array}
\right)\\[3ex]&\!\!\!\!=\!\!\!\!&\left(\begin{array}{cccc}0&0&-\frac{\alpha\varsigma}\gamma(\frac1\beta\!+\!\frac1\gamma)&0\\0&0&0&\frac{\alpha\tilde\alpha\varsigma}{\beta\gamma}+\frac\varsigma{\gamma^2}
(1\!+\!\frac{\alpha\varsigma\eta}{\beta\gamma})\\
\frac{\alpha\tilde\alpha\eta}{\beta\gamma}+\frac\eta{\beta^2}
(1\!+\!\frac{\alpha\varsigma\eta}{\beta\gamma})&0&0&0\\0&
-\frac{\alpha\eta}{\beta}(\frac1\beta\!+\!\frac1\gamma)
&0&0\end{array}\right)\cdotp\end{array}
$$
Finally, $A_3:=(PA_0-A_1)P^2+ A_2P$ reads
\begin{equation}\label{eq:A_3}
A_3=\frac{\alpha\varsigma\eta}{\beta\gamma}\left(\begin{array}{cccc}
0&\frac1{\beta^2}&0&-\frac\varsigma{\gamma^3}\\
\frac1\gamma-\frac1{\gamma^2}\bigl(1\!+\!\frac{\alpha\varsigma\eta}{\beta\gamma}\bigr)&0&\frac1\eta-\frac{\alpha\varsigma}{\beta^2\gamma}&0\\
0&0&0&\frac{\tilde\alpha}{\gamma^2}\\
0&0&-\frac\alpha{\beta^2}&0\end{array}\right)\cdotp
\end{equation}
%Hence \begin{equation}\label{eq:A3}|A_3|\leq C\end{equation}
%with $C$ depending continuously on $\alpha,$ $\beta,$ $\gamma,$ $\varsigma$ and $\eta.$\medbreak
Therefore, resuming to \eqref{eq:V}, we conclude that
$$
\frac d{dt} V+A_0V+\rho A_1V+\rho^2\bigl(PB_1+ A_2\bigr)V=\rho^2[A_1,P]V+\cO(\rho^3).
$$
Of course, the remainder term $\cO(\rho^3)$  strongly depends on the coefficients of the system.
We shall see below that the structure of $[A_1,P]$ will enable us to treat 
$\rho^2[A_1,P]$ and the nondiagonal term of $A_2$ as  small error terms as well.
\medbreak
Let us  focus on the system satisfied by $(\wh\fb,\wh\fd)$ for a while. 
We have
\begin{multline}\label{eq:hydrolf}
\frac d{dt}\left(\begin{array}{c}\wh\fb\\\wh\fd\end{array}\right)
+\rho\left(\begin{array}{cc}0&1\\-1-\frac{\alpha\varsigma\eta}{\beta\gamma}&0\end{array}\right)\left(\begin{array}{c}\wh\fb\\\wh\fd\end{array}\right)
+\rho^2\left(\begin{array}{cc}-\frac{\alpha\varsigma\eta}{\beta\gamma^2}&0\\0&1-\frac{\alpha\varsigma\eta}{\beta^2\gamma}\end{array}\right)\left(\begin{array}{c}\wh\fb\\\wh\fd\end{array}\right)\\
=\rho^2\left(\begin{array}{cc}\frac{\alpha\varsigma}\gamma(\frac1\beta+\frac1\gamma)&0\\
0&\frac{\varsigma\nu}\gamma-\frac{\alpha\tilde\alpha\varsigma}{\beta\gamma}-\frac\varsigma{\gamma^2}
(1+\frac{\alpha\varsigma\eta}{\beta\gamma})\end{array}\right)\left(\begin{array}{c}\wh\fj_0\\\wh\fj_1\end{array}\right)+\cO(\rho^3).
\end{multline}
For small enough $\rho,$ optimal estimates may be proved by taking advantage 
of the results of Appendix \ref{s:B}. 
Indeed, denoting by $\wh F_\rho$ the r.h.s. of \eqref{eq:hydrolf}, we see from  \eqref{eq:ODE5} that  if we set 
$$
\cU_\rho^2:=\biggl(1+\frac{\alpha\varsigma\eta}{\beta\gamma}\biggr)|\wh\fb|^2+|\wh\fd|^2
-\rho\biggl(1+\frac{\alpha\varsigma\eta}{\beta\gamma}\Bigl(\frac1\gamma-\frac1\beta\Bigr)\biggr)\Re(\wh\fb\,\overline{\wh\fd}),
$$
then, under the following necessary and sufficient stability condition 
\begin{equation}\label{eq:stabcond1}
\tilde\nu:=1-\frac{\alpha\varsigma\eta}{\beta\gamma}\biggl(\frac1\beta+\frac1\gamma\biggr)>0,
\end{equation}
we have (see \eqref{eq:ODE3} and \eqref{eq:ODE4})
\begin{equation}
 \cU_\rho\approx |(\wh\fb,\wh\fd)|\quad\hbox{and}\quad\frac d{dt}\cU_\rho^2+\frac{\tilde\nu}3\rho^2\cU_\rho^2\lesssim \cU_\rho |\wh F_\rho|,
\end{equation}
whenever
\begin{equation}\label{eq:hydrolf1a}
\rho\leq\frac{\sqrt{1+\frac{\alpha\varsigma\eta}{\beta\gamma}}}{1+\frac{\alpha\varsigma\eta}{\beta\gamma}(\frac1\gamma-\frac1\beta)}\cdotp
\end{equation}
So finally,  we get 
for some appropriate constant $C=C(\alpha,\beta,\gamma,\varsigma,\eta)$ 
$$
|(\wh\fb,\wh\fd)(t)|+\tilde\nu\rho^2\int_0^t|(\wh\fb,\wh\fd)|\,d\tau \leq C\biggl(
|(\wh\fb,\wh\fd)(0)|+\rho^2\!\int_0^t|(\wh\fj_0,\wh\fj_1)|\,d\tau
+\rho^3\!\int_0^t|(\wh\fb,\wh\fd,\wh\fj_0,\wh\fj_1)|\,d\tau\biggr),
$$
which, if $\rho\ll\tilde\nu,$   may be simplified into
\begin{equation}\label{eq:hydrolf2}
|(\wh\fb,\wh\fd)(t)|+\tilde\nu\rho^2\int_0^t|(\wh\fb,\wh\fd)|\,d\tau \leq C\biggl(
|(\wh\fb,\wh\fd)(0)|+\rho^2\int_0^t|(\wh\fj_0,\wh\fj_1)|\,d\tau\biggr)\cdotp
\end{equation}
The modified radiative modes $\fj_0$ and $\fj_1$ fulfill
\begin{multline}\label{eq:radialf}
\frac d{dt}\left(\begin{array}{c}\wh\fj_0\\\wh\fj_1\end{array}\right)
+\rho\left(\begin{array}{cc}0&\tilde\alpha\\-\alpha&0\end{array}\right)\left(\begin{array}{c}\wh\fj_0\\\wh\fj_1\end{array}\right)
+\left(\begin{array}{cc}\beta+\frac{\alpha\varsigma\eta}{\beta^2\gamma}\rho^2&0\\0&\gamma+\frac{\alpha\varsigma\eta}{\beta\gamma^2}\rho^2\end{array}\right)\left(\begin{array}{c}\wh\fj_0\\\wh\fj_1\end{array}\right)\\
=\rho^2\left(\begin{array}{cc}-\frac{\alpha\tilde\alpha\eta}{\beta\gamma}-\frac\eta{\beta^2}
(1+\frac{\alpha\varsigma\eta}{\beta\gamma})&0\\
0&\frac{\alpha\eta}\beta(\frac1\beta+\frac1\gamma)\end{array}\right)\left(\begin{array}{c}\wh\fb\\\wh\fd\end{array}\right)+\cO(\rho^3).
\end{multline}
Therefore we easily get
$$\displaylines{
\frac12\frac d{dt}\biggl(|\wh\fj_0|^2+\frac{\tilde\alpha}\alpha|\wh\fj_1|^2\biggr)
+\biggl(\beta+\frac{\alpha\varsigma\eta}{\beta^2\gamma}\rho^2\biggr)|\wh\fj_0|^2
+\frac{\tilde\alpha}{\alpha}\biggl(\gamma+\frac{\alpha\varsigma\eta}{\beta\gamma^2}\rho^2\biggr)|\wh\fj_1|^2
\hfill\cr\hfill\leq C\bigl(\rho^2|(\wh\fb,\wh\fd)|+\rho^3|(\wh\fb,\wh\fd,\wh\fj_0,\wh\fj_1)|\bigr).}
$$
Then, integrating and assuming that $\rho\ll1$  yields
\begin{equation}\label{eq:J}
(|(\wh\fj_0,\wh\fj_1)(t)|+\min(\beta,\gamma)\int_0^t|(\wh\fj_0,\wh\fj_1)|\,d\tau\leq
C\biggl(|(\wh\fj_0,\wh\fj_1)(0)|+\rho^2\int_0^t|(\wh\fb,\wh\fd)|\,d\tau\biggr)\cdotp\end{equation}
Combining with \eqref{eq:hydrolf2}, we can conclude that 
there exists some positive constants $\rho_0$ and $C$ depending only 
on $(\alpha,\beta,\gamma,\varsigma,\eta)$ so that for all 
\begin{equation}\label{eq:condlf}
0\leq\rho\leq \min(1,\tilde\nu)\,\rho_0,\end{equation} we have
\begin{multline}\label{eq:lf}
|(\wh\fb,\wh\fd)(t)|+\tilde\nu|(\wh\fj_0,\wh\fj_1)(t)|+\tilde\nu\rho^2\int_0^t|(\wh\fb,\wh\fd)|\,d\tau
+\tilde\nu\min(\beta,\gamma)\int_0^t|(\wh\fj_0,\wh\fj_1)|\,d\tau
\\\leq C\bigl(|(\wh\fb,\wh\fd)(0)|+\tilde\nu|(\wh\fj_0,\wh\fj_1)(0)|\bigr).\end{multline}

%%%%%%%%%%%%%%%%%%%%%%%%%%%%%%%%%

\subsubsection{Second approach}

In the case where $\beta$ and $\gamma$ are not of the same order of magnitude, 
Inequality \eqref{eq:lf} is not fully satisfactory,
first because we would like  to have a control on $\beta\int_0^t|\wh\fj_0|\,d\tau$ and
 $\gamma\int_0^t|\wh\fj_1|\,d\tau$ rather than just on $\min(\beta,\gamma)\int_0^t|(\wh\fj_0,\wh\fj_1)|\,d\tau$
and, second, because the range for which \eqref{eq:lf} holds true tends to shrink to $0$ if 
$\beta\ll\gamma$ or $\gamma\ll\beta.$

In this paragraph, we propose another approach to handle \eqref{eq:class} in the case $\beta\not=\gamma,$ still based on rewriting the system 
in the form \eqref{eq:V}, but with a different definition of $A_1$ and $B_1$ ($A_0$ and $A_2$ being unchanged).
More precisely, we now set 
$$
A_1:=\left(\begin{array}{cccc}0&1&0&0\\-1-\frac{\alpha\varsigma\eta}{\beta\gamma}&0&0&0\\0&0&0&0\\0&0&0&0
\end{array}\right)\quad\hbox{and}\quad
B_1:=\left(\begin{array}{cccc}0&0&0&-\frac\varsigma\gamma\\0&0&-\frac{\alpha\varsigma}\gamma&0\\
0&-\frac\eta\beta&0&\tilde\alpha\\
-\frac{\alpha\eta}\beta&0&-\alpha&0\end{array}\right)\cdotp
$$
Then writing the matrices coming into play   in block form, we see according to \eqref{eq:A0P}, that 
a possible choice for $P$ is 
$$
P^{11}:=0,\quad\! P^{12}:=-B^1_1\Delta^{-1},\quad\! P^{21}:=\Delta^{-1}B_1^2,\quad\!
P^{22}:=\frac1{\beta\!-\!\gamma}\left(\begin{array}{cc}0&\tilde\alpha\\\alpha&0\end{array}\right)
\ \hbox{ with }\  \tilde\alpha:=\alpha+\frac{\ep\eta}{\beta\gamma}\cdotp
$$
With this new definition of $P,$ we have
$$\displaylines{
\quad PB_1=\left(\begin{array}{cccc}-\frac{\alpha\varsigma\eta}{\beta\gamma^2}&0&-\frac{\alpha\varsigma}{\gamma^2}&0\\
0&-\frac{\alpha\varsigma\eta}{\beta^2\gamma}&0&\frac{\alpha\tilde\alpha\varsigma}{\beta\gamma}\\
\frac{\alpha\tilde\alpha\eta}{\beta(\gamma-\beta)}0&0&\frac{\alpha\varsigma\eta}{\beta^2\gamma}+\frac{\alpha\tilde\alpha}{\gamma-\beta}&0\\
0&\frac{\alpha\eta}{\beta(\gamma-\beta)}&0&\frac{\alpha\varsigma\eta}{\beta\gamma^2}+\frac{\alpha\tilde\alpha}{\beta-\gamma}\end{array}\right)\hfill\cr\hfill
\hbox{and}\quad
[P,A_1]=\left(\begin{array}{cccc}0&0&-\frac{\alpha\varsigma}{\beta\gamma}&0\\0&0&0&\frac{\varsigma}{\gamma^2}\bigl(1\!+\!\frac{\alpha\varsigma\eta}{\beta\gamma}\bigr)\\
\frac\eta{\beta^2}
\bigl(1\!+\!\frac{\alpha\varsigma\eta}{\beta\gamma}\bigr)&0&0&0\\0&
-\frac{\alpha\eta}{\beta\gamma}
&0&0\end{array}\right)\cdotp\quad}
$$
Therefore setting 
\begin{equation}\label{eq:Vc2}
V=\left(\begin{array}{c}\wh\fb\\ \wh\fd\\\wh\fj_0\\\wh\fj_1\end{array}\right)
:=\left(\begin{array}{cccc}1&0&0&\frac\varsigma{\gamma^2}\rho\\
-\frac{\alpha\varsigma\eta}{\beta^2\gamma}\rho&1&\frac{\alpha\varsigma}{\beta\gamma}\rho&\frac\varsigma\gamma\\
-\frac\eta\beta&-\frac\eta{\beta^2}\rho&1&\bigl(\frac{\tilde\alpha}{\beta-\gamma}-\frac{\varsigma\eta}{\beta^2\gamma}\bigr)\rho\\
\frac{\alpha\eta}{\gamma(\gamma-\beta)}\rho&0&\frac\alpha{\beta-\gamma}\rho&1\end{array}\right)
\left(\begin{array}{c}\wh a\\ \wh d\\\wh j_0\\\wh j_1\end{array}\right),
\end{equation}
it is clear that working with $(\wh a,\wh d, \wh j_0, \wh j_1)$ or $(\wh\fb,\wh\fd,\wh\fj_0,\wh\fj_1)$ is equivalent whenever
$$
\rho\leq C|\gamma-\beta|,
$$
for some  positive constant $C$ depending continuously on the coefficients of the system.
\medbreak
Putting together  the previous computations, we see that $V$ fulfills
$$\displaylines{
\frac d{dt}V+\left(\begin{array}{cccc}-\frac{\alpha\varsigma\eta}{\beta\gamma^2}\rho^2&0&0&0\\
0&\bigl(1\!-\!\frac{\alpha\varsigma\eta}{\beta^2\gamma}\bigr)\rho^2&0&0\\
0&0&\beta\!+\!(\frac{\alpha\varsigma\eta}{\beta^2\gamma}+\frac{\alpha\tilde\alpha}{\gamma-\beta})\rho^2&0\\
0&0&0&\gamma\!+\!\bigl(\frac{\alpha\varsigma\eta}{\beta\gamma^2}\!+\!\frac{\alpha\tilde\alpha}{\beta-\gamma}\bigr)\rho^2\end{array}\right)\!V
\hfill\cr\hfill+\rho\left(\begin{array}{cccc}0&1&0&0\\-1-\frac{\alpha\varsigma\eta}{\beta\gamma}&0&0&0\\0&0&0&0\\0&0&0&0\end{array}\right)\!V
\hfill\cr\hfill=\rho^2\left(\begin{array}{cccc}0&0&\frac{\alpha\varsigma}\gamma\bigl(\frac1\beta\!+\!\frac1\gamma\bigr)&0\\
0&0&0&\frac{\varsigma}\gamma-\frac{\alpha\tilde\alpha\varsigma}{\beta\gamma}-\frac\varsigma{\gamma^2}\bigl(1\!+\!\frac{\alpha\varsigma\eta}{\beta\gamma}\bigr)\\
\frac{\alpha\tilde\alpha\eta}{\beta(\beta-\gamma)}-\frac\eta{\beta^2}\bigl(1\!+\!\frac{\alpha\varsigma\eta}{\beta\gamma}\bigr)&0&0&0\\
0&\frac{\alpha\eta}{\gamma(\beta-\gamma)}&0&0
\end{array}\right)\hfill\cr\hfill
+\rho^3(I+\rho P)A_3(I+\rho P)^{-1}V,}
$$
with $A_3:=(PA_0-A_1)P^2+ A_2 P$ satisfying
$|A_3|\leq C\bigl(1+|\gamma-\beta|^{-3}\bigr).$
\medbreak
Next, arguing exactly as to handle \eqref{eq:hydrolf}, 
we discover that under the stability condition \eqref{eq:stabcond1} and for $\rho$ satisfying \eqref{eq:hydrolf1a} (and of course
also $\rho\leq c|\beta-\gamma|^3$), we have 
\begin{multline}\label{eq:radialf0}
|(\wh\fb,\wh\fd)(t)|+\tilde\nu\rho^2\int_0^t|(\wh\fb,\wh\fd)|\,d\tau \leq C\biggl(
|(\wh\fb,\wh\fd)(0)|\\+\rho^2\!\int_0^t|(\wh\fj_0,\wh\fj_1)|\,d\tau
+{\rho^3}\biggl(1+\frac1{|\gamma-\beta|^3}\biggr)\int_0^t|(\wh\fb,\wh\fd,\wh\fj_0,\wh\fj_1)|\,d\tau\biggr)\cdotp
\end{multline}
Now, in contrast with the first method, we can bound $\wh\fj_0$ and $\wh\fj_1$ independently from one another
from the equation satisfied by $\wh\fj_0,$ we readily get
\begin{multline}\label{eq:radialf1}
|\wh\fj_0(t)|+\biggl(\beta+\frac{\alpha\varsigma\eta}{\beta^2\gamma}\rho^2\biggr)\int_0^t|\wh\fj_0|\,d\tau\leq|\wh\fj_0(0)|+\frac{C\rho^2}{|\beta-\gamma|}\int_0^t|\wh\fb|\,d\tau
\\+{C\rho^3}\biggl(1+\frac1{|\gamma-\beta|^3}\biggr)\int_0^t|(\wh\fb,\wh\fd,\wh\fj_0,\wh\fj_1)|\,d\tau,
\end{multline}
while $\wh\fj_1$ satisfies
\begin{multline}\label{eq:radialf2}
|\wh\fj_1(t)|+\biggl(\gamma+\Bigl(\frac{\alpha\varsigma\eta}{\beta\gamma^2}\!+\!\frac{\alpha\tilde\alpha}{\beta-\gamma}\Bigr)\rho^2\biggr)
\int_0^t|\wh\fj_1|\,d\tau\leq|\wh\fj_1(0)|+\frac{C\rho^2}{|\beta-\gamma|}\int_0^t|\wh\fd|\,d\tau\\+{C\rho^3}\biggl(1+\frac1{|\gamma-\beta|^3}\biggr)\int_0^t|(\wh\fb,\wh\fd,\wh\fj_0,\wh\fj_1)|\,d\tau.
\end{multline}
Putting inequalities  \eqref{eq:radialf0}, \eqref{eq:radialf1} and \eqref{eq:radialf2} together, it is now easy to conclude that
\begin{multline}\label{eq:lf2}
|(\wh\fb,\wh\fd)(t)|
+\tilde\nu|\gamma-\beta||(\wh\fj_0,\wh\fj_1)(t)|+\tilde\nu\rho^2\int_0^t|(\wh\fb,\wh\fd)|\,d\tau\\
+\tilde\nu|\gamma-\beta|\biggl(\beta\int_0^t|\wh\fj_0|\,d\tau +\gamma\int_0^t|\wh\fj_1|\,d\tau\biggr)
\leq C\bigl(|(\wh\fb,\wh\fd)(0)|+\tilde\nu|\gamma-\beta||(\wh\fj_0,\wh\fj_1)(0)|\bigr),
\end{multline}
if, for some small enough constant $c$ depending continuously on $\alpha,$ $\beta,$ $\gamma$ and~$\varsigma,$ we have
\begin{equation}
\rho\leq c\min\biggl(1,\,\tilde\nu,\,|\gamma-\beta|^2,\,\tilde\nu|\gamma-\beta|^3\biggr)\cdotp
\end{equation}

 %%%%%%%%%%%%%%%%%%%%%%%%%%%%%%%%%%%%%%%%%%%%%%%%%%

\section{Optimal decay estimates for a toy system}\label{s:B}

For the reader convenience, we here recall some results that have
been obtained in our recent work \cite{DD3} for the following linear system of ordinary differential equations
\begin{equation}\label{eq:ODE}
\left\{\begin{array}{l}
\d_tX+a\rho Y-b\rho^2 X=A,\\
\d_tY-c\rho X+d\rho^2 Y=B.
\end{array}
\right.
\end{equation}
Above,  $\rho$ stands for a given nonnegative small parameter and $a,$ $b,$ $c$ and $d$ are four real numbers satisfying the stability condition
\begin{equation}\label{eq:stab-ODE}
a>0,\quad c>0\quad \hbox{ and }\ d-b>0.
\end{equation}
%Even though \eqref{eq:ODE} may be solved explicitly, thus giving the desired (and optimal) decay results, 
%we here aim at recovering such results by means of an energy type method. 
%We start with the following identities
%$$\begin{array}{l}\Frac12\frac d{dt}\bigl(c|X|^2+a|Y|^2\bigr)-bc\rho^2|X|^2+ad\rho^2|Y|^2=\Re(cA\bar X+aB\bar Y),\\[1ex]
%\Frac d{dt}\Re(X\bar Y)+a\rho|Y|^2-c\rho|X|^2+(d-b)\rho^2\Re(X\bar Y)=\Re(B\bar X+A\bar  Y),\end{array}$$
%from which we easily get for any real number $\eta$ 
%\begin{multline}\label{eq:ODE1}\frac12\frac d{dt}\bigl(c|X|^2+a|Y|^2-2\rho\eta\Re(X\bar Y)\bigr)
%+(\eta-b)c\rho^2|X|^2\\+(d-\eta)a\rho^2|Y|^2+\eta(b-d)\rho^3\Re(X\bar Y)=\Re\bigl(cA\bar X+aB\bar Y-2\rho\eta(B\bar X+A\bar Y)\bigr).\end{multline}
%Choosing $\eta$ so that $d-\eta=\eta-b,$ that is to say $\eta:=\frac{b+d}2,$ we discover that
Routine computations show that the following  \emph{Lyapunov functional} 
$\cL_\rho^2:=c|X|^2+a|Y|^2-\rho(d+b)\Re(X\bar Y)$ satisfies the relation
\begin{multline}\label{eq:ODE2}
\frac12\frac d{dt}\cL_\rho^2+\biggl(\frac{d-b}2\biggr)\rho^2\bigl(c|X|^2+a|Y|^2\bigr)
+\biggl(\frac{b^2-d^2}2\biggr)\rho^3\Re(X\bar Y)\\=\Re\bigl(cA\bar X+aB\bar Y-\rho(b\!+\!d)(B\bar X\!+\!A\bar Y)\bigr).
\end{multline}
Now, observe 
% from the observation that$$|\Re(X\bar Y)|\leq\frac{1}{2\sqrt{ac}}\bigl(c|X|^2+a|Y|^2\bigr),$$
that  whenever $\rho\leq\frac{\sqrt{ac}}{|b+d|},$ we have
$$\biggl|\biggl(\frac{b^2-d^2}2\biggr)\rho^3\Re(X\bar Y)\biggr|\leq \biggl(\frac{d-b}4\biggr)\rho^2\bigl(c|X|^2+a|Y|^2\bigr),
$$
and
\begin{equation}\label{eq:ODE3}
\frac12\bigl(c|X|^2+a|Y|^2\bigr)\leq\cL_\rho^2\leq\frac32\bigl(c|X|^2+a|Y|^2\bigr),
\end{equation}
% If  $A\equiv B\equiv 0$  then resuming to \eqref{eq:ODE2} leads to 
%\begin{equation}\label{eq:ODE4}\frac d{dt}\cL_\rho^2+\biggl(\frac{d-b}3\biggr)\cL_\rho^2\leq0,\end{equation}
%and thus, for any $t\geq0,$ 
which leads if $A\equiv B\equiv 0$ to
\begin{equation}\label{eq:ODE4}\frac d{dt}\cL_\rho^2+\biggl(\frac{d-b}3\biggr)\cL_\rho^2\leq0,\end{equation}
 and thus
\begin{equation}\label{eq:ODE5}
\cL_\rho(t)\leq e^{-(\frac{d-b}6)\rho^2t} \cL_\rho(0).
\end{equation}
Combining with  \eqref{eq:ODE3} and Duhamel's formula, we deduce that 
%\begin{equation}\label{eq:ODE6}
%\sqrt{c|X(t)|^2+a|Y(t)|^2}\leq \sqrt3 \sqrt{c|X(0)|^2+a|Y(0)|^2}\,e^{-\bigl(\frac{d-b}6\bigr)\rho^2 t}\quad\hbox{for }\ 
%\ \rho\leq\frac{\sqrt{ac}}{|b+d|}\cdotp
%\end{equation}
for general source terms $A$ and $B$ we have 
\begin{multline}\label{eq:ODE7}
\sqrt{c|X(t)|^2+a|Y(t)|^2}\leq \sqrt3\,e^{-(\frac{d-b}6)\rho^2 t}\biggl(\sqrt{c|X(0)|^2+a|Y(0)|^2}
\\+\int_0^t e^{(\frac{d-b}6)\tau} \sqrt{c|A|^2+a|B|^2}\,d\tau\biggr)\cdotp
\end{multline}

\end{document}